\documentclass[11pt]{article}
\usepackage[utf8]{inputenc}
\usepackage[T1]{fontenc}
\usepackage{comment}
\usepackage{authblk}
\usepackage{relsize}
\usepackage{changepage}
\usepackage{mathtools}
\usepackage[english]{babel}
\usepackage{amsmath}
\usepackage{amsthm}
\usepackage{amsfonts}
\usepackage{amssymb}
\usepackage{graphicx}
\usepackage[usenames,dvipsnames]{color}
\theoremstyle{plain}
\newtheorem{theorem}{Theorem}[section]
\newtheorem{corollary}[theorem]{Corollary}

\newtheorem{lemma}[theorem]{Lemma}
\newtheorem{sublemma}[theorem]{Sublemma}
\newtheorem{observation}[theorem]{Observation}

\makeatletter
\newcommand{\vast}{\bBigg@{4}}
\newcommand{\Vast}{\bBigg@{5}}
\makeatother
\definecolor{bulgarianrose}{rgb}{0.28, 0.02, 0.03}
\definecolor{gray}{rgb}{0.5, 0.5, 0.5}

\theoremstyle{definition}

\newtheorem{remark}[theorem]{Remark}

\theoremstyle{remark}
\usepackage[left=2cm,right=2cm,top=2cm,bottom=2cm]{geometry}
\usepackage{amssymb}
\usepackage{hyperref}
\makeatletter
\def\namedlabel#1#2{\begingroup
    #2%
    \def\@currentlabel{#2}%
    \phantomsection\label{#1}\endgroup
}
\usepackage[normalem]{ulem}
\usepackage{dsfont}
\usepackage{accents}
\usepackage{verbatim}
\usepackage{ulem}
\usepackage{pgf,tikz,pgfplots}
\usepackage[all]{xy} 

\usepackage{stackengine}
\stackMath
\newcommand\tsup[2][2]{%
 \def\useanchorwidth{T}%
  \ifnum#1>1%
    \stackon[-.5pt]{\tsup[\numexpr#1-1\relax]{#2}}{\scriptscriptstyle\sim}%
  \else%
    \stackon[.5pt]{#2}{\scriptscriptstyle\sim}%
  \fi%
}

\newcommand{\vol}{\text{vol}}
\newcommand{\degg}{\text{deg}}

\usepackage{subfig}

\usepackage{enumerate}

\title{\scshape
  On the modularity of $3-$regular random graphs and random graphs with given degree sequences}

\author[1]{Lyuben Lichev}
\author[2,3]{Dieter Mitsche\footnote{Dieter Mitsche has been supported by grant GrHyDy ANR-20-CE40-0002 and by IDEXLYON of Universit\'{e} de Lyon (Programme Investissements d'Avenir ANR16-IDEX-0005).}}
\affil[1]{Ecole Normale Supérieure de Lyon, Lyon, France}
\affil[2]{Institut Camille Jordan, Univ. Lyon 1, Lyon, France}
\affil[3]{Univ. Jean Monnet, Saint-Etienne, France}

\begin{document}

\maketitle
 
\begin{abstract}
The modularity of a graph is a parameter that measures its community structure; the higher its value (between $0$ and $1$), the more clustered the graph is.

In this paper we show that the modularity of a random $3-$regular graph is at least $0.667026$ asymptotically almost surely (a.a.s.), thereby proving a conjecture of McDiarmid and Skerman. We also improve the a.a.s.\ upper bound given therein to $0.789998$.

For a uniformly chosen graph $G_n$ over a given bounded degree sequence with average degree $d(G_n)$ and with $|CC(G_n)|$ many connected components, we distinguish two regimes with respect to the existence of a giant component.  In the subcritical regime, we compute the second term of the modularity. In the supercritical regime, we prove that there is $\varepsilon > 0$, for which the modularity is a.a.s.\ at least 
\begin{equation*}
\dfrac{2\left(1 - \mu\right)}{d(G_n)}+\varepsilon,
\end{equation*}
where $\mu$ is the asymptotically almost sure limit of $\dfrac{|CC(G_n)|}{n}$. 
\end{abstract}

\section{Introduction}
Recent years have seen a fast increase of network data available and the need for detecting clusters - disjoint groups of nodes with many connections between the elements within a single group and rather few connections between elements of different groups - has become more and more important. Identifying clusters helps to exploit a network more effectively: for example, having detected clusters in social networks allows for targeted advertisements, or having detected clusters in collaboration networks allows for identifying similar papers. Whereas traditional clustering approaches either fix the number of clusters and/or the sizes of the clusters, the concept of \emph{modularity} allows for more flexibility here: whilst rewarding a partition for containing edges within its parts, it penalizes parts incident to too many edges. Introduced by Newman and Girvan in~\cite{Girvan}, it was first studied in physics (see~\cite{Physics1Mod, Physics2Mod}) due to its connections to the Potts model in statistical physics presented in~\cite{Physics3Mod}. It was then analyzed in different applications including protein discovery and identifying connections between websites: see~\cite{SurveyMod1} and~\cite{SurveyMod2} for surveys on the use of modularity for community detection in networks. After this successful application in practice, modularity was then also studied from a mathematical point of view. We first give the definition and our results and then refer to related work in the mathematics literature.\par

\noindent For a subset $A$ of vertices of $G$, we denote by $e(A)$ the number of edges with two endvertices in $A$ and by $\vol(A)$ the sum of the degrees of the vertices in $A$. The \textit{modularity of a partition $\mathcal A = \{A_1, A_2, \dots, A_k\}$} of the vertices of a graph $G$ with $m$ edges is defined as 
\begin{equation}\label{mod:definition}
    q(\mathcal A) = \dfrac{1}{m}\sum_{i=1}^k e(A_i) - \dfrac{1}{4m^2}\sum_{i=1}^k \vol(A_i)^2.
\end{equation}
The first term corresponds to the proportion of edges of $G$ that have both endvertices in the same part of $\mathcal A$; this term ensures that the edge density within the communities is high. The second term stands for the expected proportion of edges within the given parts in a random graph with the same degree sequence as $G$ (or also in a vertex-weighted random graph e.g. given by the Chung-Lu model). Its form suggests that it may be interpreted as a sort of ``degree tax''.

For a graph $G=(V, E)$, the \textit{modularity} $q^*(G)$ of $G$ is defined as 
\begin{equation*}
    \max_{\mathcal A \in \mathcal{P}(V)} q(\mathcal{A}),
\end{equation*}
where $\mathcal{P}(V)$ stands for the set of partitions of the vertex set $V$ of $G$. It is well known and easy to see that $0 \le q^*(G) < 1$ for every graph $G$ (for a graph $G$ without edges, by convention, $q^*(G)=0$). 
For every $d \ge 1$, we denote by $\mathcal G_d(n)$ the set of all $d-$regular graphs.
Denote also by $G_d(n)$ (or simply $G_d$) the random $d-$regular graph with $n$ vertices following the uniform distribution over the set $\mathcal G_d(n)$.\par 
For a sequence of probability spaces $(\Omega_n, \mathcal F_n, \mathbb P_n)_{n\geq 1}$ and a sequence of events $(A_n)_{n\geq 1}$, where $A_n\in \mathcal F_n$ for every $n\geq 1$, we say that $(A_n)_{n\geq 1}$ happens \textit{asymptotically almost surely} or \textit{a.a.s.}, if $\underset{n\to +\infty}{\lim}\mathbb P_n(A_n) = 1$. The sequence of events $(A_n)_{n\geq 1}$ itself is said to be \textit{asymptotically almost sure} or again \textit{a.a.s.} Our first result concerns $G_3(n)$: McDiarmid and Skerman conjectured in~\cite{ColinFiona} that there exists $\delta > 0$ such that a.a.s.\ $q^*(G_3(n)) \ge \frac23 + \delta$. Our first theorem confirms this conjecture:
\begin{theorem}\label{LB3reg}
Let $G_3 \in \mathcal G_3(n)$. Then a.a.s.\ $q^*(G_3) \ge 0.667026$.
\end{theorem}
As a complementary result, we also improve on the upper bound: to our knowledge the best results before this paper were  $q^*(G_3) \le 0.804$ (see~\cite{ColinFiona} and~\cite{OstPralat}). We prove the following result:
\begin{theorem}\label{UB3reg}
Let $G_3 \in \mathcal G_3(n)$. Then a.a.s.\ $q^*(G_3) \le 0.789998$.
\end{theorem}

In fact, in the spirit of Theorem~\ref{LB3reg}, we also obtain an improved lower bound for more general degree sequences $(D_n)_{n \ge 1}$. For a graph $G$, we denote by $CC(G)$ the set of connected components of $G$. Denote by $\Delta(n)$ the maximum degree in $D_n$ and, for every $i\ge 0$, denote by $d_i(n)$ the number of vertices in $D_n$ of degree $i$. A sequence of degree sequences is \textit{bounded} if there is $\Delta\in \mathbb N$ such that for every $n\ge 1, \Delta(n)\le \Delta$. In this paper we assume that $d_0(n) = 0$ for all $n$. We say that the sequence of degree sequences $(D_n)_{n \ge 1}$ is \textit{regular} if for every positive integer $i \ge 1$ there is $p_i \ge 0$, such that the proportion of vertices $\dfrac{d_i(n)}{n}$ of degree $i$ in $D_n$ tends to a limit $p_i$ with $n$.\par

\begin{theorem}\label{LBgeneral}
Fix a sequence of bounded regular degree sequences $(D_n)_{n\ge 1}$ with limit vector $\boldsymbol{p} = (p_i)_{i\ge 1}$ and maximal degree $\Delta$. Define 
\begin{equation}
Q = Q(\boldsymbol{p}) := \sum_{i\ge 1} i(i-2)p_i
\end{equation}
and 
\begin{equation*}
M = M(\boldsymbol{p}) := \sum_{i\ge 1} ip_i.
\end{equation*}
Then
\begin{enumerate}
    \item\label{thm 3.1} If $Q < 0$, then a.a.s.\
    \begin{equation*}
    q^*(G(n)) = 1 - \dfrac{c}{Mn} + o\left(\dfrac{1}{n}\right),
    \end{equation*}
    where $c = c(\boldsymbol{p}, \Delta) > 0$ is given by the sum 
    \begin{equation*}
4 \sum_{t_2, \dots, t_{\Delta}\in \mathbb N} \dfrac{t_2+\dots+(\Delta-1)t_{\Delta}+1}{t_2+\dots+(\Delta-1)t_{\Delta}+2} \binom{t_2\dots+(\Delta - 1)t_{\Delta} + 2}{\sum (i-2)t_i + 2, t_2, \dots, t_{\Delta}} \left(\dfrac{p_1}{M}\right)^2\prod_{2\le i\le \Delta} \left(\dfrac{ip_i}{M}\left(\dfrac{p_1}{M}\right)^{(i-2)}\right)^{t_i},
    \end{equation*}
    where $0^0 = 1$ by convention.
    \item\label{thm 3.3} If $Q > 0$, there exists a constant $\varepsilon > 0$ so that a.a.s.\
    \begin{equation*}
        q^*(G_n) \ge \dfrac{2(1-\mu)}{M} + \varepsilon,
    \end{equation*}
    where $\mu = \mu(\boldsymbol{p})$ is the a.a.s.\ limit of $\dfrac{|CC(G(n)|}{n}$. 
\end{enumerate}
\end{theorem}

\begin{remark}
Point~\ref{thm 3.3} of Theorem~\ref{LBgeneral} proves on its own that the modularity of the random $3-$regular graph $G_3(n)$ is a.a.s.\ at least $2/3 + \varepsilon$ for some $\varepsilon > 0$, but the value of $\varepsilon$ is not given explicitly this time. Indeed, the random $3-$regular graph is a.a.s.\ connected (so $\mu = 0$), and its average degree is always $3$ (thus $M = 3$).
\end{remark}

\begin{remark}
It is a natural question to ask if a closed formula for $c(\boldsymbol{p}, \Delta)$ may be given. Sadly, even for $\Delta = 3$, one may only reduce the expression for $c(\boldsymbol{p}, \Delta)$ to a single sum of terms given by (a little more than) hypergeometric terms. The answer to this question seems therefore to be negative for $\Delta \ge 3$.
However, for $\Delta = 2$, the constant is not hard to calculate, it is equal to
\begin{equation*}
4\sum_{t_2\ge 1} \dfrac{t_2+1}{t_2+2}\binom{t_2+2}{2} \left(\dfrac{p_1}{M}\right)^2 \left(\dfrac{2p_2}{M}\right)^{t_2} = \dfrac{2(p_1+4p_2)}{p_1} - 2\left(\frac{p_1}{p_1+2p_2}\right)^2.
\end{equation*}
\end{remark}

\par\noindent
\textbf{Related work.}
After the introduction of the concept in the already mentioned paper by Newman and Girvan~\cite{Girvan}, due to its success in applications, modularity was analyzed for different graph classes: cycles were analyzed in~\cite{Cycles2}, lattices in~\cite{ColinFionaLattice} and~\cite{Lattices} respectively.
The study of modularity in trees was initiated by Bagrow in~\cite{Bagrow}, who showed that $k$-ary trees as well as Galton-Watson-trees have modularity tending to $1$. Later De Montgolfier, Soto and Viennot proved in~\cite{Montgolfier} that trees with maximum degree $\Delta=o(n^{1/5})$ have modularity tending to $1$, which was then extended by McDiarmid and Skerman~\cite{ColinFiona} to trees with maximum degree $o(n)$ (and more generally to graphs that are 'treelike' in the sense of having low treewidth). More generally, Ostroumova, Prokhorenkova, Pra\l{}at, and Raigorodskii showed in~\cite{OstPralat} that all connected graphs $G$ on $n$ vertices with maximal degree $\Delta(n)=o(n)$ and average degree $d(n)$ satisfy $q^*(G) \ge \frac{2}{d(n)} - 3\sqrt{\frac{\Delta(n)}{nd(n)}}-\frac{\Delta(n)}{nd(n)}$.
Modularity was also studied for random graphs: the Erd\H{o}s-R\'{e}nyi model $G(n,p)$ was studied by McDiarmid and Skerman in~\cite{ColinFiona2}: they showed that for $p \le 1/n$, a.a.s.\, $q^*(G(n,p))=1+o(1)$, whereas for $p \ge 1/n$ and $p < 1$, a.a.s.\, $q^*(G(n,p))=\Theta(\frac{1}{\sqrt{np}})$. Their results transfer also to the $G(n,M)$ model.
For random regular graphs, besides the already mentioned bounds given in~\cite{ColinFiona} and~\cite{OstPralat} of $\frac23 \le q^*(G) \le 0.804$ for $G \in \mathcal G_3(n)$, in~\cite{ColinFiona} McDiarmid and Skerman gave also lower and upper bounds for $d-$regular graphs for other values of $d$, in particular they showed that $0.7631/\sqrt{d} \le q^*(G_d(n))$ for $d$ sufficiently large, and $q^*(G_d(n)) \le 2/\sqrt{d}$ for all $d \ge 3$. For random $2-$regular graphs, they also proved $q^*(G_2(n))=1-\frac{2}{\sqrt{n}}+o(\frac{\log^2 n}{n})$. Regarding other models of random graphs, in~\cite{OstPralat} it was proved that for a graph $G$ chosen according to the preferential attachment model, when adding $m \ge 2$ edges at a time, $\max\{\frac{1}{m} - o(1),\Omega(\frac{1}{\sqrt{m}})\}  \le q^*(G) < 0.94$, where the constant hidden in the asymptotic notation is such that for $m \ge 1000$ the second lower bound is better. In the same paper, they also showed that for the spatial preferential attachment model with certain conditions on the parameters the modularity is $1+o(1)$.
From a computational point of view, finding the modularity of a graph was proved to be NP-hard and even approximation of the modularity within a constant multiplicative factor remains NP-hard, see~\cite{Brandes} and~\cite{ApproxHardModularity}. The concept of modularity was recently extended to hypergraphs, see~\cite{PawelHypergraphs}.

\par \noindent
\textbf{Overview of the proofs.} In the proof of Theorem~\ref{LB3reg} we choose an arbitrary vertex $v$ and start exploring its connected component $C$ one half-edge at a time chosen arbitrarily among the ones incident to $C$ until $\varepsilon n$ vertices have been processed. Call the explored set of vertices $C_0(\varepsilon)$. Then, start exploring one by one open half-edges sticking out of a processed vertex to grow little by little the set of processed vertices themselves, but without directly sending them in $C_0(\varepsilon)$ anymore. Throughout this process we add short paths to $C_0(\varepsilon)$ containing exactly two vertices in $C_0(\varepsilon)$ - the first and the last vertex of each path (see Figure~\ref{3-chain}). In this way one increases the modularity of $C_0(\varepsilon)$ at each step. By analysis of the first few steps of this procedure via the differential equation method and consequent optimization over $\varepsilon$, we deduce Theorem~\ref{LB3reg}.\par

\begin{figure}
\centering
\begin{tikzpicture}[line cap=round,line join=round,x=1cm,y=1cm]
\clip(-4,-6) rectangle (10,0.5);
\draw [rotate around={-1.325372843922181:(2.81,-3.15)},line width=2pt] (2.81,-3.15) ellipse (4.429072300114929cm and 2.115769703830106cm);
\draw [line width=0.5pt] (-1.08,-3.06)-- (0.24,-3.94);
\draw [line width=0.5pt] (0.24,-3.94)-- (3.4,-4.28);
\draw [line width=0.5pt] (3.4,-4.28)-- (6.7,-3.24);
\draw [line width=0.5pt] (6.7,-3.24)-- (3.12,-2.72);
\draw [line width=0.5pt] (0.24,-3.94)-- (3.12,-2.72);
\draw [line width=0.5pt] (-1.08,-3.06)-- (0.78,-2.1);
\draw [line width=0.5pt] (0.78,-2.1)-- (1.46,0.06);
\draw [line width=0.5pt] (1.46,0.06)-- (3.04,0.14);
\draw [line width=0.5pt] (3.04,0.14)-- (3.12,-2.72);
\begin{scriptsize}
\draw [fill=black] (-1.08,-3.06) circle (2.5pt);
\draw [fill=black] (6.7,-3.24) circle (2.5pt);
\draw [fill=black] (0.24,-3.94) circle (2.5pt);
\draw [fill=black] (3.4,-4.28) circle (2.5pt);
\draw [fill=black] (3.12,-2.72) circle (2.5pt);
\draw [fill=black] (0.78,-2.1) circle (2.5pt);
\draw [fill=black] (1.46,0.06) circle (2.5pt);
\draw [fill=black] (3.04,0.14) circle (2.5pt);
\draw[color=black] (3.45,0.14) node {\large $v_3$};
\draw[color=black] (1,0.14) node {\large $v_2$};
\draw[color=black] (1.2,-2.2) node {\large $v_1$};
\draw[color=black] (3.45,-2.55) node {\large $v_4$};
\draw[color=black] (5.3,-2) node {\Large $C_0(\varepsilon)$};
\end{scriptsize}
\end{tikzpicture}
\caption{The path $v_1v_2v_3v_4$ with vertices $v_1$ and $v_4$ in $C_0(\varepsilon)$ and vertices $v_2$ and $v_3$ outside $C_0(\varepsilon)$ just before being added to $C_0(\varepsilon)$ entirely.}
\label{3-chain}
\end{figure}
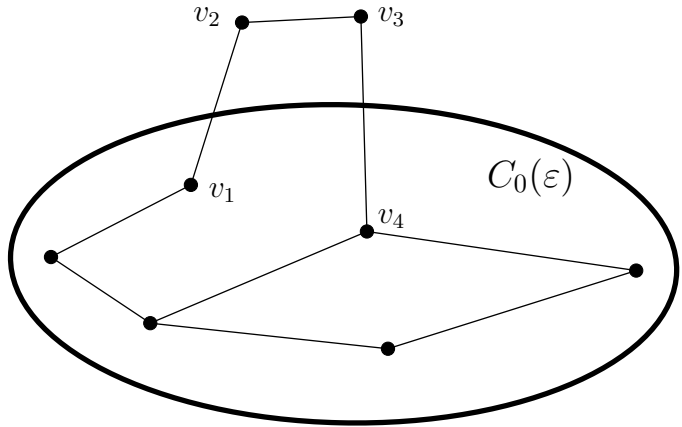

In the proof of Theorem~\ref{UB3reg}, we first observe that for the modularity of $G_3(n)$ to be at least the given upper bound, there must be a part $A_i$ in the optimal partition with
\begin{equation*}
\dfrac{2e(A_i)}{3|A_i|} - \dfrac{|A_i|}{n} \ge 0.789998.
\end{equation*}
We first rule out the possibility that the order of $A_i$ is smaller than $\varepsilon_0 n$ for some $\varepsilon_0 > 0$. The remainder of the proof is essentially an application of the first moment method.

The proof of Theorem~\ref{LBgeneral} in the supercritical regime is similar in spirit to the one of Theorem~\ref{LB3reg}, although there we are only interested in paths of length $\ell$ for some large enough $\ell$. An additional difficulty appears: when reasoning about the giant component, a number of vertices of degree 1 may arise in general. These may cause problems in case they increase the size of $C_0(\varepsilon)$ by too much. Nevertheless, we show that this is indeed not the case with probability tending to 1 as $n \to +\infty$. The subcritical regime in Theorem~\ref{LBgeneral} is based on an analysis of the orders of connected components in $G_n$.\\

\par\noindent 
\textbf{Notation.}
For a graph $G=(V,E)$ we call $|V|$ the \textit{order} of $G$ and $|E|$ the \textit{size} of $G$. For a vertex $v$ of $G$, we denote $\degg_G(v)$ or simply $\degg(v)$ the degree of the vertex $v$ in $G$. We also call $(u,w,v) \in V^3$ a \textit{cherry with center $w$}, if $uw, wv \in E$.  For a path $p$, the \textit{length} of $p$ is the number of edges in $p$. For a set $S\subseteq V$, an $S-$\textit{chain} or simply a \textit{chain} is a path of vertices $u_0=u, u_1, \ldots, u_k, u_{k+1}=v$, with $u,v \in S; u_1, \dots, u_k\in V\setminus S$ and, for every $0 \le i \le k$, $u_i u_{i+1} \in E$. For example, in Figure~\ref{3-chain} we see an $C_0(\varepsilon)$ chain of length 3. A \textit{leaf} in $G$ is a vertex of $G$ of degree 1. For two subsets $A,B \subseteq V$, we also denote by $e(A,B)$ the number of edges between $A$ and $B$.
\par\noindent
\textbf{Organization of the paper}.
The paper is organized as follows: in Section~\ref{sec:Prelim} we introduce preliminary definitions and concepts. We then prove Theorem~\ref{LB3reg} in Section~\ref{sec:Lower} and Theorem~\ref{UB3reg} in Section~\ref{sec:Upper}. Section~\ref{sec:General} is devoted to the proof of Theorem~\ref{LBgeneral}. We conclude with further remarks in Section~\ref{sec:Discussion}.

\section{Preliminaries}\label{sec:Prelim} 
In this section we collect concepts that will be used later on, both coming from graph theory as well as from probability theory.

\subsection*{Graph theoretic preliminaries}
First, for a graph $G(V,E)$ with $|V| = n$ and $|E| = m$ we define the \textit{relative modularity $q_r(A)$} of a set $A\subseteq V$:
\begin{equation*}
q_r(A) := \dfrac{n}{|A|} \left(\dfrac{e(A)}{m} - \dfrac{\vol(A)^2}{4m^2}\right).
\end{equation*}

Denoting by $d$ the average degree of the graph $G$, one may rewrite this formula as
\begin{equation}
q_r(A) := \dfrac{2e(A)}{d|A|} - \dfrac{\vol(A)^2}{d^2|A|n}.
\end{equation}

The main motivation of this definition is that one may define the modularity of a partition $\mathcal A = (A_1, A_2, \dots, A_k)$ of $V$ as a weighted average of the relative modularities of its parts:
\begin{equation*}
q(\mathcal A) = \sum_{1\le i\le k} \dfrac{|A_i|}{n} q_r(A_i).
\end{equation*}

In particular, if the modularity of the partition $\mathcal A$ is at least $q$, then there must exist a part $A_i$ with relative modularity at least $q$. At the same time, the relative modularity of a part $A_i$ depends only on $A_i$ itself and not on the partition that contains it.\par

One may easily remark that the relative modularity of a set of vertices $A$ in a $d-$regular graph can be rewritten as
\begin{equation*}
q_r(A) = \dfrac{2e(A)}{d|A|} -  \dfrac{|A|}{n}.
\end{equation*}
 
For the sake of completeness we include the proof of the following well-known result:

\begin{lemma}\label{lem:partition}
Every tree $T$ on $k\ge \lfloor\sqrt{n}\rfloor$ vertices with maximum degree $\Delta$ can be partitioned into subtrees of order between $\lfloor\sqrt{n}\rfloor$ and $\Delta\lceil\sqrt{n}\rceil$.
\end{lemma}
\begin{proof}
We argue by induction on $k$. If $k \le \Delta\lceil\sqrt{n}\rceil$, then we already have a tree of the prescribed order. Suppose that the induction hypothesis is satisfied for some $k\ge \Delta\lceil\sqrt{n}\rceil$. Let $T^{k+1}$ be a tree of order $k+1$ and maximum degree at most $\Delta$. Then, write on every edge $f$ of $T^{k+1}$ the orders $x_f,y_f$ of the two subtrees of $T^{k+1}$ in $T^{k+1}\setminus f$.\par 
We claim that there must be an edge in $T^{k+1}$, for which the minimal number of the two numbers written on it is at least $\lfloor\sqrt{n}\rfloor$. We argue by contradiction. Suppose that this is not the case. Let $e = uv$ be the edge, for which 
\begin{equation*}
x_e = |T^{k+1}_u|\le y_e = |T^{k+1}_v| \text{ and } x_e = \max_{f\in E(T^{k+1})} \min(x_f, y_f),
\end{equation*}
where $T^{k+1}_u \cup T^{k+1}_v = T^{k+1}\setminus e$ and $u\in T^{k+1}_u$. Now, since $x_e \ge \min\{x_{e_i}, y_{e_i}\}$ for every edge $e_i \in E(T^{k+1})$, deleting any edge $e_i\neq e$ incident to $v$ would yield that the smaller tree in $T^{k+1}\setminus e_i$ (which must be the one not containing $v$, see Figure~\ref{fig:lem 2.1}) of order $x_{e_i}$ satisfies $x_{e_i} \le x_e\le \lceil\sqrt{n}\rceil - 1$.  Thus   $T^{k+1}$ would contain at most $\Delta(\lfloor\sqrt{n}\rfloor - 1) + 1$ vertices, which is not the case since $k\ge \Delta\lceil\sqrt{n}\rceil$. This is a contradiction, which proves that $x_e\ge \lfloor \sqrt{n}\rfloor$.\par

Therefore, $T\setminus e$ consists of two trees of orders at least $\lfloor \sqrt{n}\rfloor$ and less than $k+1$. The induction hypothesis thus applies to both of them. The lemma is proved.
\end{proof}

\begin{figure}
\centering
\begin{tikzpicture}[line cap=round,line join=round,x=1cm,y=1cm]
\clip(-4.5,-7) rectangle (10,2);
\draw [line width=0.5pt] (-0.44,-2.3)-- (3.4,-2.28);
\draw [line width=0.5pt] (3.4,-2.28)-- (4.68,0.5);
\draw [line width=0.5pt] (3.4,-2.28)-- (7.04,-1.84);
\draw [line width=0.5pt] (3.4,-2.28)-- (4.74,-4.96);
\draw [rotate around={36.86989764584469:(-0.84,-2.6)},line width=2pt] (-0.84,-2.6) ellipse (1.609524614471849cm and 1.5298919846155026cm);
\draw [line width=0.5pt] (-1.24,-2.9)-- (-0.44,-2.3);
\draw [line width=0.5pt] (-1.24,-2.9)-- (-0.12,-3.16);
\draw [line width=0.5pt] (-0.12,-3.16)-- (-0.94,-3.8);
\draw [line width=0.5pt] (-1.24,-2.9)-- (-1.92,-2.18);
\draw [line width=0.5pt] (-1.92,-2.18)-- (-1.2,-1.7);
\draw [line width=0.5pt] (-1.24,-2.9)-- (-1.94,-3.4);
\draw [line width=0.5pt] (4.74,-4.96)-- (5.08,-5.8);
\draw [line width=0.5pt] (4.74,-4.96)-- (5.66,-5.22);
\draw [line width=0.5pt] (5.66,-5.22)-- (5.82,-6.08);
\draw [line width=0.5pt] (5.66,-5.22)-- (6.72,-5.52);
\draw [line width=0.5pt] (6.72,-5.52)-- (7.5,-4.86);
\draw [line width=0.5pt] (7.04,-1.84)-- (7.88,-2.14);
\draw [line width=0.5pt] (7.04,-1.84)-- (8.08,-1.32);
\draw [line width=0.5pt] (4.68,0.5)-- (5.48,0.84);
\draw [line width=0.5pt] (5.48,0.84)-- (6.36,0.56);
\draw [line width=0.5pt] (5.48,0.84)-- (5.78,0.22);
\draw [rotate around={88.73632744834094:(5.72,-2.5)},line width=2pt] (5.72,-2.5) ellipse (4.341233882335336cm and 3.3829442237696363cm);
\begin{scriptsize}
\draw [fill=black] (-0.44,-2.3) circle (2.5pt);
\draw [fill=black] (3.4,-2.28) circle (2.5pt);
\draw [fill=black] (4.68,0.5) circle (2.5pt);
\draw [fill=black] (7.04,-1.84) circle (2.5pt);
\draw [fill=black] (4.74,-4.96) circle (2.5pt);
\draw [fill=black] (-1.24,-2.9) circle (2.5pt);
\draw [fill=black] (-0.12,-3.16) circle (2.5pt);
\draw [fill=black] (-0.94,-3.8) circle (2.5pt);
\draw [fill=black] (-1.92,-2.18) circle (2.5pt);
\draw [fill=black] (-1.2,-1.7) circle (2.5pt);
\draw [fill=black] (-1.94,-3.4) circle (2.5pt);
\draw [fill=black] (5.08,-5.8) circle (2.5pt);
\draw [fill=black] (5.66,-5.22) circle (2.5pt);
\draw [fill=black] (5.82,-6.08) circle (2.5pt);
\draw [fill=black] (6.72,-5.52) circle (2.5pt);
\draw [fill=black] (7.5,-4.86) circle (2.5pt);
\draw [fill=black] (7.88,-2.14) circle (2.5pt);
\draw [fill=black] (8.08,-1.32) circle (2.5pt);
\draw [fill=black] (5.48,0.84) circle (2.5pt);
\draw [fill=black] (6.36,0.56) circle (2.5pt);
\draw [fill=black] (5.78,0.22) circle (2.5pt);
\draw [fill=black] (6.84,-3.16) circle (0.5pt);
\draw [fill=black] (6.64,-3.56) circle (0.5pt);
\draw [fill=black] (6.44,-3.92) circle (0.5pt);
\draw [fill=black] (6.22,-4.32) circle (0.5pt);
\draw[color=black] (-0.5,-2.08) node {\large $u$};
\draw[color=black] (0.1,-2.3) node[fill = white, draw] {\large $x_e$};
\draw[color=black] (2.8,-2.3) node[fill = white, draw] {\large $y_e$};
\draw[color=black] (3.28,-2.08) node {\large $v$};
\draw[color=black] (1.45,-2.55) node {\large $e$};
\draw[color=black] (4.5,0) node[fill = white, draw] {\large $x_{e_1}$};
\draw[color=black] (3.7,-1.65) node[fill = white, draw] {\large $y_{e_1}$};
\draw[color=black] (3.7,-0.85) node {\large $e_1$};
\draw[color=black] (6.5,-1.9) node[fill = white, draw] {\large $x_{e_2}$};
\draw[color=black] (4.3,-2.2) node[fill = white, draw] {\large $y_{e_2}$};
\draw[color=black] (5.2,-1.8) node {\large $e_2$};
\draw[color=black] (4.6,-4.5) node[fill = white, draw] {\large $x_{e_s}$};
\draw[color=black] (3.6,-2.8) node[fill = white, draw] {\large $y_{e_s}$};
\draw[color=black] (4.4,-3.65) node {\large $e_s$};
\end{scriptsize}
\end{tikzpicture}
\caption{Illustration of the Proof of Lemma~\ref{lem:partition}.
}
\label{fig:lem 2.1}
\end{figure}
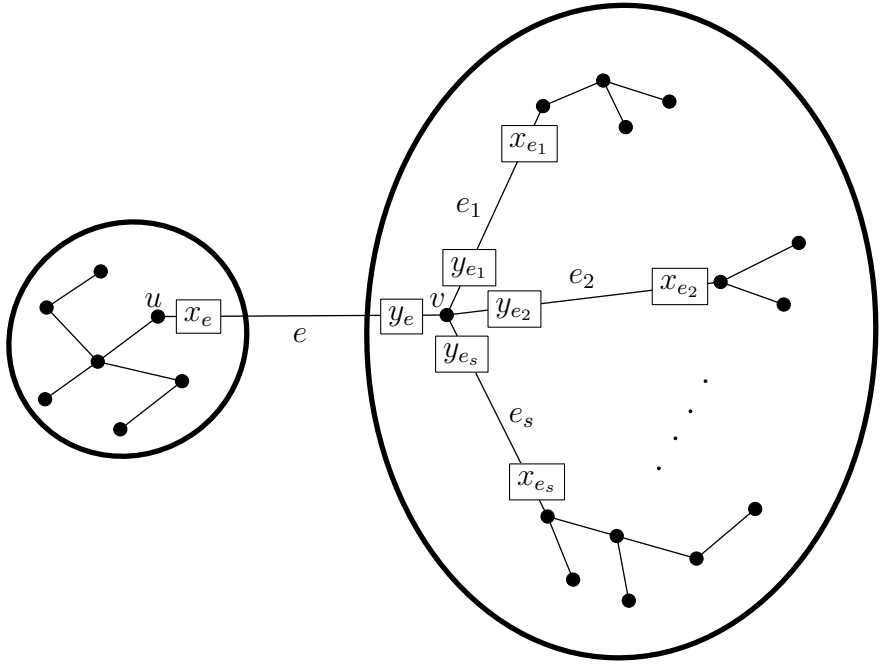

The next lemma counts the number of graphs of given degree sequence. It can be found in a more general form in \cite{Bol}.

\begin{lemma}[\cite{Bol}, Theorem 2.16]\label{count}
For a fixed number $\Delta$ and 
\begin{equation*}
2\le \deg(1) \le \deg(2) \le \dots \le \deg(n) \le \Delta, \sum_{1\le i\le n} \deg(i) = 2m,
\end{equation*}
the number of simple graphs on the vertex set $[n]$ is equivalent, for $n\to +\infty$, to
\begin{equation*}
\exp(-\lambda/2 -\lambda^2/4) \dfrac{(2m)!}{2^m m!} \prod_{1\le j\le n} \dfrac{1}{\deg(j)!},
\end{equation*}
where 
\begin{equation*}
\lambda = \dfrac{1}{m} \sum_{1\le i\le n} \binom{\deg(i)}{2}.
\end{equation*}
\qed
\end{lemma}

Under the conditions of Lemma~\ref{count}, Wormald proves in \cite{Wor81b} that the probability that a random graph on a given bounded degree sequence contains for a given $\ell \in \mathbb N$, exactly $c_i$ cycles of length $i$ for every $i\in [\ell]$, is given by some function of $c_1, c_2, \dots, c_{\ell}$, and in particular it is bounded from below over the set of degree sequences given in Lemma~\ref{count} by a (universal) positive constant depending only on $\Delta$. For $c_1 = c_2 = 0$ this implies that the configuration model over the sequence 
\begin{equation*}
    2\le \deg(1) \le \deg(2) \le \dots \le \deg(n) \le \Delta
\end{equation*}
produces a simple graph with probability that is bounded from below by a universal positive constant. From here we deduce the following corollary of Lemma~\ref{count}.

\begin{corollary}\label{cor 2.3}
The number of (multi-)graphs on $n$ vertices and $m$ edges on the degree sequence 
\begin{equation*}
    2\le \deg(1) \le \deg(2) \le \dots \le \deg(n) \le \Delta, \sum_{1\le i\le n} \deg(i) = 2m
\end{equation*}
is 
\begin{equation*}
\Theta\left(\dfrac{(2m)!}{2^m m!} \prod_{1\le j\le n} \dfrac{1}{\deg(j)!}\right).
\end{equation*}
\qed
\end{corollary}

We continue with a simple general lower bound on the modularity:
\begin{lemma}\label{lem:simplelb}
The modularity of a graph $G$ on $n$ vertices, with average degree $d$ and maximal degree bounded from above by $\Delta\in \mathbb N$, is at least
\begin{equation*}
    \dfrac{2(n - |CC(G)|)}{dn} - O\left(\dfrac{1}{\sqrt{n}}\right).
\end{equation*}
\end{lemma}

In the proof of Lemma~\ref{lem:simplelb} we apply Lemma~\ref{lem:partition} to divide the graph into components of order at most $\Delta\lceil \sqrt{n}\rceil$.

\begin{proof}[Proof of Lemma~\ref{lem:simplelb}]
Let us call $\mathcal C = (C_1, C_2, \dots, C_r)$ the partition of the vertices of $G$ into connected components. For every connected component $C_i$ with more than $\Delta\lceil \sqrt{n}\rceil$ vertices, apply Lemma~\ref{lem:partition} to an arbitrary spanning tree of $C_i$. We deduce that $C_i$ can be partitioned into connected subgraphs of order between $\lfloor \sqrt{n} \rfloor$ and $\Delta\lceil \sqrt{n}\rceil$. Let $\mathcal A = \{A_1, A_2, \dots, A_k\}$ be the partition obtained from $\mathcal C$ after dividing the connected components of $G$ of order more than $\Delta\lceil \sqrt{n}\rceil$.\par

By convexity of the function $x\in \mathbb R \mapsto x^2\in \mathbb R$, the sum $\sum_{1\le i\le k} |A_i|^2$ is maximal when all but at most one of the terms are either equal to $\Delta\lceil \sqrt{n}\rceil$ or to 0 for any fixed $k$. We have that
\begin{align*}
   &\sum_{1\le i\le k} |A_i|^2 \\
   \le  
   &\hspace{0.25em} \max_{1\le i\le k} |A_i| \sum_{1\le i\le k} |A_i|\\
   \le 
   &\hspace{0.25em} \Delta n\lceil \sqrt{n}\rceil.
\end{align*}

Since in the end we have exactly $k$ connected components induced by the vertex sets in $\mathcal A$, we obtain
\begin{equation*}
    q^*(G_n) \ge q(\mathcal A) = \dfrac{2}{nd}\sum_{i=1}^{k} e(A_i) - \dfrac{1}{n^2d^2}\sum_{i=1}^{k} \vol(A_i)^2 \ge \dfrac{2(n-1-(r-1)-(k-r))}{nd} - \dfrac{\Delta^3 \lceil\sqrt{n}\rceil}{nd^2}.
\end{equation*}
It remains to observe that all parts in $\mathcal A$ obtained from the division of connected components of $G$ of order more than $\Delta\lceil \sqrt{n}\rceil$ have order at least $\lfloor \sqrt{n}\rfloor$ by definition, so $k-r \le \dfrac{n}{\lfloor \sqrt{n}\rfloor} = O\left(\sqrt{n}\right)$. This proves the lemma.
\end{proof}

\subsection*{Probabilistic preliminaries}
In this subsection we gather probabilistic concepts used throughout the paper. 
We first recall the following version of Chernoff's bound, see for example (\cite{JLR}, Corollary 2.3).

\begin{lemma}\label{chernoff:bd}
Let $X \sim Bin(n,p)$ be a binomial random variable with $\mathbb E[X]=np=\mu$. For every $0 \le \delta \le 1$,
\begin{equation*}
\mathbb P(|X-\mu| \ge \delta \mu) \le 2\exp\bigg (-\dfrac{\delta^2 \mu}{3}\bigg).
\end{equation*}
\qed
\end{lemma}

\subsubsection*{Configuration model}
The probability space of (multi-)graphs, with which we will be working until the end of this paper, is the \textit{configuration model} introduced by Bender and Canfield in~\cite{BC} and further developed by Bollobás in~\cite{Bol} and by Wormald in~\cite{W78}. We describe it first for the case of $d$-regular graphs: we are given $dn$ points (also called \textit{half-edges}), with $dn$ being even, indexed by $(P_{i,j})_{1\leq i\leq d, 1\leq j\leq n}$ and regrouped into $n$ buckets according to their second index. The probability space we work with is the space of perfect matchings of these $dn$ points equipped with the uniform probability. We call \textit{configuration} a perfect matching of $(P_{i,j})_{1\leq i\leq d, 1\leq j\leq n}$. 
We now reconstruct the random $d-$regular graph model as follows: we identify the $d-$point buckets with the vertices of our random graph. By abuse of terminology, we use both buckets and vertices in the sequel to refer to the same objects by the above identification. An edge in the random regular graph between two (not necessarily different) vertices $v$ and $v'$ corresponds to an edge of the configuration between a point $P$ in the bucket $v$ and a point $P'$ in the bucket $v'$. It is well known that this model is contiguous to the uniform distribution on random $d-$regular graphs for constant values of $d$, see~\cite{J95b}. This model can then be easily generalized for graphs with given degree sequences: given a sequence $(d_1, \ldots, d_n)$ with $d_i$ denoting the degree of the $i$-th vertex such that $\sum_{i=1}^n d_i = 2m$ for some $m \in \mathbb{N}$, identify the $i$-th vertex with a bucket having $d_i$ points. As before, choose a perfect matching uniformly at random among all pairings and add an edge in the graph between the two vertices $v_1$ and $v_2$ for every pair of points $(P_1, P_2)$ from the buckets $v_1$ and $v_2$ that participates in a common edge of the matching.

\subsection*{Differential equation method}
The theory of differential equations used to describe the evolution of a discrete random process was introduced by Wormald (see~\cite{Nick1, Nick2, Nick3}). Given a sequence of discrete random variables $(X_t)_{t \ge 0}$, the basic idea is to consider the expected change between times $t$ and $t+1$. Regarding the trajectories $(X_t)_{t \ge 0}$ (properly rescaled) as continuous, one may write the ordinary differential equations suggested by the expected changes. Concentration results from martingale theory are then used to show that, as the size of the input grows large, under relatively mild conditions the trajectory $(X_t)_{t \ge 0}$ is highly concentrated around the value suggested by the solution of the differential equation for a wide range of $t$.\par

The precise formulation of the theorem given here is taken from~\cite{Lutz}: we say that a function $f$ is said to be $L$-Lipschitz  on $D \subseteq R^{\ell}$, if $|f(x) - f(x')| \le L \max_{1 \le k \le \ell} |x_k - x'_k|$ holds for all points $x=(x_1, \ldots, x_{\ell})$ and $x'=(x'_1, \ldots, x'_{\ell})$ in $D$, where $\max_{1 \le k \le \ell} |x_k - x'_k|$ is the $\ell^{\infty}-$distance between $x$ and $x'$.
\begin{theorem}\cite{Lutz}\label{Thm:DEMethod}
Given $a, n \ge 1$, a bounded domain $D \subseteq \mathbb{R}^{a+1}$, functions $(F_k)_{1 \le k \le a}$ with $F_k: D \to \mathbb{R}$, and $\sigma$-algebras ${\mathcal{F}}_{0}\subseteq {\mathcal{F}}_{1} \subseteq \ldots$, suppose that the random variables $(Y_k(i))_{1 \le k \le a}$ are ${\mathcal{F}}_{i}$-measurable for $i \ge 0$. Suppose also that for all $i \ge 0$ and all $1 \le k \le a$, the following holds whenever $(i/n, Y_1(i)/n, \ldots, Y_a(i)/n) \in D$:
\begin{enumerate}
\item $\left| \mathbb{E}(Y_k(i+1) - Y_k(i) \mid  {\mathcal{F}}_{i}) - F_k(i/n, Y_1(i)/n, \ldots, Y_a(i)/n)\right| \le \delta$ for some $\delta \ge 0$, with $F_k$ being $L$-Lipschitz for $L \in \mathbb{R}$.
\item $|Y_k(i+1) - Y_k(i)| \le \beta$ for some $\beta > 0$,
\item $\max_{1 \le k \le a}|Y_k(0) - \hat{y}_kn| \le \lambda n$ for some $\lambda > 0$, for some $(0, \hat{y}_1, \ldots, \hat{y}_a) \in D$.
\end{enumerate}
Then there are $R=R(D,(F_k)_{1 \le k \le a},L) \in [1, \infty)$ and $T=T(D) \in (0, \infty)$ such that for $\lambda \ge \delta \min \{T, 1/L\} + R/n$, so that with probability at least $1-2a\exp\left(-\dfrac{n\lambda^2}{8T\beta^2}\right)$ we have
\begin{equation*}
\max_{0 \le i \le \sigma n} \max_{1 \le k \le a} |Y_k(i) - y_k(i/n) n| \le 3\exp(LT) \lambda n,
\end{equation*}
where $(y_k(t))_{1 \le k \le a}$ is the unique solution to the system of differential equations $y_k'(t)=F_k(t,y_1(t), \ldots, y_a(t))$ with $y_k(0)=\hat{y}_k$ for $1 \le k \le a$, and $\sigma=\sigma(\hat{y}_1, \ldots, \hat{y}_a) \in [0,T]$ is any choice of $\sigma \ge 0$ with the property that $(t, y_1(t), \ldots, y_A(t))$ has $\ell_{\infty}$-distance at least $3\exp(LT)\lambda$ from the boundary of $D$ for all $t \in [0, \sigma)$. 
\qed
\end{theorem}

\begin{remark}
In this paper we only work with differential equations of the type 
\begin{equation*}
x'(t) = F(x(t), t),
\end{equation*}
where $F$ is a Lipschitz function on a domain $D$. Thus, every differential equation with given initial values will admit a unique solution.
\qed
\end{remark}

Let $a,b$ be two positive real numbers with $b < a$. Let $(U_i)_{1\le i\le \lfloor an\rfloor}$ be urns, each of them with space for at most two balls, and $(B_j)_{1\le j\le \lfloor 2bn\rfloor}$ be balls that are, one after the other, thrown uniformly into some urn, where the probability that a ball is thrown into an urn is proportional to the free space in this urn at the moment of throwing.

\begin{lemma}\label{urns}
A.a.s. at the end of the process there will be $\dfrac{b(2a-b)}{a}n + o(n)$ urns with at least 1 ball.
\end{lemma}
\begin{proof}
Denote by $N_0$ the number of urns containing no ball after $2bn$ steps. We have that
\begin{equation*}
    \mathbb E[N_0] = \dfrac{\binom{2an-2}{2bn}}{\binom{2an}{2bn}} an = (1+o(1)) \dfrac{(a-b)^2}{a} n
\end{equation*}
and 
\begin{equation*}
    \mathbb E[N_0(N_0-1)] = \dfrac{\binom{2an-4}{2bn}}{\binom{2an}{2bn}} an (an-1) = (1+o(1)) \dfrac{(a-b)^4}{a^2} n^2.
\end{equation*}
One may conclude by the second moment method that the number of urns with at least 1 ball in the end of the process is a.a.s. $(1+o(1))\left(an - \frac{(a-b)^2}{a} n\right) = (1+o(1)) \frac{b(2b-a)}{a} n$.

\end{proof}

We finish this section with a direct consequence of Theorem 2.2 in \cite{vHF}.

\begin{lemma}[see \cite{vHF}, Theorem 2.2]\label{connectivity}
Given $\Delta \ge 2$, let $(D_n)_{n\ge 1}$ be a bounded sequence of degree sequences such that one has $d_i(n)$ vertices of degree $i$ for every $1 \le i\le \Delta$ and $d_j(n) = 0$ for every $j\ge \Delta + 1$. Suppose that there is a constant $c < 1$ such that, for every $n$, $d_1(n) = 0$ and $d_2(n)\le cn$. The probability that the random graph on the degree sequence $D_n$ is connected is bounded from below by $(1-cn/m)^{1/2} + o(1)$, and moreover its largest component $C_{\max}(n)$ contains a.a.s.\ all but at most $\log n$ vertices.
\qed
\end{lemma}

We remark that in the above theorem, $\log n$ may be replaced by any function $\omega(n)$ that tends to infinity with $n$.\par

\section{\texorpdfstring{Lower bound in the case of $3-$regular graphs}{}}\label{sec:Lower}
In this section we prove Theorem~\ref{LB3reg}, thereby improving Lemma~\ref{lem:simplelb}, which only gives an a.a.s.\ lower bound of $\frac{2}{3} - O\left(\frac{1}{\sqrt{n}}\right)$ for the modularity of the random $3-$regular graph. Indeed, since the random $3-$regular graph is a.a.s.\ connected, see for example (\cite{Bol}, Section 7.6), $|CC(G_3(n))| = 1$.\par

We work in the configuration model defined above. Choose a random vertex $v_0$ and start an exploration process, which goes as follows. At every step, fix a uniformly chosen open half-edge at some explored vertex (if there is no such edge, stop the process - this does not happen a.a.s.). Once we have chosen this half-edge, look where it goes. If it adds a new vertex, add it to the already explored ones and go to the next step. If it goes back to an already explored vertex, construct it and continue. By abuse of notation we call the explored graph at time $t$ the \textit{component at time $t$} since a.a.s.\ it is connected. We stop the exploration process when at least $\varepsilon n$ vertices have been explored. Together with the explored edges they will form a graph $C_0 = C_0(\varepsilon)$. We first give an a.a.s.\ estimate of the number of edges that will be present in $C_0$. Let $X(t)$ be the number of vertices in the component at time $t$, and let $\mathcal F_t$ be the complete history of explored vertices and edges up to time $t$. We have that
\begin{equation*}
    \mathbb E[X(t+1)\hspace{0.2em}|\hspace{0.2em}\mathcal F_t] = X(t) + \dfrac{3(n - X(t))}{3n - 2t - 1}.
\end{equation*}
Since the reasoning behind this formula is often used in the sequel, we explain it here. Fix one unmatched half-edge just after step $t$ emanating from a vertex already in the component at time $t$. Then, there are in total $3n-2t-1$ remaining unmatched half-edges, and exactly $3(n-X(t))$ of them are sticking out of vertices not yet in the component. Hence, the probability that the fixed half-edge is paired with a half-edge incident to a vertex not yet in the component is $\frac{3(n - X(t))}{3n - 2t - 1}$. Note that we also have $X(0)=0$. Our goal now is to apply Theorem~\ref{Thm:DEMethod}: we first transform the difference equation corresponding to the expected change into a differential equation and justify this step afterwards:
\begin{equation*}
x'(t)=\frac{3-3x(t)}{3-2t} \text{ with initial value } x(0)=0.
\end{equation*}

The solution is given by
\begin{equation*}
    x(t) = 1 - \left(1 - \dfrac{2t}{3}\right)^{3/2}.
\end{equation*}

Fix $t_0 = t_0(\varepsilon) := \frac{3(1 - (1 - \varepsilon)^{2/3})}{2}$. Then, there are exactly $\varepsilon n$ explored vertices in $t_0 n +o(n)$ steps both in expectation and a.a.s. (the proof of the corresponding a.a.s. statement is given in the next paragraph).\par
We now justify the passage to a differential equation: first, the expected difference between $X(t+1)$ and $X(t)$ differs from $\frac{3-3x(t)}{3-2t}$, for $n$ sufficiently large, by at most some term $\delta=O(1/n)$ as long as the number of non-explored edges is still $c n$ for some $c > 0$. Next, for every $t\ge 0$, the difference between $X(t+1)$ and $X(t)$ is at most $1$, and the initial values of the differential equation and of the difference equation match. Hence the three conditions of Theorem~\ref{Thm:DEMethod} are satisfied (with $a=1$). Choosing $\lambda=n^{-1/3}$ (in fact, every $\lambda$ of the type $n^{-\delta}$ with $0 < \delta < 1/2$ would work as well) and $\sigma=t_0(\varepsilon)$, we have that, by Theorem~\ref{Thm:DEMethod}, with probability at least $1-\exp(-\Theta(n^{1/3}))$, $\max_{0 \le t \le \sigma n}|X(i)-n x(i/n)| =O(\lambda n)=O(n^{2/3})$. All subsequent transformations of difference equations to differential equations could be justified in an analogous way, and thus we omit them in the sequel.\\

\underline{Phase 1.}\\

After having found $\varepsilon n$ vertices, we are now ready for the first phase. Recall that we explored a component $C_0$, which is a.a.s.\ connected, but not necessarily an induced subgraph of $G_3(n)$. We thus explore the open half-edges going out of the vertices of $C_0$ in search for cherries, whose center is an unexplored vertex, but whose two leaves are in $C_0$, and also for edges in the component that have not been seen in the 0-th phase of construction of $C_0$ (see Figure~\ref{fig:phase 1}).\par

\begin{figure}
\centering
\begin{tikzpicture}[line cap=round,line join=round,x=1cm,y=1cm]
\clip(-3,-7.3) rectangle (10,1.5);
\draw [rotate around={-2.1768617100377403:(3.34,-2.83)},line width=2pt] (3.34,-2.83) ellipse (4.703983564874782cm and 3.2270979809438787cm);
\draw [line width=1pt] (-0.08,-2.7)-- (1.16,-3.88);
\draw [line width=1pt] (1.16,-3.88)-- (1.6,-1.72);
\draw [line width=1pt] (1.16,-3.88)-- (3.6,-2.4);
\draw [line width=1pt] (3.6,-2.4)-- (3.58,-3.8);
\draw [line width=1pt] (3.58,-3.8)-- (5.36,-3.32);
\draw [line width=1pt] (5.36,-3.32)-- (5.24,-2.12);
\draw [line width=1pt] (5.24,-2.12)-- (6.76,-2.96);
\draw [line width=1pt, opacity = 0.2] (6.76,-2.96)-- (5.36,-3.32);
\draw [line width=1pt] (3.6,-2.4)-- (3.56,-0.88);
\draw [line width=1pt] (3.58,-3.8)-- (3.08,-5.02);
\draw [line width=1] (3.08,-5.02)-- (5.44,-4.4);
\draw [line width=1pt, opacity = 0.2pt] (1.6,-1.72)-- (3.56,-0.88);
\draw [line width=1pt, opacity = 0.2pt] (-0.08,-2.7)-- (-2.02,-1.94);
\draw [line width=1pt, opacity = 0.2pt] (-0.08,-2.7)-- (-0.26,0.14);
\draw [line width=1pt, opacity = 0.2pt] (-0.26,0.14)-- (1.6,-1.72);
\draw [line width=1pt, opacity = 0.2pt] (3.56,-0.88)-- (5.48,0.88);
\draw [line width=1pt, opacity = 0.2pt] (5.48,0.88)-- (5.24,-2.12);
\draw [line width=1pt, opacity = 0.2pt] (6.76,-2.96)-- (9.64,-2.8);
\draw [line width=1pt, opacity = 0.2pt] (5.44,-4.4)-- (7.8,-5.74);
\draw [line width=1pt, opacity = 0.2pt] (5.44,-4.4)-- (5.52,-6.92);
\draw [line width=1pt, opacity = 0.2pt] (3.08,-5.02)-- (2.12,-6.76);
\draw [line width=1pt, opacity = 0.2pt] (-0.26,0.14)-- (-1.56,0.14);
\draw [line width=1pt, opacity = 0.2pt] (5.48,0.88)-- (7,1);
\begin{scriptsize}
\draw[color=black] (-0.1, 0.3) node {\large $v_1$};
\draw[color=black] (5.2, 1) node {\large $v_2$};
\draw[color=black] (2.6, -1.5) node {\large $e_1$};
\draw[color=black] (6.1, -3.3) node {\large $e_2$};
\draw [fill=black] (-0.08,-2.7) circle (2.5pt);
\draw [fill=black] (6.76,-2.96) circle (2.5pt);
\draw [fill=black] (1.16,-3.88) circle (2.5pt);
\draw [fill=black] (1.6,-1.72) circle (2.5pt);
\draw [fill=black] (3.6,-2.4) circle (2.5pt);
\draw [fill=black] (3.58,-3.8) circle (2.5pt);
\draw [fill=black] (5.36,-3.32) circle (2.5pt);
\draw [fill=black] (5.24,-2.12) circle (2.5pt);
\draw [fill=black] (3.56,-0.88) circle (2.5pt);
\draw [fill=black] (3.08,-5.02) circle (2.5pt);
\draw [fill=black] (5.44,-4.4) circle (2.5pt);
\draw [fill=black] (-2.02,-1.94) circle (2.5pt);
\draw [fill=black] (-0.26,0.14) circle (2.5pt);
\draw [fill=black] (-1.56,0.14) circle (2.5pt);
\draw [fill=black] (5.48,0.88) circle (2.5pt);
\draw [fill=black] (7,1) circle (2.5pt);
\draw [fill=black] (9.64,-2.8) circle (2.5pt);
\draw [fill=black] (7.8,-5.74) circle (2.5pt);
\draw [fill=black] (5.52,-6.92) circle (2.5pt);
\draw [fill=black] (2.12,-6.76) circle (2.5pt);
\end{scriptsize}
\end{tikzpicture}
\caption{The black graph in the figure is $C_0$. The solid black edges are the ones that have been explored during the 0-th phase, the opaque edges were not explored during the 0-th phase. The edges $e_1$ and $e_2$ are added to $C_0$ during the first phase since these are edges between two vertices explored during the 0-th phase. The vertices $v_1$ and $v_2$ are added to $C_0$ during the first phase since these are centers of cherries $(u_i, v_i, w_i)$ with $u_i, w_i\in C_0$ for both $i=1,2$.}
\label{fig:phase 1}
\end{figure}
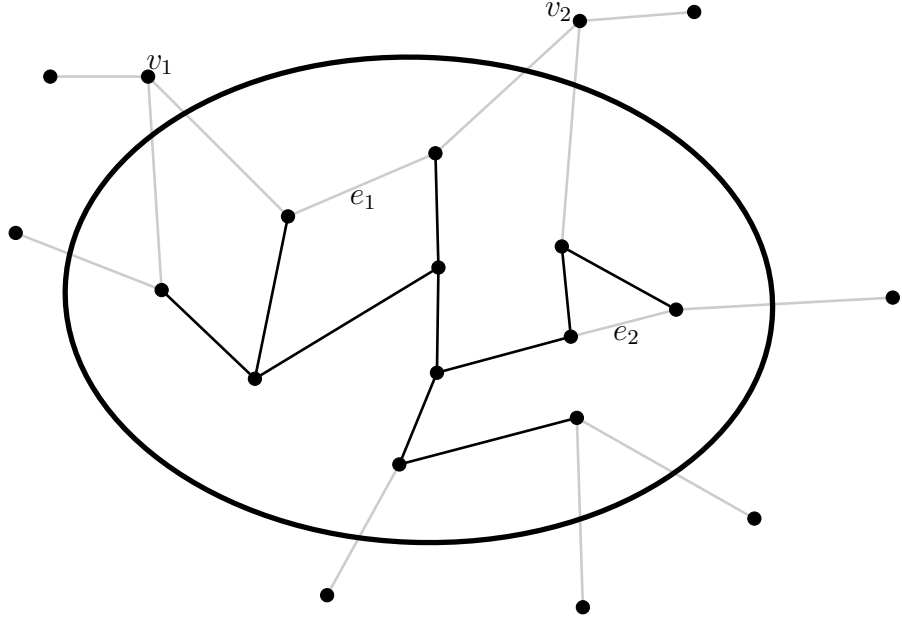

We order the half-edges in the component not yet matched and at any step we check where a half-edge goes. Translate time so that the first phase starts at $t = 0$ and not at $t = t_0(\varepsilon)$, as it should have since it comes right after the 0-th phase. We denote by $X^1_0(t)$ the number of vertices of degree 0 at time $t$ (that is, the number of vertices outside $C_0$ at time $t$, for which none of the three incident half-edges has been exposed), by $X^1_1(t)$ the vertices of degree 1 outside $C_0$ at time $t$ (one half-edge of such vertex has been exposed) and by $X^1_{2,3}(t)$ the vertices of degree 2 or 3 outside $C_0$ at time $t$ (2 or 3 half-edges of such a vertex have been exposed). We underline that $X^1_1(t)$ and $X^1_{2,3}(t)$ count only vertices that have had degree 0 at the end of the 0-th phase. Thus, in the beginning, $X^1_0(0) = (1 - \varepsilon)n$ and $X^1_1(0) = X^1_{2,3}(0) = 0$. We also denote by $A^1(t)$ the number of edges constructed in the component up to time $t$ that have not been there at the end of the 0-th phase and by $H^1(t)$ the number of half-edges remaining to be tested. We have the following initial conditions: $A^1(0) = 0, H^1(0) = (3\varepsilon - 2t_0(\varepsilon))n = 3(\varepsilon + (1-\varepsilon)^{2/3} - 1)n$. Let $\mathcal F_t$ denote the $\sigma-$algebra containing the complete history of explored half-edges up to time $t$. We have the following equations:
\begin{itemize}
    \item The vertices of degree 0 can only disappear, and this happens exactly when one new vertex of degree 1 appears:
    \begin{equation*}
        \mathbb E[X^1_0(t+1)\hspace{0.2em}|\hspace{0.2em}\mathcal F_t] = X^1_0(t) - \dfrac{3X^1_0(t)}{3n - 2t_0n - 2t - 1}.
    \end{equation*}
    In this case, the number of half-edges decreases by 1.
    \item The vertices of degree 1 disappear when a cherry is formed and appear when a vertex of degree 0 disappears:
    \begin{equation*}
        \mathbb E[X^1_1(t+1)\hspace{0.2em}|\hspace{0.2em}\mathcal F_t] = X^1_1(t) + \dfrac{3X^1_0(t)}{3n - 2t_0n - 2t - 1} - \dfrac{2X^1_1(t)}{3n - 2t_0n - 2t - 1}.
    \end{equation*}
    Here as well, the number of half-edges decreases by 1.
    \item The vertices of degree at least 2 counted by $X^1_{2,3}(t)$ appear exactly when a vertex of degree 1 disappears:
    \begin{equation*}
        \mathbb E[X^1_{2,3}(t+1)\hspace{0.2em}|\hspace{0.2em}\mathcal F_t] = X^1_{2,3}(t) + \dfrac{2X^1_1(t)}{3n - 2t_0n - 2t - 1}.
    \end{equation*}
    Moreover, at the creation of each vertex of degree 2, this vertex is immediately added to the explored component and its third half-edge, which stays unmatched up to this moment, is added to the ones to be tested. Thus, the number of half-edges does not change. 
    \item The number of edges between vertices of the component can only increase at each step. The probability of this event depends on the number of half-edges yet to be tested:
    \begin{equation*}
        \mathbb E[A^1(t+1)\hspace{0.2em}|\hspace{0.2em}\mathcal F_t] = A^1(t) + \dfrac{H^1(t) - 1}{3n - 2t_0n - 2t - 1}.
    \end{equation*}
    At any step a new edge inside the explored component is constructed, and hence $H^1(t)$ decreases by 2.
    \item Finally, the equation for $H^1(t)$ is given by
    \begin{equation*}
        \mathbb E[H^1(t+1) - H^1(t)\hspace{0.2em}|\hspace{0.2em}\mathcal F_t] = - 1 + \mathbb E[X^1_{2,3}(t+1) - X^1_{2,3}(t)\hspace{0.2em}|\hspace{0.2em}\mathcal F_t] - \mathbb E[A^1(t+1) - A^1(t)\hspace{0.2em}|\hspace{0.2em}\mathcal F_t].
    \end{equation*}
\end{itemize}

We remark that we continue until time $t_1(\varepsilon)$, which is the hitting time of $\varepsilon' n$ of the process $H^1(t)$ for some arbitrary $\varepsilon' > 0$, that is, the point, where the number of half-edges remaining to be tested is $\varepsilon' n$. In fact, for the purpose of Theorem~\ref{Thm:DEMethod} in this phase, one needs to choose $\varepsilon'$ to be strictly positive so that $\sigma$ can be set equal to $t_1(\varepsilon)$. However, $\varepsilon'$ can be chosen as close to 0 as we wish. Since our work will come down to purely numerical computation in the end, we may assume that $\varepsilon' \approx 10^{-17}$, that is, smaller than the numerical error of our calculations. In the same way as before, it can be checked that the conditions for transforming the above equations of expected changes into differential equations are satisfied. Therefore, the use of differential equations as approximation of the random processes defined above is justified by Theorem~\ref{Thm:DEMethod}. Rescaling the first process as $x_0(t) = X^1_0(\lfloor tn\rfloor)/n$ gives the following differential equation for the rescaled time parameter: 
\begin{equation*}
    x'_0(t) = - \dfrac{3x_0(t)}{3 - 2t_0 - 2t} \text{ with } x_0(0) = 1 - \varepsilon.
\end{equation*}
It has solution 
\begin{equation*}
    x_0(t) = (1 - \varepsilon)\left(1 - \dfrac{2t}{3 - 2t_0}\right)^{3/2}.
\end{equation*}

Plugging in this solution into the second differential equation for $x_1(t) = X^1_1(\lfloor tn\rfloor)/n$ and after rescaling of the time parameter we get
\begin{equation*}
    x'_1(t) = \dfrac{3x_0(t)}{3 - 2t_0 - 2t} - \dfrac{2x_1(t)}{3 - 2t_0 - 2t} \text{ with } x_1(0) = 0.
\end{equation*}
It has solution 
\begin{equation*}
    x_1(t) = 3(1-\epsilon)\left(\left(1 - \dfrac{2t}{3 - 2t_0}\right) - \left(1 - \dfrac{2t}{3 - 2t_0}\right)^{3/2}\right)
\end{equation*}
Finally, the evolution of $x_2(t) = X^1_{2,3}(\lfloor tn\rfloor)/n$ is described by the equation
\begin{equation*}
    x'_2(t) = \dfrac{2x_1(t)}{3 - 2t_0 - 2t}\text{ with } x_2(0) = 0.
\end{equation*}
After integrating we get 
\begin{equation*}
    x_2(t) = \dfrac{6(1-\varepsilon)}{3 - 2t_0}t + 2(1 - \varepsilon)\left(\left(1 - \dfrac{2t}{3 - 2t_0}\right)^{3/2} - 1\right).
\end{equation*}
Now, the same rescaling for $A^1(t)$ and $H^1(t)$ gives respectively
\begin{equation*}
    a'(t) = \dfrac{h(t)}{3 - 2t_0 - 2t} \text{ with } a(0) = 0
\end{equation*}
and
\begin{equation*}
    h'(t) = -1 + x'_2(t) - a'(t) = -1 + x'_2(t) - \dfrac{h(t)}{3 - 2t_0 - 2t} \text{ with } h(0) = 3\varepsilon - 2t_0 = 3(\varepsilon + (1-\varepsilon)^{2/3} - 1).
\end{equation*}

The solution of the second differential equation with this initial condition is given by 
\begin{equation*}
    h(t) = (6\varepsilon - 3 - 2t_0)\left(1 - \dfrac{2t}{3 - 2t_0}\right) + 3(1 - \varepsilon)\left(1 - \dfrac{2t}{3-2t_0}\right)^{3/2}.
\end{equation*}

Integrating the first equation to obtain $a(t)$ and using the initial condition we obtained yields
\begin{equation*}
    a(t) = (2\varepsilon - 1/2 - t_0) - (3\varepsilon - 3/2 - t_0)\left(1 - \dfrac{2t}{3 - 2t_0}\right) - (1 - \varepsilon)\left(1 - \dfrac{2t}{3-2t_0}\right)^{3/2}.
\end{equation*}

It remains to deduce the smallest time $t_1 = t_1(\varepsilon)$, for which $h(t_1) = 0$ and the first step terminates\footnote{Once again, formally we have to stop the process a little bit earlier in order to be able to apply Theorem~\ref{Thm:DEMethod} with $\sigma$ being the corresponding boundary point. Here and in what follows, we ignore this fact due to numerical errors we commit anyway.}. This time is the minimal positive solution of
\begin{equation*}
    \left(6\varepsilon - 3 - 2t_0 + 3(1-\varepsilon)\sqrt{1 - \dfrac{2t}{3-2t_0}}\right)\left(1 - \dfrac{2t}{3-2t_0}\right) = 0.
\end{equation*}
The two solutions of this equation are
\begin{equation*}
    \tsup[1]t(\varepsilon) = \dfrac{3 - 2t_0}{2} = \dfrac{3(1-\varepsilon)^{2/3}}{2}\text{ and } \tsup[2]t(\varepsilon) = \dfrac{3}{2}(4(1-\varepsilon)^{1/3} - 3(1-\varepsilon)^{2/3}-1).
\end{equation*}
Moreover, $\tsup[2]t(\varepsilon) < \tsup[1]t(\varepsilon)$ for every $\varepsilon$ such that
\begin{equation*}
6\varepsilon - 3 - 2t_0(\varepsilon) < 0 \iff \varepsilon < \dfrac{7}{8}.
\end{equation*}
In the sequel we assume that $\varepsilon < 7/8$ and therefore $t_1 = t_1(\varepsilon) := \tsup[2]t(\varepsilon)$.

By choosing $\sigma=t_1(\varepsilon)$ and $\lambda=n^{-1/3}$ in Theorem~\ref{Thm:DEMethod} we deduce that with probability at least $1-e^{-\Theta(n^{1/3})}$ one has
\begin{align*}
& \max_{0 \le t \le \sigma n}|X^1_0(i)-n x_0(i/n)| =O(\lambda n)=O(n^{2/3}), \\
& \max_{0 \le t \le \sigma n}|X^1_1(i)-n x_1(i/n)| =O(\lambda n)=O(n^{2/3}), \\
& \max_{0 \le t \le \sigma n}|X^1_{2,3}(i)-n x_2(i/n)| =O(\lambda n)=O(n^{2/3}), \\
& \max_{0 \le t \le \sigma n}|A^1(i)-n a(i/n)| =O(\lambda n)=O(n^{2/3}) \text{ and } \\
& \max_{0 \le t \le \sigma n}|H^1(i)-n h(i/n)| =O(\lambda n)=O(n^{2/3}).
\end{align*}

Call the component that was built at the end of the first phase $C_1 = C_1(\varepsilon)$, and also call the component consisting of all explored vertices and edges by $\overline{C_1} = \overline{C_1}(\varepsilon)$. Clearly $C_0\subseteq C_1\subseteq \overline{C_1}$, see Figure~\ref{fig:phase 2}.\\

\underline{Phase 2.}\\

Now, having $2nx_1(t_1)$ open half-edges attached to the vertices of $\overline{C_1}\setminus C_1$, we start testing for chains of length 3, for which only the first and the fourth vertex are in $C_1$ and the rest are in $\overline{C_1}\setminus C_1$. In the beginning, we order the $2nx_1(t_1)$ half-edges given above and match them one by one to free half-edges in vertices of $G_3\setminus C_1$. We underline that, for every half-edge that is matched to a vertex in $G_3\setminus \overline{C_1}$, we reveal only the information that this half-edge is matched to an unexplored vertex and do not reveal to which one exactly. Once again, we do a translation of the time parameter $t$ in order to start from 0 and not from $t_0+t_1$. We define the random variable $Z_0(t)$ to be the number of edges leading to vertices in $G_3\setminus \overline{C_1}$, and $Z_1(t)$ to be the number of edges formed between two vertices of $\overline{C_1}\setminus C_1$, see Figure~\ref{fig:phase 2}.

\begin{figure}
\centering
\begin{tikzpicture}[line cap=round,line join=round,x=1cm,y=1cm]
\clip(-2.7,-8) rectangle (9,2.1);
\draw [rotate around={-1.1160103914072743:(4.44,-3.18)},line width=2pt] (4.44,-3.18) ellipse (4.104632153389974cm and 2.712564306084342cm);
\draw [rotate around={0:(1,0)},line width=2pt] (3.3,-2.7) ellipse (5.5cm and 4.7cm);
\draw [line width=2pt] (1.36,-3.12)-- (2.16,-4.18);
\draw [line width=2pt] (1.36,-3.12)-- (2.16,-1.92);
\draw [line width=2pt] (2.16,-1.92)-- (3.76,-1.76);
\draw [line width=2pt] (3.76,-1.76)-- (3.08,-3.08);
\draw [line width=2pt] (3.08,-3.08)-- (2.16,-4.18);
\draw [line width=2pt] (3.08,-3.08)-- (4.36,-3.92);
\draw [line width=2pt] (4.36,-3.92)-- (5.5,-2.92);
\draw [line width=2pt] (5.5,-2.92)-- (7.52,-3.24);
\draw [line width=2pt] (7.52,-3.24)-- (6.06,-4.32);
\draw [line width=2pt] (6.06,-4.32)-- (5.5,-2.92);
\draw [line width=2pt] (7.52,-3.24)-- (5.86,-1.88);
\draw [line width=2pt] (5.86,-1.88)-- (3.76,-1.76);
\draw [line width=2pt] (4.36,-3.92)-- (4.06,-4.94);
\draw [line width=2pt] (4.06,-4.94)-- (6.06,-4.32);
\draw [line width=2pt] (2.16,-4.18)-- (2.64,-5.12);
\draw [line width=2pt] (2.64,-5.12)-- (4.06,-4.94);
\draw [line width=0.75pt] (2.64,-5.12)-- (1.36,-6.34);
\draw [line width=0.75pt] (1.36,-3.12)-- (-1.52,-3.6);
\draw [line width=0.75pt] (2.16,-1.92)-- (1.08,0.16);
\draw [line width=0.75pt] (5.86,-1.88)-- (6.4,0.54);
\draw [line width=0.75pt, opacity = 0.3] (6.4,0.54)-- (1.08,0.16);
\draw [line width=0.75pt, opacity = 0.3] (-1.52,-3.6)-- (1.08,0.16);
\draw [line width=0.75pt, opacity = 0.3] (6.4,0.54)-- (7.46,1.96);
\draw [line width=0.75pt, opacity = 0.3] (-1.52,-3.6)-- (-2.94,-4.6);
\draw [line width=0.75pt, opacity = 0.3] (1.36,-6.34)-- (1.42,-7.56);
\draw [line width=0.75pt, opacity = 0.3] (1.36,-6.34)-- (-0.1,-7);
\begin{scriptsize}
\draw [fill=black] (1.36,-3.12) circle (3pt);
\draw [fill=black] (7.52,-3.24) circle (3pt);
\draw [fill=black] (2.16,-4.18) circle (3pt);
\draw [fill=black] (2.16,-1.92) circle (3pt);
\draw [fill=black] (3.76,-1.76) circle (3pt);
\draw [fill=black] (3.08,-3.08) circle (3pt);
\draw [fill=black] (4.36,-3.92) circle (3pt);
\draw [fill=black] (5.5,-2.92) circle (3pt);
\draw [fill=black] (6.06,-4.32) circle (3pt);
\draw [fill=black] (5.86,-1.88) circle (3pt);
\draw [fill=black] (4.06,-4.94) circle (3pt);
\draw [fill=black] (2.64,-5.12) circle (3pt);
\draw [fill=black] (1.36,-6.34) circle (3pt);
\draw [fill=black] (-1.52,-3.6) circle (3pt);
\draw [fill=black] (1.08,0.16) circle (3pt);
\draw [fill=black] (6.4,0.54) circle (3pt);
\draw[color=black] (3.8,0.6) node {\large $e_1$};
\draw[color=black] (1.8,-6.4) node {\large $v_4$};
\draw[color=black] (7,-4.8) node {\Large $C_1$};
\draw[color=black] (6,-6.3) node {\Large $\overline{C_1}$};
\draw[color=black] (-1.4,-4) node {\large $v_3$};
\draw[color=black] (1,0.5) node {\large $v_2$};
\draw[color=black] (6.8,0.4) node {\large $v_1$};
\draw[color=black] (0,-1.8) node {\large $e_2$};
\end{scriptsize}
\end{tikzpicture}
\caption{The figure describes the situation after the first phase. The thick black edges are the ones in $C_1$, the thin black edges connect $C_1$ to explored vertices, which did not form cherries during the first phase and therefore are only present in $\overline{C_1}$, but not in $C_1$. The grey edges are the ones explored during the second phase. The edges $e_1$ and $e_2$ and the vertices $v_1, v_2, v_3$ are added to the component after the first phase, since all of them participate in chains of length 3. For the remaining grey edges, we learn during the second phase that they are matched to unexplored vertices, and we therefore leave them outside $C_2$. The vertex $v_4$ is also left outside $C_2$ since it does not participate in a chain of length 3.}
\label{fig:phase 2}
\end{figure}
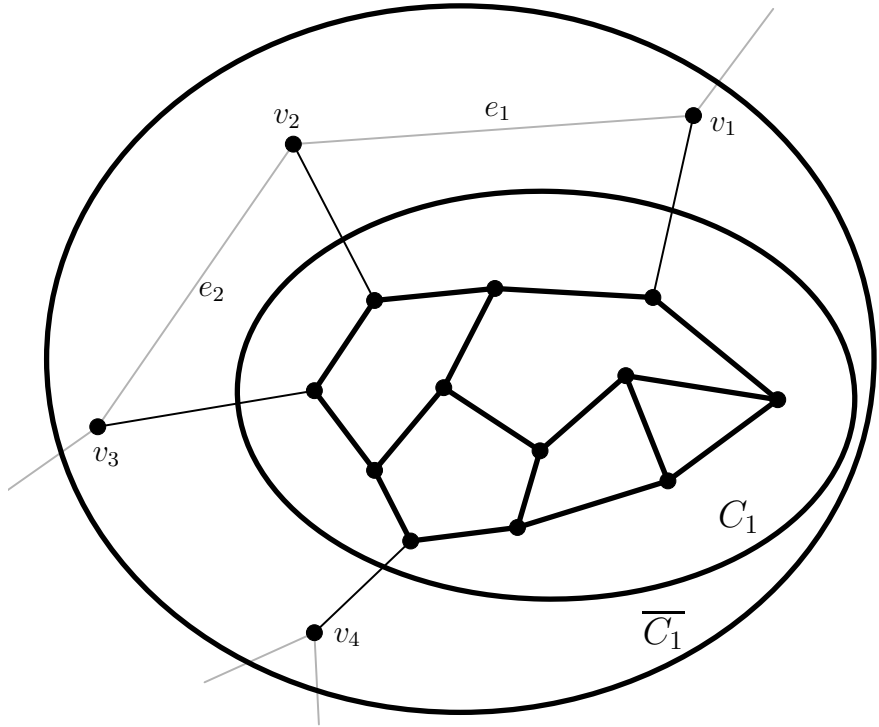

We have $Z_0(t) + Z_1(t) = t$. Moreover, at time $t$ (before scaling), one has $3X^1_0(t_1) - Z_0(t)$ open half-edges in vertices, unexplored after phase 1 - that is, in $G_3\setminus \overline{C_1}$. On the other hand, the total amount of open half-edges up to time $t$ is $3n - 2t_0n - 2t_1n - 2t - 1$. Thus clearly, for every half-edge matched to some vertex in $G_3\setminus \overline{C_1}$, the number of open half-edges attached to unexplored vertices in $G_3\setminus \overline{C_1}$ decreases by 1. We deduce that

\begin{equation*}
    \mathbb E[Z_0(t+1)\hspace{0.2em}|\hspace{0.2em}\mathcal F_t] = Z_0(t) + \dfrac{3X^1_0(t_1n) - Z_0(t)}{3n - 2t_0n - 2t_1n - 2t - 1}.
\end{equation*}

We also have $Z_0(0)=0$. As before, we rescale the time parameter $t$ and transform the difference equation into a differential equation by setting $z_0(t)=Z_0(\lfloor tn \rfloor)/n$.  We stop when $Z_0(t) + 2Z_1(t) = 2X^1_1(t_1n)$ or equivalently $Z_0(t)=2t-2X^1_1(t_1n)$: at this moment we know that all $2X^1_1(t_1)$ open half-edges in the vertices of $\overline{C_1}\setminus C_1$ have been processed during phase 2.

In other words, we will be looking for the smallest positive solution $t_2 = t_2(\varepsilon)$ of the corresponding equation for the rescaled time parameter
\begin{equation}\label{condition1}
    z_0(t) = 2t - 2x_1(t_1),
\end{equation}

where $z_0$ is given by the solution of the differential equation

\begin{equation*}
    z'_0(t) = \dfrac{3x_0(t_1) - z_0(t)}{3 - 2t_0 - 2t_1 - 2t} \text{ with } z_0(0) = 0.
\end{equation*}

One easily verifies that it is given by

\begin{equation*}
    z_0(t) = 3x_0(t_1)\left(1 - \sqrt{1 - \dfrac{2t}{3 - 2t_0 - 2t_1}}\right).
\end{equation*}

Solving~\eqref{condition1} for the rescaled time parameter boils down to solving the quadratic equation
\begin{equation*}
4t^2 + \left(\dfrac{18x^2_0(t_1)}{Q} - 4P\right)t + (P^2-9x^2_0(t_1)) = 0,
\end{equation*}
where 
\begin{equation*}
    P = P(\varepsilon) := 3x_0(t_1(\varepsilon)) + 2x_1(t_1(\varepsilon)) \text{ and }Q = Q(\varepsilon) := 3 - 2t_0(\varepsilon) - 2t_1(\varepsilon).
\end{equation*}

The solutions are given by
\begin{equation*}
    t_{\pm}(\varepsilon) = \dfrac{P}{2} -  \dfrac{9x^2_0(t_1) \mp \sqrt{81x^4_0(t_1) - 36x^2_0(t_1)PQ + 36x^2_0(t_1)Q^2}}{4Q}.
\end{equation*}

By Vieta's formulas one has $t_-(\varepsilon) t_+(\varepsilon) = \dfrac{P^2-9x^2_0(t_1)}{4} > 0$, so the two roots have the same sign, and $t_-(\varepsilon) + t_+(\varepsilon) = P - \dfrac{9x^2_0(t_1)}{2Q} > 0$ under the assumption $\varepsilon < 7/8$ (this can be checked by elementary algebraic transformations). Thus, in our setting both roots are positive and thus one has
\begin{equation*}
t_2 = t_2(\varepsilon) := t_-(\varepsilon).
\end{equation*}

In order to continue with the analysis, we first state and prove the following lemma:

\begin{lemma}\label{aux}
Conditionally on having $e$ edges on the set of $nx_1(t_1)$ vertices of $\overline{C_1}\setminus C_1$, the distribution of the graph on this set of vertices and $e$ edges is uniform among the graphs of degree at most 2 and $e$ edges.
\end{lemma}
\begin{proof}
This follows from the fact that the matching of the open half-edges after the first phase is uniform (i.e., a configuration model). Indeed, any conditioning on a uniform matching distribution leads to a uniform distribution on the set of configurations, which satisfy the restrictions, imposed by the conditioning. In our case, the restriction is that the number of edges between vertices of $\overline{C_1}\setminus C_1$ is fixed. 
\end{proof}

Thus, in order to calculate how many vertices of $\overline{C_1}\setminus C_1$ participate in chains of length 3, we apply Lemma~\ref{urns} with $a = x_1(t_1)$, where the urns are the vertices with 2 open half-edges, and $b = z_1(t_2)$, where the balls are the half-edges matched to vertices in $\overline{C_1}\setminus C_1$, which participate in edges between two vertices in $\overline{C_1}\setminus C_1$. Indeed, Lemma~\ref{aux} justifies that, conditionally on the number of edges in the graph induced by $\overline{C_1}\setminus C_1$, this graph may be constructed by attaching half-edges one by one uniformly at random so that no vertex is attached to more than 2 new half-edges, and matching them according to the configuration model. We deduce that the number of vertices participating in chains of length 3 is a.a.s.\ $\dfrac{z_1(t_2)(2x_1(t_1)-z_1(t_2))}{x_1(t_1)}n +o(n)$. We add then the vertices and edges participating in chains of length 3 to the component, thus adding a.a.s.\
\begin{equation}\label{see}
    \dfrac{z_1(t_2)(2x_1(t_1)-z_1(t_2))}{x_1(t_1)}n+o(n)
\end{equation}
vertices and
\begin{equation*}
    \dfrac{z_1(t_2)(2x_1(t_1)-z_1(t_2))}{x_1(t_1)}n + z_1(t_2)n+o(n)
\end{equation*}
edges to the component $C_1$. This produces a component $C_2$ with a.a.s.\
\begin{equation*}
    \left(\varepsilon + x_2(t_1)+\dfrac{z_1(t_2)(2x_1(t_1)-z_1(t_2))}{x_1(t_1)}\right)n+o(n)
\end{equation*}
vertices and 
\begin{equation*}
    \left(t_0+a(t_1)+2x_2(t_1)+\dfrac{z_1(t_2)(2x_1(t_1)-z_1(t_2))}{x_1(t_1)} + z_1(t_2)\right)n+o(n)
\end{equation*}
edges. We condition on this information in the sequel. This finishes the second phase.\\

\underline{Phase 3.}\\

The third phase will count the cherries that have their first and third vertex in $\overline{C_1}\setminus C_2$ and center in $G_3\setminus \overline{C_1}$. Of course, some vertices from $G_3\setminus \overline{C_1}$ can be connected also to the vertices in $C_2\setminus C_1$ that we added during the second phase to $C_1$, but this would only increase the modularity. Since our lower bound will not be sharp, we allow ourselves a bit of a tolerance in this third phase for the sake of a less technical analysis.
Our analysis goes as follows: we know that a.a.s.\ there are $z_0(t_2)n + o(n)$ edges between the vertices of $\overline{C_1} \setminus C_2$ and $G_3 \setminus \overline{C_1}$. First, choose the half-edges in the vertices of $G_3 \setminus \overline{C_1}$ that participate in the above edges uniformly at random. Then, match them uniformly at random to the half-edges sticking out of the vertices in $\overline{C_1}\setminus C_2$. This two-step procedure will allow us to learn the number of vertices in $G_3\setminus \overline{C_1}$ to be added to $C_2$ at the first step and the number of vertices in $\overline{C_1}\setminus C_2$ to be addded to $C_2$ at the second step.\par

In the beginning, we have $\left(x_1(t_1) - \dfrac{z_1(t_2)(2x_1(t_1)-z_1(t_2))}{x_1(t_1)}\right)n + o(n)$ vertices in $\overline{C_1}\setminus C_2$ and twice as many edges between $G_3\setminus \overline{C_1}$ and $\overline{C_1}\setminus C_2$ (recall that we condition on success of the previous stages). Start attaching these edges to vertices in $G_3\setminus \overline{C_1}$. Let $W_i(t)$ be the random variable counting the number of vertices in $G_3 \setminus \overline{C_1}$ of degree $i$ at time $t$ for $i = 0,1,2,3$. We have the following initial condition (directly after scaling, as in previous phases): 
\begin{equation*}
w_0(0) = x_0(t_1), w_1(0) = w_2(0) = w_3(0) = 0.
\end{equation*}
We have
\begin{align*}
& \mathbb E[W_0(t+1)\hspace{0.2em}|\hspace{0.2em}\mathcal F_t] = W_0(t) - \dfrac{3W_0(t)}{3x_0(t_1)n - t}, \\
& \mathbb E[W_1(t+1)\hspace{0.2em}|\hspace{0.2em}\mathcal F_t] = W_1(t) + \dfrac{3W_0(t)}{3x_0(t_1)n - t} - \dfrac{2W_1(t)}{3x_0(t_1)n - t}, \\
& \mathbb E[W_2(t+1)\hspace{0.2em}|\hspace{0.2em}\mathcal F_t] = W_2(t) + \dfrac{2W_1(t)}{3x_0(t_1)n - t} - \dfrac{W_2(t)}{3x_0(t_1)n - t}, \\
& \mathbb E[W_3(t+1)\hspace{0.2em}|\hspace{0.2em}\mathcal F_t] = W_3(t) + \dfrac{W_2(t)}{3x_0(t_1)n - t}.
\end{align*}

Transforming these into differential equations (as before) gives 
\begin{align*}
& w'_0(t) =  - \dfrac{3w_0(t)}{3x_0(t_1) - t} \text{ with }w_0(0) = x_0(t_1), \\
& w'_1(t) = \dfrac{3w_0(t)}{3x_0(t_1) - t} - \dfrac{2w_1(t)}{3x_0(t_1) - t}\text{ with }w_1(0) = 0, \\
& w'_2(t) = \dfrac{2w_1(t)}{3x_0(t_1) - t} - \dfrac{w_2(t)}{3x_0(t_1) - t}\text{ with }w_2(0) = 0, \\
& w'_3(t) = \dfrac{w_2(t)}{3x_0(t_1) - t}\text{ with }w_3(0) = 0.
\end{align*}

Solving these differential equations we obtain
\begin{align*}
    & w_0(t)=\dfrac{(3x_0(t_1)-t)^3}{27x_0(t_1)^2}, \\
    & w_1(t) = \dfrac{ t(t-3x_0(t_1))^2}{9x_0(t_1)^2}, \\
    & w_2(t)= \dfrac{t^2(3x_0(t_1)-t)}{9x_0(t_1)^2}, \\
    & w_3(t)=\dfrac{t^3}{27x_0(t_1)^2}.
\end{align*}

The time $t_3 = t_3(\varepsilon)$, at which the process stops, is the available number of edges between $G_3\setminus \overline{C_1}$ and $\overline{C_1}\setminus C_2$, more precisely

\begin{equation*}
    t_3 := 2\left(x_1(t_1) - \dfrac{z_1(t_2)(2x_1(t_1)-z_1(t_2))}{x_1(t_1)}\right)n + o(n).
\end{equation*}

Now we observe that, once the half-edges are attached to the vertices of $G_3\setminus \overline{C_1}$, we can match them uniformly at random to the $2\left(x_1(t_1) - \dfrac{z_1(t_2)(2x_1(t_1)-z_1(t_2))}{x_1(t_1)}\right)n + o(n)$ open half-edges, sticking out of the vertices in $\overline{C_1}\setminus C_2$. Indeed, since we consider a restriction of the configuration model, this additional matching is done uniformly at random and, once again, can be analyzed via the differential equation method (we omit the justification). We use Lemma~\ref{urns} with $a = \left(x_1(t_1) - \dfrac{z_1(t_2)(2x_1(t_1)-z_1(t_2))}{x_1(t_1)}\right)$, which is the number of vertices in $\overline{C_1}\setminus C_2$, and $2b = 2w_2(t_3)+3w_3(t_3)$, which is the number of half-edges in vertices of $G_3\setminus \overline{C_1}$, having at least two edges to $\overline{C_1}\setminus C_2$, to conclude that there are a.a.s.\
\begin{equation*}
    \dfrac{\left(w_2(t_3)+\dfrac{3w_3(t_3)}{2}\right)\left(2\left(x_1(t_1) - \dfrac{z_1(t_2)(2x_1(t_1)-z_1(t_2))}{x_1(t_1)}\right)-w_2(t_3)-\dfrac{3w_3(t_3)}{2}\right)}{\left(x_1(t_1) - \dfrac{z_1(t_2)(2x_1(t_1)-z_1(t_2))}{x_1(t_1)}\right)}n + o(n)
\end{equation*}
vertices in $\overline{C_1}\setminus C_2$ to be added to $C_2$ after phase 3 to form the component $C_3 = C_3(\varepsilon)$.\par

Finally, the total number of vertices in the component $C_3$ after the third phase is a.a.s.\
\begin{align*}
    v_3(\varepsilon) = & \left(\varepsilon + x_2(t_1)+\dfrac{z_1(t_2)(2x_1(t_1)-z_1(t_2))}{x_1(t_1)}\right)n +     w_2(t_3)n + w_3(t_3)n +\\ & \dfrac{\left(w_2(t_3)+\dfrac{3w_3(t_3)}{2}\right)\left(2\left(x_1(t_1) - \dfrac{z_1(t_2)(2x_1(t_1)-z_1(t_2))}{x_1(t_1)}\right)-w_2(t_3)-\dfrac{3w_3(t_3)}{2}\right)}{\left(x_1(t_1) - \dfrac{z_1(t_2)(2x_1(t_1)-z_1(t_2))}{x_1(t_1)}\right)}n +o(n).
\end{align*}
The total number of edges induced by the vertices of $C_3$ is a.a.s.\ also at least
\begin{align*}
        e_3(\varepsilon) = \left(t_0+a(t_1)+2x_2(t_1)+\dfrac{z_1(t_2)(2x_1(t_1)-z_1(t_2))}{x_1(t_1)} + z_1(t_2)\right)n + 2w_2(t_3)n+3w_3(t_3)n +\\
        \dfrac{\left(w_2(t_3)+\dfrac{3w_3(t_3)}{2}\right)\left(2\left(x_1(t_1) - \dfrac{z_1(t_2)(2x_1(t_1)-z_1(t_2))}{x_1(t_1)}\right)-w_2(t_3)-\dfrac{3w_3(t_3)}{2}\right)}{\left(x_1(t_1) - \dfrac{z_1(t_2)(2x_1(t_1)-z_1(t_2))}{x_1(t_1)}\right)}n + o(n).
\end{align*}

We now calculate the relative modularity of the component $C_3$ after the third phase - it is given by 
\begin{equation*}
q_r(C_3)=\frac{2e_3}{3v_3}-v_3.
\end{equation*}

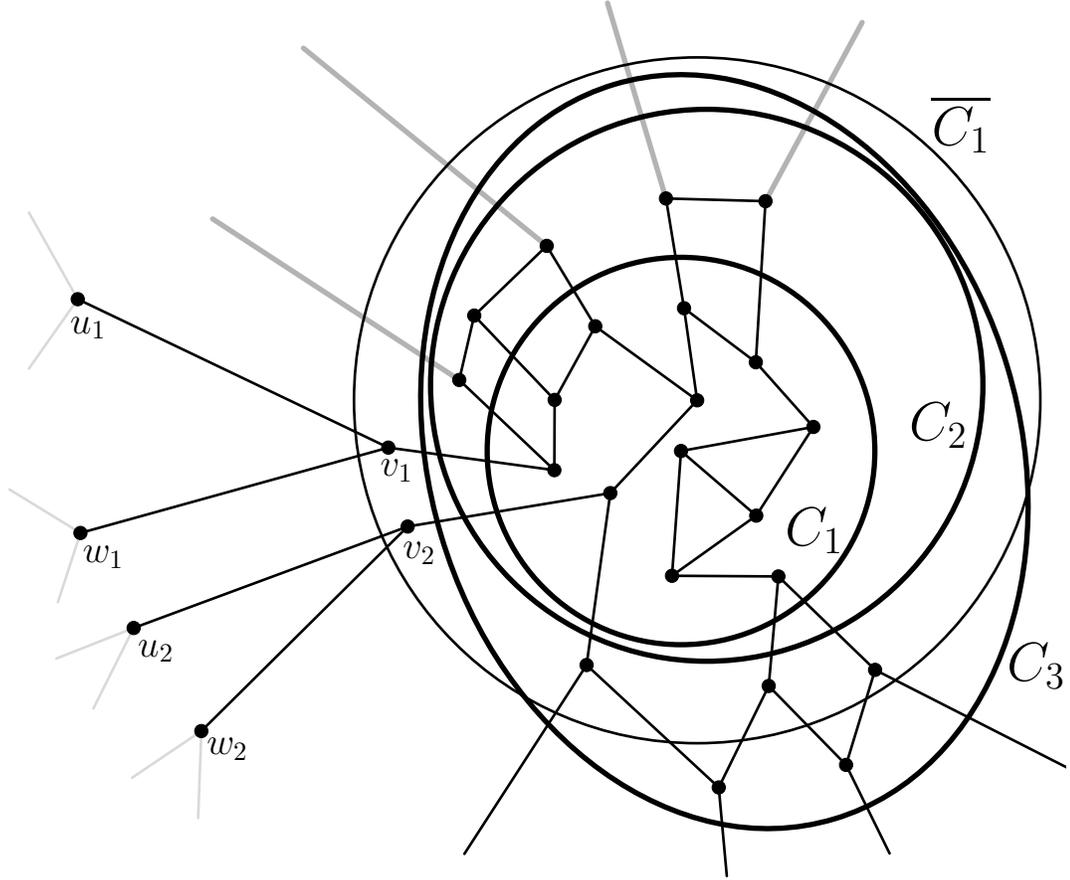
\begin{figure}
\centering
\begin{tikzpicture}[line cap=round,line join=round,x=1cm,y=1cm]
\clip(-3,-8.5) rectangle (12,4.4);
\draw [line width=2pt] (6.88,-2.34) circle (2.576897359228729cm);
\draw [line width=1pt] (7.094545454545452,-1.6654545454545453) circle (4.562246800400935cm);
\draw [line width=1pt] (3.9316079430398907,-1.3947362896577427)-- (4.13090909090909,-0.5381818181818181);
\draw [line width=1pt] (4.13090909090909,-0.5381818181818181)-- (5.094545454545453,0.38909090909090915);
\draw [line width=1pt] (5.094545454545453,0.38909090909090915)-- (5.74,-0.68);
\draw [line width=1pt] (5.74,-0.68)-- (7.094545454545452,-1.6654545454545453);
\draw [line width=1pt] (3.9316079430398907,-1.3947362896577427)-- (5.195862914300283,-2.5947071098370973);
\draw [line width=1pt] (4.13090909090909,-0.5381818181818181)-- (5.2,-1.66);
\draw [line width=1pt] (5.2,-1.66)-- (5.74,-0.68);
\draw [line width=1pt] (6.92,-0.44)-- (6.68,1.02);
\draw [line width=1pt] (6.68,1.02)-- (8.002937511505559,0.9837773003406233);
\draw [line width=1pt] (8.002937511505559,0.9837773003406233)-- (7.874369209343485,-1.15902773569394);
\draw [line width=1pt] (6.92,-0.44)-- (7.094545454545452,-1.6654545454545453);
\draw [line width=1pt] (6.92,-0.44)-- (7.874369209343485,-1.15902773569394);
\draw [line width=1pt] (5.2,-1.66)-- (5.195862914300283,-2.5947071098370973);
\draw [line width=1pt] (6.88,-2.34)-- (7.88,-3.2);
\draw [line width=1pt] (7.88,-3.2)-- (6.76,-4);
\draw [line width=1pt] (6.76,-4)-- (6.88,-2.34);
\draw [line width=2pt, opacity=0.3] (5.094545454545453,0.38909090909090915)-- (1.86,3.02);
\draw [line width=2pt, opacity=0.3] (6.68,1.02)-- (5.902988576191688,3.619427494663136);
\draw [line width=2pt, opacity=0.3] (8.002937511505559,0.9837773003406233)-- (9.288620533126297,3.3622908903389885);
\draw [line width=1pt] (5.94,-2.9)-- (7.094545454545452,-1.6654545454545453);
\draw [line width=1pt] (6.88,-2.34)-- (8.64,-2.02);
\draw [line width=1pt] (8.64,-2.02)-- (7.874369209343485,-1.15902773569394);
\draw [line width=1pt] (8.64,-2.02)-- (7.88,-3.2);
\draw [line width=1pt] (8.174361914388324,-4.008958433619909)-- (6.76,-4);
\draw [line width=1pt] (8.174361914388324,-4.008958433619909)-- (8.04579361222625,-5.466065858123412);
\draw [line width=1pt] (8.174361914388324,-4.008958433619909)-- (9.46004493600906,-5.2517853545199555);
\draw [line width=1pt] (5.94,-2.9)-- (5.624423921507194,-5.187501203438919);
\draw [line width=1pt] (5.94,-2.9)-- (3.2459103315088305,-3.3446888724491943);
\draw [line width=1pt] (5.195862914300283,-2.5947071098370973)-- (2.9887737271846833,-2.2947144047922583);
\draw [line width=2pt] (7.22181818181818,-1.4654545454545453) circle (3.6715569692587353cm);
\draw [line width=1pt] (2.9887737271846833,-2.2947144047922583)-- (-1.14,-0.32);
\draw [line width=1pt] (8.04579361222625,-5.466065858123412)-- (7.3815240510555356,-6.816033030825188);
\draw [line width=1pt] (7.3815240510555356,-6.816033030825188)-- (5.624423921507194,-5.187501203438919);
\draw [line width=1pt] (9.46004493600906,-5.2517853545199555)-- (9.07434002952284,-6.516040325780349);
\draw [line width=1pt] (9.07434002952284,-6.516040325780349)-- (8.04579361222625,-5.466065858123412);
\draw [line width=1pt] (9.46004493600906,-5.2517853545199555)-- (12.867104943304016,-6.987457433707952);
\draw [line width=1pt] (2.9887737271846833,-2.2947144047922583)-- (-1.1039838916413305,-3.430401073890577);
\draw [line width=1pt] (-0.3968582297499251,-4.694656045150969)-- (3.2459103315088305,-3.3446888724491943);
\draw [line width=1pt] (3.2459103315088305,-3.3446888724491943)-- (0.5031198853845911,-6.066051268213091);
\draw [rotate around={-73.35496178532401:(7.456522227316747,-2.348284530693123)},line width=2pt] (7.456522227316747,-2.348284530693123) ellipse (5.100603432953521cm and 3.9276897694090076cm);
\draw [line width=2pt, opacity=0.3] (3.9316079430398907,-1.3947362896577427)-- (0.6531162379070105,0.7480687463768206);
\draw [line width=1pt] (9.07434002952284,-6.516040325780349)-- (9.652897389252171,-7.694583095599358);
\draw [line width=1pt] (7.3815240510555356,-6.816033030825188)-- (7.488664302857264,-7.994575800644197);
\draw [line width=1pt, opacity=0.15] (-1.14,-0.32)-- (-1.7896815031723905,0.8337809478182032);
\draw [line width=1pt, opacity=0.15] (-1.14,-0.32)-- (-1.7896815031723905,-1.2447399371353232);
\draw [line width=1pt, opacity=0.15] (-1.1039838916413305,-3.430401073890577)-- (-2.046818107496538,-2.8518437141612454);
\draw [line width=1pt, opacity=0.15] (-1.1039838916413305,-3.430401073890577)-- (-1.4039765966861693,-4.3518072393854395);
\draw [line width=1pt, opacity=0.15] (-0.3968582297499251,-4.694656045150969)-- (-1.425404647046515,-5.101789001997537);
\draw [line width=1pt] (5.624423921507194,-5.187501203438919)-- (4,-7.7);
\draw [line width=1pt, opacity=0.15] (-0.3968582297499251,-4.694656045150969)-- (-0.9325594887585656,-5.766058563168252);
\draw [line width=1pt, opacity=0.15] (0.5031198853845911,-6.066051268213091)-- (-0.4182862801102707,-6.687464728663114);
\draw [line width=1pt, opacity=0.15] (0.5031198853845911,-6.066051268213091)-- (0.46026378466389983,-7.223165987671755);
\begin{scriptsize}
\draw[color=black] (8.65,-3.4) node {\huge $C_1$};
\draw[color=black] (10.3,-2) node {\huge $C_2$};
\draw[color=black] (11.6,-5.2) node {\huge $C_3$};
\draw[color=black] (10.6,2) node {\huge $\overline{C_1}$};
\draw [fill=black] (6.88,-2.34) circle (2.5pt);
\draw [fill=black] (7.094545454545452,-1.6654545454545453) circle (2.5pt);
\draw [fill=black] (3.9316079430398907,-1.3947362896577427) circle (2.5pt);
\draw [fill=black] (4.13090909090909,-0.5381818181818181) circle (2.5pt);
\draw [fill=black] (5.094545454545453,0.38909090909090915) circle (2.5pt);
\draw [fill=black] (5.74,-0.68) circle (2.5pt);
\draw [fill=black] (5.195862914300283,-2.5947071098370973) circle (2.5pt);
\draw [fill=black] (5.2,-1.66) circle (2.5pt);
\draw [fill=black] (6.92,-0.44) circle (2.5pt);
\draw [fill=black] (6.68,1.02) circle (2.5pt);
\draw [fill=black] (8.002937511505559,0.9837773003406233) circle (2.5pt);
\draw [fill=black] (7.874369209343485,-1.15902773569394) circle (2.5pt);
\draw [fill=black] (7.88,-3.2) circle (2.5pt);
\draw [fill=black] (6.76,-4) circle (2.5pt);
\draw [fill=black] (5.94,-2.9) circle (2.5pt);
\draw [fill=black] (8.64,-2.02) circle (2.5pt);
\draw [fill=black] (8.174361914388324,-4.008958433619909) circle (2.5pt);
\draw [fill=black] (8.04579361222625,-5.466065858123412) circle (2.5pt);
\draw [fill=black] (9.46004493600906,-5.2517853545199555) circle (2.5pt);
\draw [fill=black] (5.624423921507194,-5.187501203438919) circle (2.5pt);
\draw [fill=black] (3.2459103315088305,-3.3446888724491943) circle (2.5pt);
\draw [fill=black] (2.9887737271846833,-2.2947144047922583) circle (2.5pt);
\draw [fill=black] (7.3815240510555356,-6.816033030825188) circle (2.5pt);
\draw [fill=black] (9.07434002952284,-6.516040325780349) circle (2.5pt);
\draw [fill=black] (-1.1039838916413305,-3.430401073890577) circle (2.5pt);
\draw [fill=black] (-0.3968582297499251,-4.694656045150969) circle (2.5pt);
\draw [fill=black] (-1.14,-0.32) circle (2.5pt);
\draw [fill=black] (0.5031198853845911,-6.066051268213091) circle (2.5pt);
\draw[color=black] (0.85,-6.3) node {\Large $w_2$};
\draw[color=black] (-0.1,-5) node {\Large $u_2$};
\draw[color=black] (3.4,-3.7) node {\Large $v_2$};
\draw[color=black] (-0.8,-3.75) node {\Large $w_1$};
\draw[color=black] (-1,-0.7) node {\Large $u_1$};
\draw[color=black] (3.1,-2.6) node {\Large $v_1$};
\end{scriptsize}
\end{tikzpicture}
\caption{The first, the second and the third phase in one figure. The exposed edges after the third phase are colored in black. After the third phase, do a step-by-step exposure of the thick grey edges - these are the unexposed edges attached to vertices in $C_2\setminus C_1$ - and repeat the procedure of adding cherries to $C_3$ performed in phase 1. After that, contract the remaining paths in $G_3\setminus \overline{C_3}$ of length 2 of exposed edges left with centers in $\overline{C_1}\setminus C_3$ as, for example, $(u_1, v_1, w_1)$ and $(u_2, v_2, w_2)$ in the figure. Then, the vertices outside $\overline{C_3}$ after the contractions have degrees 2, 3 and 4, and the modified graph after the contractions follows a configuration model on the given degree sequence (note that all explored edges have been contracted).}
\label{fig:final phase}
\end{figure}

Optimizing $q_r(C_3) = q_r(C_3(\varepsilon))$ over $\varepsilon\in [0,7/8]$ gives $\varepsilon = 0.037562$. For this choice of $\varepsilon$ one obtains $q_r(C_3)= 0.674701 > \frac{2}{3}$. It would now be sufficient to prove that the graph induced by $V \setminus C_3$ has relative modularity at least $2/3$. Rather than doing this directly, however, we add some more vertices to $C_3$ first.\par

First, let us analyze under which conditions the operation of adding a cherry (recall that this is a path of length 2 with first and last vertex in a component and center outside of the component) to a component $C$ increases the relative modularity of $C$. Indeed, one needs that
\begin{equation*}
\dfrac{2e(C)}{3|C|}-\dfrac{|C|}{n} \le \dfrac{2(e(C)+2)}{3(|C|+1)}-\dfrac{|C|+1}{n} \iff n(3|C| + (9|C|-6e(C))) \ge 3|C|(3|C| + 3).
\end{equation*}
Since we are in the setting of $3-$regular graphs, $9|C|\ge 6e(C)$ and therefore, for a cherry to increase the modularity of a component $C$ when (its center is) added to $C$ it is sufficient that $|C| \le n/3 - 1$.\par

We now construct a component $\overline{C_3}$ by adding any cherries with first and last vertex in $C_3$ and center outside $C_3$. For this, one may explore the half-edges attached to vertices in $C_2\setminus C_1$, which were unexplored until the end of the third phase. The procedure of adding cherries was explained in the analysis of phase 1. We just underline that one path of length 2 may be such that one may not add it directly to $C_3$, but after a couple of other cherries have been already added to $C_3$ it may become a cherry itself. Since one may explicitly calculate the size of $C_3$ immediately after the third phase, which is $|C_3| = |C_3(0.037562)| = 0.044783 n$, even when we add all available cherries we will not increase the size of $C_3$ more than four times since at each step $e(C_3, V\setminus C_3)$ decreases by 1 and the maximal degree in $G_3$ is 3. Thus, $|C_3|$ remains smaller than $n/3 - 1$ throughout the whole process. By doing another analysis via the differential equation method one could get an explicit value for $\overline{C_3}$ and thus get an improved lower bound, but in the sequel we only use that $|\overline{C_3}|\ge|C_3|$.\par

Let us now analyze how the complement of $\overline{C_3}$ looks like after all cherries with respect to $C_3$ have been consecutively added to $C_3$. It contains only vertices of degree 2 and 3. However, some edges have been exposed by now - these are the edges incident to the vertices in $\overline{C_1}\setminus C_3$. Let $v$ be one such vertex and $u,w$ be its neighbors outside $C_3$. There are two cases (see Figure~\ref{fig:final phase}):
\begin{enumerate}
    \item $u$ (or $w$) became a cherry at one point with respect to $C_3$ in the construction of $\overline{C_3}$. Then, one adds $u$ to $C_3$ and $v$ becomes a cherry with respect to $C_3$. Then, one adds $v$.
    \item\label{contractions} Neither $u$ nor $w$ become cherries. Then, at least one edge incident to $u$ and at least one edge incident to $w$ go to vertices of $G_3\setminus \{\overline{C_3} \cup v\}$. Therefore, contracting the edges $uv$ and $vw$ produces a vertex of degree at least 2 in $G_3\setminus \overline{C_3}$.
\end{enumerate}
Performing the contractions given in point~\ref{contractions} above, one is left with a configuration model for the complement of $\overline{C_3}$ of minimal degree 2, where a positive proportion of the vertices have degree at least 3. Indeed, one has that $|\overline{C_3}(0.037562)|\le 4\times 0.044784 n = 0.179136 n$ and $|\overline{C_1}(0.037562)\setminus C_2(0.037562)|\le 3|C_3(0.037562)|\le 0.134353 n$. One deduces that, first, the number of contracted edges is at most $2 \times 0.134353 n$, and second, the number of vertices of degree 2 in $G_3\setminus \overline{C_3}(0.037562)$ is at most $3|\overline{C_3}(0.037562)|\le 0.537412 n$. One deduces that the number of vertices of degree 3 in $G_3\setminus \overline{C_3}$ is at least $n - 3|\overline{C_3}(0.037562)| -  |\overline{C_3}(0.037562)| = 0.283456 n$: we take out the vertices in $\overline{C_3}(0.037562)$ as well as the vertices that might participate in edge contractions - these are vertices at distance at most 2 from $\overline{C_3}(0.037562)$.\par

Now, a direct application of Lemma~\ref{connectivity} for the complement of $\overline{C_3}$ after the contractions of the explored edges in point~\ref{contractions} above gives that a.a.s.\ at most $\log n$ vertices of $G_3\setminus \overline{C_3}$ are outside the giant component $C'_{\max}$ in $G_3\setminus \overline{C_3}$. Since contractions do not modify connectivity, before contractions one should have that $G_3\setminus \overline{C_3}$ must contain a giant component $C_{\max}$ and all but at most $3\log n$ vertices must be in it a.a.s.\ (indeed, every vertex participates in at most 2 contracted edges).

By Lemma~\ref{lem:partition} applied to a spanning tree of $C_{\max}$, one may divide $C_{\max}$ into connected parts $(A_1, A_2, \dots, A_k)$ of orders between $\lfloor \sqrt{n}\rfloor$ and $\Delta \lceil \sqrt{n}\rceil$. Thus, setting 
\begin{equation*}
\mathcal A=\bigg\{ A_0 := \overline{C_3}, A_1, A_2, \dots, A_k, A_{k+1} := V\setminus \{\overline{C_3}\cup C_{\max}\}\bigg\},
\end{equation*}
and observing that every part $A_i$ satisfies $q_r(A_i)\ge \frac23+o(1)$ (which holds for a tree of maximum degree at most 3) we obtain
\begin{align*}
q^*(G) &\ge\hspace{0.25em} q(\mathcal A) \ge \hspace{0.25em} q_r(\overline{C_3})\frac{|\overline{C_3}|}{n}+\sum_{1\le i\le k}q_r(A_i) \frac{|A_i|}{n}\\
&\ge\hspace{0.25em} q_r(C_3)\frac{|\overline{C_3}|}{n}+\dfrac{2}{3}\frac{(1-|\overline{C_3}|)}{n}\\
&\ge\hspace{0.25em} 0.667026 - o(1),
\end{align*}
where we used that $q_r(\overline{C_3}) \ge q_r(C_3)$ and that $|\overline{C_3}|\ge|C_3|\ge 0.044783n$. Once again we point out that a slightly improved lower bound could be obtained by calculating an explicit value of $|\overline{C_3}|$ but we left this out for the sake of simplicity. The proof of Theorem~\ref{LB3reg} is completed.

\section{\texorpdfstring{Upper bound for random $3-$regular graphs}{}}\label{sec:Upper}

In this section we show that a.a.s.\ the modularity of a random 3-regular graph $G_3 = G_3(n)$ is at most $0.789998$.

Suppose therefore that in $G_{n, 3}$ there is a partition $\mathcal A = (A_1, A_2, \dots, A_k)$ with modularity at least $0.789998$. More formally,
\begin{equation*}
    \sum_{i=1}^{k} \dfrac{|A_i|}{n} q_r(A_i) \ge 0.789998,
\end{equation*}
and therefore there is $i\in [k]$, for which $q_r(A_i) \ge 0.789998$. By definition of $A_i$, this means that 
\begin{equation}\label{eq:mod}
    \dfrac{2e(A_i)}{3|A_i|} - \dfrac{|A_i|}{n} \ge 0.789998.
\end{equation}

We will show that a.a.s.\ such a set does not exist in $G_3(n)$. Fix the set $A\subseteq V$ with the largest relative modularity in $G_3$ among all subsets of $V$ and suppose that its relative modularity is at least $0.789998$. We start with a simple observation:

\begin{observation}\label{ob 4.0}
$G_3[A]$ is a connected graph.
\end{observation}
\begin{proof}
Let $S\subseteq V(G_3)$ be a union of disjoint non-empty sets $S', S''$ such that $e(S', S'') = 0$. Since $e(S'\cup S'') = e(S')+e(S'')$,
\begin{equation*}
    q_r(S) = \dfrac{2e(S)}{3|S|} - \dfrac{|S|}{n} < \dfrac{|S'|}{|S|}\left(\dfrac{2e(S')}{3|S'|} - \dfrac{|S'|}{n}\right) + \dfrac{|S''|}{|S|}\left(\dfrac{2e(S'')}{3|S''|} - \dfrac{|S''|}{n}\right)\le \max\{q_r(S'), q_r(S'')\}.
\end{equation*}
Since $A$ is the vertex set with maximal relative modularity, $G_3[A]$ must be a connected graph.

\end{proof}

\begin{lemma}\label{lem 4.1}
A.a.s. there is $\varepsilon_0 > 0$, for which $|A|\ge \varepsilon_0 n$.
\end{lemma}
\begin{proof}
Suppose the contrary and fix $\varepsilon_0 > 0$ to be chosen later. Let $|A|=s\le \varepsilon_0 n$. Note that in order for $A_i = A$ to satisfy~\eqref{eq:mod} we must have
\begin{equation*}
e(A) \ge \dfrac{3\times 0.789998 s}{2} = 1.184997 s,   
\end{equation*}
or equivalently in terms of the density, we must have 
\begin{equation*}
\dfrac{e(A)}{s} \ge 1.184997. 
\end{equation*}

First, for any constant $C > 0$ we cannot have $s\le C$, since by~\cite{JLR}, Proof of Theorem~9.5, a.a.s.\ there is no subgraph of fixed size with more edges than vertices, so we may assume in the sequel that $s$ is larger than any fixed positive constant (in fact $s \ge 11$ suffices). Next, we may assume that the graph induced by $A$ contains a spanning tree of maximal degree $3$ on $s$ vertices, and there are at least $0.184997 s$ edges added inside $A$ on top of the spanning tree. Since choosing such an unlabeled tree on $s$ vertices can be done in at most $\dfrac{C^s}{s^{3/2}}$ ways with $C \approx 2.483253$, see \cite{Ency1} and \cite{Ency2}. Now we bound from above the probability that there is a set $A$ of size $s$ such that $q_r(A) \ge 0.789998$ using a union bound over all subsets of $V$ of size $s$:

\begin{itemize}
    \item There are $\binom{n}{s}$ subsets of size $s$ of $V$.
    \item One chooses an unlabeled spanning tree on the given $s$ vertices of maximal degree at most 3 in $\dfrac{C^s}{s^{3/2}}$ ways.
    \item One chooses the labels of the $s$ vertices in $s!$ ways.
    \item One chooses $s'\in [0.184997 s, s]$ vertices, which will be incident to the additional $m\ge 0.184997 s$, edges in the component, i.e., the ones that do not participate in the spanning tree.
    \item For every vertex, one multiplies by a factor of 2 to choose if one or two half-edges (if present) will participate in the additional edges outside the spanning tree.
    \item Then, one multiplies by $(2m-1)!!$  to choose the matching within the chosen edges.
    \item Then, one multiplies by $(3n-2(s-1)-2m-1)!!$ for the matching of all other half-edges.
    \item Then we divide by the probability $\dfrac{1}{(3n-1)!!}$ that a particular graph appears.
\end{itemize}

In total, we obtain 

\begin{align*}
& \binom{n}{s} \dfrac{C^s}{s^{3/2}} s! \binom{s}{s'} 2^{s'} (2m-1)!! (3n - 2(s-1) - 2m-1)!! \dfrac{1}{(3n-1)!!} \\
\le\hspace{0.35em}
& \dfrac{(3C)^s}{\sqrt{s}} \binom{n}{s} \dfrac{(s-1)!(2m-1)!!(3n-2(s-1)-2m - 1)!!}{(3n-1)!!} \\
\le\hspace{0.35em}
& \dfrac{(3C)^s}{s^{3/2}} \prod_{0\le i\le s-1} \dfrac{n-i}{3n-2i-1} \prod_{s\le i\le s+m-2} \dfrac{2m+2s-2i-1}{3n-2i-1} \\
\le\hspace{0.35em}
& \dfrac{(3C)^s}{s^{3/2}} \dfrac{2}{3^s} \prod_{0\le i\le m-2} \dfrac{2m-2i-1}{3n-2s-2i-1}\\ 
\le\hspace{0.35em}
& \dfrac{2C^s}{s^{3/2}} \left(\dfrac{2m-1}{3n-2s-1}\right)^{m-1}\\ 
\le\hspace{0.35em}
& \dfrac{2C^s}{s^{3/2}} \left(\dfrac{s+1}{3n-2s-1}\right)^{0.184997 s - 1} \\
\le\hspace{0.35em}
& 2 C^{\frac{1}{0.184997}} \left(\dfrac{C^{\frac{1}{0.184997}}(s+1)}{3n-2s-1}\right)^{0.184997 s - 1}.
\end{align*}

For every large enough $n$, summing the above upper bound over the interval $s\in [11, \log n]$ gives an upper bound of 
\begin{equation*}
\dfrac{24 C^{11} \log n}{(3n-23)^{1.03497}},
\end{equation*}
and summing over the interval $s\in [\log n, \varepsilon_0 n]$ with $\varepsilon_0$ such that 
\begin{equation*}
2C^{\frac{1}{0.184997}} \varepsilon_0 = 3 - 2\varepsilon_0
\end{equation*}
gives an upper bound of 
\begin{equation*}
2C^{\frac{1}{0.184997}} \sum_{\log n \le i\le \varepsilon_0 n} \dfrac{1}{2^{0.184997i}} \le \dfrac{2 C^{\frac{1}{0.184997}}}{(1 - \frac{1}{2^{0.184997}}) 2^{0.184997 \log n}}.
\end{equation*}

Summing both bounds gives that the probability of having a subset of $V(G_3)$ of at most $\varepsilon_0 n$ vertices, where $\varepsilon_0 > 0$ was given above, inducing a subgraph of $G_3$ of relative modularity at least $0.789998$ tends to 0 with $n$. The lemma is proved.
\end{proof}

Due to Lemma~\ref{lem 4.1} we assume from now on that $|A|\ge \varepsilon_0 n$. Let $|A| = \varepsilon n$ for some $\varepsilon \ge \varepsilon_0 > 0$. Since $q_r(A)\ge 0.789998$, we conclude that 
\begin{equation*}
e(A) \ge \dfrac{3\varepsilon n}{2}(0.789998+ \varepsilon) \ge \dfrac{3\varepsilon n}{2}\left(\frac{2}{3}+0.123331+ \varepsilon\right) =(1 + 3(0.123331+ \varepsilon)/2)\varepsilon n,
\end{equation*}
i.e., the density of $A$ is at least $(1 + 3(0.123331 + \varepsilon)/2)$.\par

Now, for $k\in [\varepsilon_0 n, n/2]$, let $B_k$ be the subset of $V$ of size $k$ inducing a connected graph $G_3[B_k]$ with maximal number of edges. By assumption there is $k\in [\varepsilon_0 n, n/2]$ such that the given density is at least
\begin{equation*}
1 + \dfrac{3}{2}\left(\dfrac{k}{n} + 0.123331\right).
\end{equation*}

We prove that if such a set $B_k$ exists for some $k$ in the given range, then one may find a set $B$ such that 
\begin{enumerate}
\item\label{cond 1} $G_3[B]$ and $G_3[V\setminus B]$ both contain only vertices of degree 2 and 3,
\item\label{cond 2} $\varepsilon_0 n\le |B|\le n/2$, and
\item\label{cond 3}
\begin{equation*}
e(G_3[B]) \ge \left(1 + \dfrac{3}{2}\left(0.123331 + \dfrac{|B|}{n}\right)\right)|B|.
\end{equation*}
\end{enumerate}

For the sake of contradiction, assume that no set $B$ exists. Let $\overline{B}$ be a set that satisfies conditions~\ref{cond 2} and~\ref{cond 3} above (such a set exists by assumption), for which the quantity 
\begin{equation*}
e(G_3[\overline{B}]) - \left(1 + \dfrac{3}{2}\left(0.123331 + \dfrac{|B|}{n}\right)\right)|\overline{B}|
\end{equation*}
is maximal.\par

\begin{lemma}\label{lem 4.2}
For every $n\ge 10$, the graph $G_3[V\setminus \overline{B}]$ contains only vertices of degree 2 and 3.
\end{lemma}
\begin{proof}
We argue by contradiction. Suppose that there are edges $uv,vw$ in $G_3$ such that $u,w\in \overline{B}$ and $v\in V\setminus\overline{ B}$. Then, adding $v$ to $\overline{B}$ produces a graph with $e(G_3[\overline{B}])+2$ edges. On the one hand, one has that
\begin{align*}
& \left(1 + \dfrac{3}{2}\left(0.123331 + \dfrac{|\overline{B}|+1}{n}\right)\right)(|\overline{B}|+1)\\ 
=\hspace{0.25em} & \left(1 + \dfrac{3}{2}\left(0.123331 + \dfrac{|\overline{B}|}{n}\right)\right)|\overline{B}| + 1 + \dfrac{3}{2}\left(0.123331 + \dfrac{|\overline{B}|}{n}\right) + \dfrac{3}{2n}|\overline{B}| + \dfrac{3}{2n}\\
=\hspace{0.25em} & \left(1 + \dfrac{3}{2}\left(0.123331 + \dfrac{|\overline{B}|}{n}\right)\right)|\overline{B}| + 1.1849965 + \dfrac{3|\overline{B}|+3/2}{n}.
\end{align*}

On the other hand, since the density of a component cannot become larger than $3/2$, we have that
\begin{equation*}
1 + \dfrac{3}{2}\left(0.123331 + \dfrac{|\overline{B}|}{n}\right) \le \dfrac{3}{2} \iff |\overline{B}| \le 0.210002 n.
\end{equation*}

Since this is the case, one has that for every $n\ge 10$
\begin{align*}
& e(G_3[\overline{B}\cup v]) - \left(1 + \dfrac{3}{2}\left(0.123331 + \dfrac{|\overline{B}\cup v|}{n}\right)\right)(|\overline{B}\cup v|)\\ 
=\hspace{0.25em} & e(G_3[\overline{B}]) + 2 - \left(1 + \dfrac{3}{2}\left(0.123331 + \dfrac{|\overline{B}|}{n}\right)\right)|\overline{B}| - 1.1849965 - \dfrac{3|\overline{B}|+3/2}{n}\\
>\hspace{0.25em} & e(G_3[\overline{B}]) - \left(1 + \dfrac{3}{2}\left(0.123331 + \dfrac{|\overline{B}|}{n}\right)\right)|\overline{B}|.
\end{align*}
This is a contradiction with the choice of $\overline{B}$. The lemma is proved.
\end{proof}

\begin{lemma}\label{lem 4.3}
For every $n\ge 10$, the graph $G_3[\overline{B}]$ contains only vertices of degree 2 and 3.
\end{lemma}
\begin{proof}
We argue by contradiction. Suppose that $v$ is a vertex of degree 1 in $G_3[\overline{B}]$. Then, one has $e(G_3[\overline{B}\setminus v]) = e(G_3[\overline{B}]) - 1$ and
\begin{align*}
& \left(1 + \dfrac{3}{2}\left(0.123331 + \dfrac{|\overline{B}|-1}{n}\right)\right)(|\overline{B}|-1)\\
= & \left(1 + \dfrac{3}{2}\left(0.123331 + \dfrac{|\overline{B}|}{n}\right)\right)|\overline{B}| - \left(1 + \dfrac{3}{2}\left(0.123331 + \dfrac{|\overline{B}|}{n}\right)\right) - \dfrac{3|\overline{B}|}{2n} + \dfrac{3}{2n}.
\end{align*}
One immediately deduces that
\begin{equation*}
e(G_3[\overline{B}\setminus v]) - \left(1 + \dfrac{3}{2}\left(0.123331 + \dfrac{|\overline{B}\setminus v|}{n}\right)\right)(|\overline{B}\setminus v|) \hspace{0.3em}>\hspace{0.3em} e(G_3[\overline{B}]) - \left(1 + \dfrac{3}{2}\left(0.123331 + \dfrac{|\overline{B}|}{n}\right)\right)|\overline{B}|.
\end{equation*}
This is a contradiction with the choice of $\overline{B}$. The lemma is proved.
\end{proof}

\begin{corollary}\label{cor 4.4}
For every $n\ge 10$, the set $\overline{B}$ satisfies~\ref{cond 1}. 
\end{corollary}
\begin{proof}
This follows directly from Lemma~\ref{lem 4.2} and Lemma~\ref{lem 4.3}.
\end{proof}

\begin{proof}[Proof of Theorem~\ref{UB3reg}]
We apply the first moment method to count the number of sets $B$ with $|B|=\varepsilon n$ for some $\varepsilon\in [\varepsilon_0, 1/2]$ satisfying conditions~\ref{cond 1} and~\ref{cond 3} from above. We proceed as follows:

\begin{itemize}
\item First we choose $\varepsilon n$ vertices out of $n$ that belong to $B$ in $\binom{n}{\varepsilon n}$ ways.
\item We choose a number of $\beta n$ vertices in $B$ in $\binom{\varepsilon n}{\beta n}$ ways and $\beta n$ vertices in $V\setminus B$ in $\binom{(1-\varepsilon)n}{\beta n}$ ways. These vertices will be endvertices of the edges between $B$ and $V\setminus B$ and therefore $\beta \le (1 - 3(0.123331 + \varepsilon))\varepsilon$, which follows directly from condition~\ref{cond 3}.
\item For each of the vertices chosen above, choose a half-edge that will participate in the edge between $B$ and $V\setminus B$ in three ways, and match the given half-edges in $(\beta n)!$ ways.
\item Choose a graph on the vertices of $B$ with $\beta n$ vertices of degree 2 (the ones chosen above) and $(\varepsilon -\beta)n$ vertices of degree 3. By Corollary~\ref{cor 2.3} with $m = (3\varepsilon -\beta)n/2$ there are 
\begin{equation*}
    \Theta\left(\dfrac{((3\varepsilon -\beta)n)!}{2^{(3\varepsilon -\beta)n/2} ((3\varepsilon -\beta)n/2)!} \dfrac{1}{2^{\beta n}}\dfrac{1}{6^{(\varepsilon - \beta) n}}\right)
\end{equation*}
choices. Multiply by a factor of $2^{\beta n}6^{(\varepsilon - \beta)n}$ to count configurations with labeled half-edges rather than graphs.
\item Then, choose a graph over the vertices of $V\setminus B$ with $\beta n$ vertices of degree 2 (the ones chosen above) and $(1 - \varepsilon -\beta)n$ vertices of degree 3. By Corollary~\ref{cor 2.3} with $m = (3(1-\varepsilon) -\beta)n/2$ there are 
\begin{equation*}
    \Theta\left(\dfrac{((3(1-\varepsilon) - \beta)n)!}{2^{(3(1-\varepsilon) -\beta)n/2} ((3(1-\varepsilon) -\beta)n/2)!} \dfrac{1}{2^{\beta n}}\dfrac{1}{6^{(1-\varepsilon - \beta) n}}\right)
\end{equation*}
choices. Multiply by a factor of $2^{\beta n}6^{(1-\varepsilon - \beta)n}$ to count configurations with labeled half-edges rather than graphs.
\item Divide by the total number $(3n-1)!!$ of configurations to reduce the counting above to an expectation.
\end{itemize}

Multiplying all factors leads to the following formula, which gives the order of the expectation of the number of cuts $(B, V\setminus B)$ in $G_3(n)$:
\begin{equation*}
\binom{n}{\varepsilon n} \binom{\varepsilon n}{\beta n}
\binom{(1-\varepsilon)n}{\beta n} 3^{2\beta n} (\beta n)!
\dfrac{((3\varepsilon -\beta)n)!}{2^{(3\varepsilon -\beta)n/2} ((3\varepsilon -\beta)n/2)!}
\dfrac{((3(1-\varepsilon) -\beta)n)!}{2^{(3(1-\varepsilon) -\beta)n/2} ((3(1-\varepsilon) -\beta)n/2)!}
\dfrac{1}{(3n-1)!!}.
\end{equation*}

Applying Stirling's formula given by $k! \underset{k\to +\infty}{\sim} (k/e)^k \sqrt{2\pi k}$ to all factorials and taking the $n$-th root while ignoring factors of subexponential order leads us to
\begin{equation*}
\dfrac{3^{2\beta} \beta^{\beta}}{\beta^{2\beta}(\varepsilon - \beta)^{\varepsilon - \beta}(1 - \varepsilon - \beta)^{1-\varepsilon-\beta}} 
\dfrac{(3\varepsilon - \beta)^{3\varepsilon - \beta}}{2^{(3\varepsilon - \beta)/2}((3\varepsilon - \beta)/2)^{(3\varepsilon - \beta)/2}} 
\dfrac{(3(1-\varepsilon) - \beta)^{3(1-\varepsilon) - \beta}}{2^{(3(1-\varepsilon) - \beta)/2}((3(1-\varepsilon) - \beta)/2)^{(3(1-\varepsilon) - \beta)/2}}
\dfrac{1}{3^{3/2}}.
\end{equation*}

Simplifying further we get
\begin{equation*}
\frac{3^{2\beta}(3\varepsilon-\beta)^{\frac{3\varepsilon-\beta}{2}}(3-3\varepsilon-\beta)^{\frac{3-3\varepsilon-\beta}{2}}}{3^{3/2}(\varepsilon-\beta)^{\varepsilon-\beta}\beta^{\beta}(1-\varepsilon-\beta)^{1-\varepsilon-\beta}}.  
\end{equation*}

Taking logarithms we obtain the following function:
\begin{align}
f(\beta, \varepsilon)&:= 2\beta \log 3 + \frac12(3\varepsilon-\beta)\log(3\varepsilon-\beta)+\frac12(3-3\varepsilon-\beta)\log(3-3\varepsilon-\beta)-\beta \log \beta \nonumber \\
&-(\varepsilon-\beta)\log(\varepsilon-\beta)-(1-\varepsilon-\beta)\log(1-\varepsilon-\beta)-\frac32\log 3, \label{eq:function}
\end{align}
Recall that by assumption $\varepsilon_0 \le  \varepsilon \le \frac{1}{2}$ and $0 \le \beta \le (1-3(0.123331+\varepsilon))\varepsilon$.
Taking the derivative with respect to $\beta$, we obtain
\begin{equation*}
\dfrac{\partial f}{\partial \beta} (\beta, \varepsilon) = 2 \log 3-\frac12 \log(3\varepsilon-\beta)-\frac12 \log(3-3\varepsilon-\beta)-\log \beta+\log(\varepsilon-\beta)+\log(1-\varepsilon-\beta).  
\end{equation*}

We show that for every $\varepsilon$, $\dfrac{\partial f}{\partial \beta}$ is non-negative for every $\beta\in [0, (1-3(0.123331+\varepsilon))\varepsilon]$ and thus $f(\beta, \varepsilon)$ is maximized for $\beta = (1-3(0.123331+\varepsilon))\varepsilon$. Since $\beta < \varepsilon$, this is equivalent to
\begin{equation}\label{eq:derivative}
81(\varepsilon-\beta)^2(1-\varepsilon-\beta)^2 \ge (3\varepsilon - \beta)(3-3\varepsilon-\beta)\beta^2,
\end{equation}
which can be rewritten as
\begin{equation*}
81(\varepsilon-\beta)(1-\varepsilon-\beta) - \left(3 + \dfrac{2\beta}{\varepsilon-\beta}\right)\left(3+\dfrac{2\beta}{1-\varepsilon-\beta}\right)\beta^2 \ge 0.
\end{equation*}
For fixed $\varepsilon > 0$ this a decreasing function of $\beta\in [0, \varepsilon]$. Replacing $\beta = \beta(\varepsilon) = (1-3(0.123331+\varepsilon))\varepsilon$, using standard analysis techniques one verifies the positivity of the above expression for every $\varepsilon\in [0, 1/2]$ - indeed, it is increasing as a function of $\varepsilon$ for $\varepsilon \in [0, +\infty]$ and equal to 0 at $\varepsilon = 0$.\par

Define $g(\varepsilon):= f(\beta(\varepsilon), \varepsilon)$. It remains to verify that $g(\varepsilon) < 0$ for every $\varepsilon\in [\varepsilon_0, 1/2]$. One readily verifies that the derivative of $g$ is negative on the interval $[0, 0.005221]$, positive on the interval $[0.005221, 0.026271]$ and then again negative on $[0.026271, 0.5]$. Since $g(0) = 0$, we have that
\begin{equation*}
\max_{\varepsilon_0\le \varepsilon \le 0.5} g(\varepsilon) = \max(g(\varepsilon_0), g(0.026271)) \le \max(g(\varepsilon_0), -0.891947\times 10^{-5}) < 0.
\end{equation*}

Therefore, summing over all possible $|B|\in [\varepsilon_0 n, n/2]$ and over all possible sizes $e(B, V\setminus B)\in [0, (1-3(0.123331+\varepsilon))\varepsilon]$, the expected number of sets $B$ satisfying conditions~\ref{cond 1},~\ref{cond 2} and~\ref{cond 3} is smaller than $c^n$ for every constant $c\in (\exp(\max_{\varepsilon_0\le \varepsilon \le 0.5} g(\varepsilon)), 1)$ and for every large enough $n$. By Markov's inequality, a.a.s.\ there is no set $B$ satisfying conditions~\ref{cond 1},~\ref{cond 2} and~\ref{cond 3}. This is a contradiction with our assumption that there exists a set $A\subseteq V$ with $q_r(A)\ge 0.789998$. In particular, this shows that that the modularity of $G_3(n)$ is a.a.s.\ less than $0.789998$. Theorem~\ref{UB3reg} is proved.
\end{proof}

\section{An improved lower bound for more general degree sequences}\label{sec:General}
Lemma~\ref{lem:simplelb} gives a simple lower bound on the modularity of any deterministic graph, which in the case of $3-$regular graphs yields an a.a.s.\ lower bound of $2/3$. On the other hand, Section~\ref{sec:Lower} shows that the modularity of $G_3(n)$ is strictly larger. In the current section we prove Theorem~\ref{LBgeneral}, and in particular we almost \footnote{As Theorem~\ref{LBgeneral} suggests, the case $Q = 0$ is excluded from our analysis.} characterize the set of sequences of bounded regular degree sequences for which the bound given by Lemma~\ref{lem:simplelb} may be improved (thus also characterizing those for which the bound given by Lemma
~\ref{lem:simplelb} is sharp).\par

Let $G$ be a graph on $n\ge 6$ vertices and $m$ edges. Assume that $G$ contains no isolated vertices and its maximal degree is at most $\Delta$. Let also $\mathcal A = (A_1, A_2, \dots, A_k)$ be any partition of $V(G)$ with $q(\mathcal A) = q^*(G)$.

\begin{lemma}\label{lem 5.1}
$G[A_i]$ is a connected graph for every $i\in [k]$.
\end{lemma}
\begin{proof}
We argue by contradiction. Suppose that for some $i\in [k]$ one has that $G[A_i]$ is a union of two non-empty graphs $G_1$ and $G_2$ with $e(V(G_1), V(G_2)) = 0$. Then, we have 
\begin{equation*}
e(G_1) + e(G_2) = e(A_i) \text{ and } \vol(V(G_1))^2 + \vol(V(G_2))^2 < \vol(A_i)^2.
\end{equation*}
Thus, dividing $A_i$ into $V(G_1)$ and $V(G_2)$ increases the modularity of the partition $\mathcal A$, which was assumed to be maximal - contradiction. The lemma is proved.
\end{proof}

\begin{lemma}\label{lem 5.2}
For every two different parts $A_i, A_j$ of $\mathcal A$ we have that 
\begin{equation}
\emph{\vol}(A_i) \emph{\vol}(A_j) \ge 2m\cdot e(A_i, A_j).
\end{equation}
\end{lemma}
\begin{proof}
Let $\mathcal A'$ be a partition of $V(G)$, obtained from $\mathcal A$ by replacing the parts $A_i$ and $A_j$ by their union $A_i\cup A_j$. Then
\begin{equation*}
q(\mathcal A) - q(\mathcal A') = \dfrac{2\vol(A_i)\vol(A_j)}{4m^2} - \dfrac{e(A_i, A_j)}{m}.
\end{equation*}
By maximality of $q(\mathcal A)$ over all partitions of $V(G)$, the above expression must be non-negative, so
\begin{equation*}
\vol(A_i)\vol(A_j) \ge 2m\cdot e(A_i, A_j).
\end{equation*}
The lemma is proved.
\end{proof}

\begin{corollary}\label{cor 5.3}
For every connected component $C$ of $G$ of size less than $\frac12\sqrt{n}$, $V(C)$ participates as a part in $\mathcal A$.
\end{corollary}
\begin{proof}
We argue by contradiction. Suppose that for every $i\in [k]$, $V(C)$ is different from $A_i$. By Lemma~\ref{lem 5.1}, there are parts $A_{i_1}, A_{i_2},\dots, A_{i_r}$ of $\mathcal A$, for which 
\begin{equation*}
    V(C) = \bigcup_{1\le s\le r} A_{i_s}.
\end{equation*}
On the one hand, since $C$ is a connected graph, there are two parts among $A_{i_1}, A_{i_2},\dots, A_{i_r}$, say $A_{i_1}$ and $A_{i_2}$ without loss of generality, for which $e(A_{i_1}, A_{i_2})\ge 1$. On the other hand, since the graph $G$ has no isolated vertices, it contains at least $n/2$ edges. One concludes that
\begin{equation*}
    \vol(A_{i_1})\vol(A_{i_2})\le \vol(C)^2 < (\sqrt{n})^2 \le 2m\cdot e(A_{i_1}, A_{i_2}).
\end{equation*}
This is a contradiction with Lemma~\ref{lem 5.2}. The corollary follows.
\end{proof}

Before diving into the proof of Theorem~\ref{LBgeneral}, we state a criterion for the existence of a giant component due to Molloy and Reed, see \cite{MR}. Here we present a version for bounded degree sequences although the theorem itself is more general.

\begin{theorem}[\cite{MR}, Theorem 1]\label{MR criterion}
Under the conditions of Theorem~\ref{LBgeneral}:
\begin{itemize}
    \item If $Q < 0$, then there are constants $R_1 = R_1(\boldsymbol{p}, \Delta), R_2 = R_2(\boldsymbol{p}, \Delta) > 0$, such that all components of $G(n)$ have size at most $R_1\log n$, and the total number of cycles in $G(n)$ is at most $R_2\log n$ a.a.s.
    \item If $Q > 0$, then there are constants $\xi_1 = \xi_1(\boldsymbol{p}, \Delta), \xi_2 = \xi_2(\boldsymbol{p}, \Delta) > 0$, for which the largest component in $G(n)$ contains at least $\xi_1 n$ vertices and $\xi_2 n$ cycles a.a.s. Moreover, there is a positive constant $\gamma = \gamma(\boldsymbol{p}, \Delta) > 0$ such that the second largest component in $G(n)$ has size at most $\gamma \log n$ a.a.s.
\end{itemize}
\qed
\end{theorem}

In the sequel we assume that $p_i > 0$ for every $i\in [\Delta]$. This is a technical assumption: one may only work on the set of non-zero $p_i-$s and deduce the same results as the ones presented below. For every $i\in [\Delta]$, denote by $d_i(n)$ the number of vertices of degree $i$ in $G(n)$ and set $D(n) = \sum_{1\le i\le \Delta} id_i(n)$.

\subsection{Proof of point~\ref{thm 3.1} of Theorem~\ref{LBgeneral} - the subcritical regime}\label{subsec 5.1}

Under the assumptions of point~\ref{thm 3.1} of Theorem~\ref{LBgeneral}, by the criterion given by Theorem~\ref{MR criterion} we know that the largest component in $G(n)$ is a.a.s.\ of size at most $R_1\log n$ for some positive constant $R_1 = R_1(\boldsymbol{p}, \Delta) > 0$. Thus, by Lemma~\ref{lem 5.1} and Corollary~\ref{cor 5.3} a.a.s.\ the only partition of $V(G(n))$ with maximal modularity is the one given by the vertex sets of the connected components of $G(n)$. We denote it by $\mathcal A = (A_1, A_2, \dots, A_k)$.\par

By definition one has $\sum_{1\le i\le k} e(G(A_i)) = m$. Therefore, we get
\begin{equation*}
q^*(G(n)) = q(\mathcal A) = 1-\sum_{1\le i\le k} \dfrac{\vol(A_i)^2}{4m^2}.
\end{equation*}

For every $n\in \mathbb N$, denote by $\mathcal N_n(H)$ the random variable, equal to the number of isolated copies of the graph $H$ in $G(n)$.

\begin{lemma}\label{Number of copies}
For every tree $T$ of order $t\le \sqrt{\log n}$ and for every large enough $n$ we have
\begin{equation*}
\mathbb P(|\mathcal N_n(T) - \mathbb E[\mathcal N_n(T)]|\ge n^{2/3}) \le \dfrac{1}{n^{1/4}}.
\end{equation*}
\end{lemma}
\begin{proof}
We apply the second moment method. Let, for every $i\in [\Delta]$, $t_i$ be the number of vertices of degree $i$ in $T$, and let $Aut(T)$ be the automorphism group of $T$. On the one side, the expected number of copies of $T$ is

\begin{align*}
\mathbb E[\mathcal N_n(T)] 
&=\hspace{0.3em} \sum_{\substack{A\subseteq V;\\ |A| = t}} \mathbb P(G[A] = T \text{ and } e(A, V\setminus A)=0)\\ 
&=\hspace{0.3em} \dfrac{1}{|Aut(T)|}\left(\prod_{1\le i\le \Delta} \binom{d_i(n)}{t_i} t_i! i!^{t_i}\right) \dfrac{(D(n) - 1 - 2(t-1))!!}{(D(n) - 1)!!}.
\end{align*}

Indeed, one needs to:
\begin{itemize}
    \item for every $i\in [\Delta]$, choose the $t_i$ vertices of degree $i$ in $\binom{d_i(n)}{t_i}$ ways;
    \item decide on the position of each vertex of degree $i$ in $T$ in $t_i!$ ways;
    \item decide on the order of the half-edges, attached to every vertex of degree $i$ in $T$, in $i!$ ways (since we are counting configurations here, half-edges are labeled);
    \item and finally, divide by the probability of constructing all $t-1$ edges in $T$.
\end{itemize}

One also has that:
\begin{itemize}
    \item $|Aut(T)|\le t!$,
    \item for every $i\in [\Delta]$,
    \begin{equation*}
        \prod_{0\le j\le t_i-1} \dfrac{d_i(n) - j}{d_i(n)} = \exp\left(-(1+o(1))\sum_{0\le j\le t_i-1} \dfrac{j}{d_i(n)}\right) = 1-o(1),
    \end{equation*}
    \item 
    \begin{equation*}
        \prod_{0\le j\le t-2} \dfrac{D(n) - 1 - 2(j-1)}{D(n)} = \exp\left(-(1+o(1))\sum_{0\le j\le t-2} \dfrac{1 + 2(j-1)}{D(n)}\right) = 1-o(1).
    \end{equation*}
\end{itemize}

Therefore, since $t\le \sqrt{\log n}$, in the regime $n\to +\infty$ one may bound the expectation from below by  
\begin{equation*}
(1-o(1)) \dfrac{1}{t!} D(n) \prod_{1\le i\le \Delta} \left(\dfrac{i d_i(n)}{D(n)}\right)^{t_i} = n^{1-o(1)}.
\end{equation*}

The variance of $\mathcal N_n(T)$ is given by 
\begin{align*}
\mathbb Var(\mathcal N_n(T)) = &\hspace{0.2em} \mathbb E[\mathcal N_n(T)^2] - \mathbb E[\mathcal N_n(T)]^2\\ = 
&\hspace{0.2em} \mathbb E[\mathcal N_n(T)] +
\dfrac{1}{|Aut(T)|^2}\mathbb \prod_{1\le i\le \Delta} \binom{d_i(n)}{t_i}\binom{d_i(n)-t_i}{t_i} t_i!^2 i!^{2t_i} \dfrac{(D(n) - 1 - 2(2t-2))!!}{(D(n) - 1)!!} -\\ 
&\hspace{0.2em} \left(\dfrac{1}{|Aut(T)|} \prod_{1\le i\le \Delta} \binom{d_i(n)}{t_i} (t_i)! i!^{t_i} \dfrac{(D(n) - 1 - 2(t-1))!!}{(D(n) - 1)!!}\right)^2\\ = 
&\hspace{0.2em} \mathbb E[\mathcal N_n(T)] + \dfrac{1}{|Aut(T)|^2}\left(\prod_{1\le i\le \Delta} \binom{d_i(n)}{t_i}(t_i)! i!^{t_i} \dfrac{(D(n) - 1 - 2(t-1))!!}{(D(n) - 1)!!}\right)^2\times\\
&\hspace{0.2em} \left(\left(\prod_{1\le i\le \Delta}\prod_{0\le j\le t_i-1} \dfrac{d_i(n) - t_i - j}{d_i(n) - j}\right)\left(\prod_{0\le j\le t-2}\dfrac{D(n) - 1 - 2j}{D(n) - 1 - 2(t-1) - 2j}\right) - 1\right).
\end{align*}
Now, one may use the standard bound
\begin{equation*}
\prod_{1\le i\le t} (1 + x_i) \le \exp\bigg(\sum_{1\le i\le t}x_i\bigg)
\end{equation*}
to deduce that
\begin{equation}
\prod_{0\le j\le t_i-1} \dfrac{d_i(n) - t_i - j}{d_i(n) - j} \le \exp\left(-t_i\sum_{0\le j\le t_i-1} \dfrac{1}{d_i(n) - j}\right)\le \exp\left(-\dfrac{t_i^2}{d_i(n)}\right)
\end{equation}
and 
\begin{align*}
& \prod_{0\le j\le t-2}\dfrac{D(n) - 1 - 2j}{D(n) - 1 - 2(t-1) - 2j}\\ \le
& \exp\left(2(t-1)\sum_{0\le j\le t-2} \dfrac{1}{D(n)-1-2(t-1)-2j}\right)\\ \le 
& \exp\left(\dfrac{2t^2}{D(n)-4t}\right).
\end{align*}

We deduce that for every large enough $n$ 
\begin{align*}
& \mathbb Var(\mathcal N_n(T))\\
\le\hspace{0.25em}
& \mathbb E[\mathcal N_n(T)] + \mathbb E[\mathcal N_n(T)]^2 \left(\exp\left(-\sum_{1\le i\le \Delta} \dfrac{t_i^2}{d_i(n)}+\dfrac{2t^2}{D(n) - 4t}\right) - 1\right)\\ 
\le\hspace{0.25em}
&\mathbb E[\mathcal N_n(T)] + 2\mathbb E[\mathcal N_n(T)]^2 \left(\dfrac{2t^2}{D(n) - 4t}-\sum_{1\le i\le \Delta} \dfrac{t_i^2}{d_i(n)}\right) \le\hspace{0.25em} \mathbb E[\mathcal N_n(T)] + \dfrac{(4+o(1))\log n}{D(n)} \mathbb E[\mathcal N_n(T)]^2.
\end{align*}
We conclude by Chebyshev's inequality that
\begin{equation*}
\mathbb P(|\mathcal N_n(T) - \mathbb E[\mathcal N_n(T)]|\ge \alpha) \le \dfrac{\mathbb Var(\mathcal N_n(T))}{\alpha^2}.
\end{equation*}
Choosing for example $\alpha = \mathbb E[\mathcal N_n(T)]^{2/3}$ leads to
\begin{equation*}
\mathbb P(|\mathcal N_n(T) - \mathbb E[\mathcal N_n(T)]|\ge \mathbb E[\mathcal N_n(T)]^{2/3}) \le \dfrac{1}{\mathbb E[\mathcal N_n(T)]^{1/3}} + \dfrac{(4+o(1))(\log n)\mathbb E[\mathcal N_n(T)]^{2/3}}{D(n)} \le \dfrac{1}{n^{1/4}}.
\end{equation*}
Since $n\ge \mathbb E[\mathcal N_n(T)]$, we have that 
\begin{equation*}
    \mathbb P(|\mathcal N_n(T) - \mathbb E[\mathcal N_n(T)]|\ge n^{2/3})\le \mathbb P(|\mathcal N_n(T) - \mathbb E[\mathcal N_n(T)]|\ge \mathbb E[\mathcal N_n(T)]^{2/3})\le \dfrac{1}{n^{1/4}}.
\end{equation*}
The lemma is proved.
\end{proof}

The next observation is a well-known fact and, as such, it will not be proved here. In a nutshell, its proof relies on, first, the fact that the local limit of $G(n)$ under the conditions of point~\ref{thm 3.1} of Theorem~\ref{LBgeneral} is a subcritical Galton-Watson tree, and second, that the order of a subcritical Galton-Watson tree is a random variable with exponential upper tail. For more details on the topic, we refer the reader to~\cite{Curien} and~\cite{JustinThese}.

\begin{observation}\label{ob 5.6}
Under the conditions of point~\ref{thm 3.1} of Theorem~\ref{LBgeneral} there is a constant $c = c(\boldsymbol{p}, \Delta) > 0$ such that the connected component $C(v)$ of a uniformly chosen vertex $v\in G(n)$ is of order at least $t$ with probability at most $\exp(-ct)$.\qed
\end{observation}

Let $\mathcal{T}(t)$ be the set of trees with maximal degree $\Delta$ and order at most $t$.

\begin{lemma}\label{equiv}
A.a.s.
\begin{equation*}
\sum_{1\le i\le k} \emph{\vol}(A_i)^2 =(1+o(1))\sum_{\substack{i: G[A_i]\in \mathcal T\left(\sqrt{\log n}\right)}} \emph{\vol}(A_i)^2.
\end{equation*}
\end{lemma}
\begin{proof}
First, note that the expected number of trees that consist of a single edge is 
\begin{equation*}
    \dfrac{1}{2}\left(p_1 n\times \dfrac{p_1}{M}\right).
\end{equation*}
By Lemma~\ref{Number of copies}, the number of isolated edges in $G(n)$ is sharply concentrated. Hence, a.a.s. $k = k(n)$ is at least $\dfrac{p^2_1}{4M} n$ and
\begin{equation*}
\sum_{1\le i\le k} \vol(A_i)^2 \ge \dfrac{p^2_1 n}{M}.
\end{equation*}

By Theorem~\ref{MR criterion}, there are at most $R_2 \log n$ cycles in $G(n)$, each containing at most $R_1 \log n$ vertices a.a.s. In total, the components with cycles contribute to $\sum_{1\le i\le k} \vol(A_i)^2$ at most
\begin{equation*}
\Delta^2 R^2_1R_2\log^3 n = O\left(\log^3 n\right) = o(n)
\end{equation*}
a.a.s.

Also, by Observation~\ref{ob 5.6}, there is some positive constant $c > 0$ such that the number of trees of order at least $\sqrt{\log n}$ contains at most $\exp(-c\sqrt{\log n})$ vertices a.a.s. Thus, the connected components of order at least $\sqrt{\log n}$ contribute to $\sum_{1\le i\le k} \vol(A_i)^2$ at most
\begin{equation*}
n\exp\left(-c\sqrt{\log n}\right) (R_1\log n)^2 = o(n).
\end{equation*}
Putting together the three statements above proves the lemma.
\end{proof}

\begin{lemma}\label{lem 5.8}
A.a.s.
\begin{equation*}
\max_{T\in \mathcal{T}\left(\sqrt{\log n}\right)} |\mathcal N_n(T) - \mathbb E[\mathcal N_n(T)]| \le n^{2/3}.
\end{equation*}
\end{lemma}
\begin{proof}
By Lemma~\ref{Number of copies} we have that for a particular tree $T$ of order at most $\sqrt{\log n}$,
\begin{equation*}
\mathbb P(|\mathcal N_n(T) - \mathbb E[\mathcal N_n(T)]| \ge n^{2/3}) \le \dfrac{1}{n^{1/4}}.
\end{equation*}
On the other hand, it is well-known that there are $t^{t-2}$ trees of order $t$ (see for example~\cite{Prufer}), and therefore by a union bound
\begin{align*}
& \mathbb P\left(\max_{T\in \mathcal{T}\left(\sqrt{\log n}\right)} |\mathcal N(T) - \mathbb E[\mathcal N(T)]| \ge n^{2/3}\right)\\ 
\le\hspace{0.25em}
& \sum_{1\le k\le \sqrt{\log n}} k^{k-2} \dfrac{1}{n^{1/4}} \\ 
\le\hspace{0.25em}
& \sqrt{\log n}^{\sqrt{\log n}}\dfrac{1}{n^{1/4}} =\hspace{0.25em} o(1).
\end{align*}
The lemma is proved.
\end{proof}

\begin{corollary}\label{cor 5.9}
A.a.s.
\begin{equation*}
\left|\mathbb E\left[\dfrac{\sum_{G[A_i]\in \mathcal{T}\left(\sqrt{\log n}\right)} \emph{\vol}(A_i)^2}{4m^2}\right] - \dfrac{\sum_{G[A_i]\in \mathcal{T}\left(\sqrt{\log n}\right)} \emph{\vol}(A_i)^2}{4m^2}\right| = o\left(\dfrac{1}{n}\right).
\end{equation*}
\end{corollary}
\begin{proof}
We have that
\begin{align*}
& \left|\mathbb E\left[\dfrac{\sum_{G[A_i]\in \mathcal{T}\left(\sqrt{\log n}\right)} \vol(A_i)^2}{4m^2}\right] - \dfrac{\sum_{G[A_i]\in \mathcal{T}\left(\sqrt{\log n}\right)} \vol(A_i)^2}{4m^2}\right|\\
\le & \sum_{G[A_i]\in \mathcal{T}\left(\sqrt{\log n}\right)} \dfrac{1}{4m^2} |\mathbb E[\mathcal N(T)] - \mathcal N(T)|(2e(T))^2.
\end{align*}
By Lemma~\ref{lem 5.8}, we have that a.a.s.\ the last expression is at most 
\begin{equation*}
\sum_{G[A_i]\in \mathcal{T}\left(\sqrt{\log n}\right)} \dfrac{1}{4m^2} n^{2/3} (2e(T))^2 \le \dfrac{\left|\mathcal{T}\left(\sqrt{\log n}\right)\right| n^{2/3}\log n}{m^2}.
\end{equation*}
Since the number of trees on $t$ vertices is $t^{t-2}$, one has
\begin{equation*}
\dfrac{\left|\mathcal{T}\left(\sqrt{\log n}\right)\right| n^{2/3}\log n}{m^2} \le \dfrac{\sqrt{\log n}^{\sqrt{\log n}} n^{2/3}\log n}{m^2} = o\left(\dfrac{1}{n}\right).
\end{equation*}
The corollary is proved.
\end{proof}

\begin{proof}[Proof of point~\ref{thm 3.1} of Theorem~\ref{LBgeneral}]

By combining Lemma~\ref{equiv} and Corollary~\ref{cor 5.9} it is sufficient to estimate 
\begin{equation*}
\mathbb E\left[\dfrac{\sum_{G[A_i]\in \mathcal{T}\left(\sqrt{\log n}\right)} \vol(A_i)^2}{4m^2}\right].
\end{equation*}
We do this now.\par

\begin{itemize}
\item First, for every $i\in [\Delta]$, choose the vertices of degree $i$ in $T$ in $\binom{d_i(n)}{t_i}$ ways. 
\item A well-known result gives the number of trees on $t$ vertices with $t_i$ vertices of degree $i$ for every $i\in [\Delta]$, which is equal to (again, see for example \cite{Prufer}):
\begin{equation*}
\dfrac{(t-2)!}{(\Delta-1)!^{t_{\Delta}}(\Delta-2)!^{t_{\Delta-1}} \dots 2!^{t_3}}.
\end{equation*}

Observe that a sequence $(t_1, t_2, \dots, t_{\Delta})$ is a \textit{tree sequence} (that is, there is a tree with $t_i$ vertices of degree $i$ for every $i\in [\Delta]$) if and only if satisfies that the number of edges is one less than the number of vertices:
\begin{equation*}
\dfrac{t_1 + 2t_2 + \dots \Delta t_{\Delta}}{2} = t_1 + \dots + t_{\Delta} - 1,
\end{equation*}
which is equivalent to
\begin{equation}\label{tree constraint}
t_1 = 2 + \sum_{3\le i\le \Delta} (i - 2) t_i.
\end{equation}

\item Since we count configurations and not simply graphs, for every vertex we multiply by the product of the factorials of the degrees
\begin{equation*}
    \prod_{1\le i\le \Delta} i!^{t_i}
\end{equation*}
(this gives the number of permutations of the half-edges at each vertex). 
\item Finally, multiply by the probability that all $e(T)$ edges are present, which is given by
\begin{equation*}
\dfrac{1}{(D(n) - 1)(D(n) - 3)\dots (D(n) - 1 - 2(e(T)-1))}.
\end{equation*}
\end{itemize}

We get that the expected number of trees with $t_i$ vertices of degree $i$, for every $i\in [\Delta]$, is given by 

\begin{align}
&\label{line 11} \dfrac{\prod_{1\le i\le \Delta} \binom{d_i(n)}{t_i}\prod_{1\le i\le \Delta} i!^{t_i}}{(D(n) - 1)(D(n) - 3)\dots (D(n) - 1 - 2(t - 2))} \dfrac{(t - 2)!}{\prod_{1\le i\le \Delta} (i-1)!^{t_i}} \\ =\hspace{0.25em} 
& D(n) \prod_{0\le i\le t-2}\dfrac{1}{1 - \frac{1+2i}{D(n)}} \prod_{0\le i\le \Delta} \left(\prod_{0\le j\le t_i-1}\left(1 - \frac{j}{d_i(n)}\right)\right) \dfrac{1}{t(t - 1)} \binom{t}{t_1, t_2, \dots, t_{\Delta}} \prod_{1\le i\le \Delta} \left(\dfrac{id_i(n)}{D(n)}\right)^{t_i} \nonumber.
\end{align}

Standard analysis shows that
\begin{equation*}
    \prod_{0\le i\le t-2} \left(1-\dfrac{1+2i}{D(n)}\right) = 1 - \dfrac{(t-1)^2}{D(n)} + o\left(\dfrac{(t-1)^2}{D(n)}\right) = 1 - O\left(\dfrac{\log n}{n}\right),
\end{equation*}
and by assumption
\begin{equation*}
\prod_{0\le i\le \Delta} \left(\prod_{0\le j\le t_i-1}\left(1 - \frac{j}{d_i(n)}\right)\right) = (1+o(1))\prod_{1\le i\le \Delta} \left(1 - \dfrac{t_i(t_i-1)}{2d_i(n)}\right) = 1 - O\left(\dfrac{\log n}{n}\right).
\end{equation*}

Thus, for every tree sequence $(t_1, t_2, \dots, t_{\Delta})$, the above two products do not modify the first order in (\ref{line 11}). Since for every tree $T$ we are interested in the volume of $T$, we multiply each term by $\vol(T)^2 = (2e(T))^2 = 4(t - 1)^2$ and sum over all tree sequences to deduce the value of $\mathbb E\left[\sum_{1\le i\le k}\dfrac{ \vol^2(A_i)}{4m^2}\right]$ up to a $(1+o(1))-$factor:

\begin{align}
&\label{line 13} \sum_{1\le t\le \sqrt{\log n}}\sum_{\substack{t_1+t_2+\dots + t_{\Delta} = t;\\ (t_1, t_2, \dots, t_{\Delta})\\ \text{is a tree sequence}}}\dfrac{ \left(1+O\left(\dfrac{\log n}{n}\right)\right)\dfrac{4(t-1)}{t} \dbinom{t}{t_1, t_2, \dots, t_{\Delta}} \left(\prod_{1\le i\le \Delta} \left(\dfrac{id_i(n)}{D(n)}\right)^{t_i}\right) D(n)}{4\left(\dfrac{D(n)}{2}\right)^2}.
\end{align}

Notice that the factor 
\begin{equation*}
\binom{t}{t_1, t_2, \dots, t_{\Delta}} \left(\prod_{1\le i\le \Delta} \left(\dfrac{id_i(n)}{D(n)}\right)^{t_i}\right)
\end{equation*}
in (\ref{line 13}) can be interpreted as the probability that a die with outcomes $1,2,\dots, \Delta$ with respective probabilities $\dfrac{d_1(n)}{D(n)}, \dfrac{2d_2(n)}{D(n)}, \dots, \dfrac{\Delta d_{\Delta}(n)}{D(n)}$ gives $t_1$ ones, $t_2$ twos, etc. over $t$ independent trials. Under the above interpretation, denote by $X_i$ the number of trials with outcome $i$. Thus, by a Chernoff bound (Lemma~\ref{chernoff:bd}) we have for every $\delta\in (0,1)$ that
\begin{equation*}
\mathbb P\left(\left|X_i - \dfrac{id_i(n) t}{D(n)}\right|\ge \dfrac{\delta id_i(n)t}{D(n)}\right) \le 2\exp\left(-\dfrac{\delta^2 id_i(n) t}{3D(n)}\right).
\end{equation*}

Recall that under the conditions of point~\ref{thm 3.1} of Theorem~\ref{LBgeneral} one has
\begin{equation*}
\sum_{1\le i\le \Delta} i(i-2)p_i < 0.
\end{equation*}

Let 
\begin{equation*}
    \delta(n) = - \sum_{1\le i\le \Delta} \dfrac{i(i-2)d_i(n)}{D(n)}.
\end{equation*}

Since for every $i\in [\Delta]$ we have that $\dfrac{d_i(n)}{n}\underset{n\to +\infty}{\longrightarrow} p_i$, for every large enough $n$ we have 
\begin{equation*}
    \delta(n) \underset{n\to +\infty}{\longrightarrow} \delta := - \sum_{1\le i\le \Delta} \dfrac{i(i-2)p_i}{M}\in (0,1].
\end{equation*}

Moreover, for every $n$, $\delta(n)\le 1$. We deduce that for every large enough $n$  
\begin{align*}
& \mathbb P((X_1, X_2, \dots, X_{\Delta}) \text{ is a tree sequence }|\hspace{0.2em}X_1+\dots+X_{\Delta} = t)\\ 
=\hspace{0.25em} & \mathbb P\left(\sum_{1\le i\le \Delta}(i - 2)X_i = -2\hspace{0.2em}\bigg|\hspace{0.2em}X_1+\dots+X_{\Delta} = t\right)\\
=\hspace{0.25em} & \mathbb P\left(\sum_{1\le i\le \Delta}(i - 2)(X_i-\mathbb E[X_i]) = -\sum_{1\le i\le \Delta}(i - 2)\mathbb E[X_i]-2\hspace{0.2em}\bigg|\hspace{0.2em}X_1+\dots+X_{\Delta} = t\right)\\
\le\hspace{0.25em} & \mathbb P\left(\sum_{1\le i\le \Delta}(i - 2)(X_i-\mathbb E[X_i]) \ge \dfrac{\delta (n) t}{2}\hspace{0.2em}\bigg|\hspace{0.2em}X_1+\dots+X_{\Delta} = t\right)\\
\le\hspace{0.25em} & \sum_{1\le i\le \Delta, i\neq 2} \mathbb P\left(|(i-2)(X_i-\mathbb E[X_i])| \ge \dfrac{|i-2|id_i(n)}{2D(n)} t\hspace{0.2em}\bigg|\hspace{0.2em}X_1+\dots+X_{\Delta} = t\right)\\
=\hspace{0.25em} & \sum_{1\le i\le \Delta, i\neq 2} \mathbb P\left(|X_i-\mathbb E[X_i]| \ge \dfrac{id_i(n)}{2D(n)} t\hspace{0.2em}\bigg|\hspace{0.2em}X_1+\dots+X_{\Delta} = t\right)\\
\le & \sum_{1\le i\le \Delta, i\neq 2} \exp\left(- \dfrac{\left(\frac{1}{2}\right)^2 \frac{id_i(n)}{D(n)}}{3} t\right)\le \hspace{0.25em} \Delta \max_{1\le i\le \Delta, i\neq 2} \exp\left(- \dfrac{ip_i}{24 M} t\right).
\end{align*}

Thus, for every large enough $n$, we have
\begin{align*}
\sum_{\substack{t_1+t_2+\dots +t_{\Delta} = t;\\ (t_1, t_2, \dots, t_{\Delta})\\ \text{ is a tree sequence}}} \dfrac{t-1}{t} \binom{t}{t_1, t_2, \dots, t_{\Delta}} \prod_{1\le i\le \Delta} \left(\dfrac{id_i(n)}{D(n)}\right)^{t_i} 
&\le\hspace{0.3em} \Delta \max_{1\le i\le \Delta, i\neq 2, p_i \neq 0} \exp\left(- \dfrac{ip_i}{24 M} t\right)\\
&=\hspace{0.3em} \Delta  \exp\left(- \min_{1\le i\le \Delta, i\neq 2, p_i \neq 0} \left\{\dfrac{ip_i}{24 M}\right\} t\right).
\end{align*}

By using the dominated convergence theorem we conclude that the sum 
\begin{equation*}
    \sum_{1\le t\le \sqrt{\log n}}\sum_{\substack{t_1+t_2+\dots + t_{\Delta} = t;\\ (t_1, t_2, \dots, t_{\Delta})\\ \text{is a tree sequence}}}\dfrac{4(t-1)}{t} \binom{t}{t_1, t_2, \dots, t_{\Delta}} \left(\prod_{1\le i\le \Delta} \left(\dfrac{id_i(n)}{D(n)}\right)^{t_i}\right)
\end{equation*}
converges to a constant $c = c(\boldsymbol{p}, \Delta) > 0$ given by
\begin{equation*}
    \sum_{t\ge 1}\sum_{\substack{t_1+t_2+\dots + t_{\Delta} = t;\\ (t_1, t_2, \dots, t_{\Delta})\\ \text{is a tree sequence}}}\dfrac{4(t-1)}{t} \binom{t}{t_1, t_2, \dots, t_{\Delta}} \left(\prod_{1\le i\le \Delta} \left(\dfrac{ip_i}{M}\right)^{t_i}\right),
\end{equation*}
which by (\ref{tree constraint}) can be rewritten as
\begin{equation*}
4 \sum_{t_2, \dots, t_{\Delta}\in \mathbb N} \dfrac{t_2+2t_3+\dots+(\Delta-1)t_{\Delta}+1}{t_2+2t_3+\dots+(\Delta-1)t_{\Delta}+2} \binom{t_2+2t_3\dots+(\Delta - 1)t_{\Delta} + 2}{\sum (i-2)t_i + 2, t_2, \dots, t_{\Delta}} \left(\dfrac{p_1}{M}\right)^2 \prod_{2\le i\le \Delta} \left(\dfrac{ip_i}{M}\left(\dfrac{p_1}{M}\right)^{(i-2)}\right)^{t_i}.
\end{equation*}
This finishes the proof of point~\ref{thm 3.1} of Theorem~\ref{LBgeneral}.
\end{proof}

\subsection{Proof of point~\ref{thm 3.3} of Theorem~\ref{LBgeneral} - the supercritical regime}

In this subsection, we prove point~\ref{thm 3.3} of Theorem~\ref{LBgeneral}. First, by Theorem~\ref{MR criterion} a.a.s.\ all connected components except the largest one have size $O(\log n)$. Thus, since $G(n)$ has maximal degree $\Delta$ and contains no isolated vertices, by Corollary~\ref{cor 5.3} each connected component but the largest one forms a part in every partition $\mathcal A$ of $V(G(n))$ with $q(\mathcal A) = q^*(G(n))$ a.a.s. Moreover, the largest component, which we denote by $C_{max}$, contains a.a.s.\ at least $\gamma n$ vertices with $\gamma = \gamma(\boldsymbol{p}, \Delta) > 0$. Thus, it remains to study the giant component. From now on, we condition on the set of vertices of $C_{max}$.\par

The next lemma states that we can find inside the giant component a set of vertices of linear order and high enough density. The component $A_{max}$ found in this lemma plays a role similar to that of $\overline{C_3}$ in the lower bound for random $3-$regular graphs. In the sequel, we identify $A_{max}$ with the graph induced by this vertex set in $C_{max}$.\par

\begin{lemma}\label{new main lemma}
There is a constant $c > 0$ such that the following holds: for every $C>0$ there exists $\varepsilon' > 0$ so that a.a.s.\ there is a set of vertices $A_{max}\subseteq C_{max}$ such that, first,
\begin{equation*}
|A_{max}|\ge \varepsilon' n \mbox{ and } e(A_{max}) \ge |A_{\max}| + C\dfrac{|A_{max}|^2}{n},
\end{equation*}
and second, $C_{max}\setminus A_{max}$ consists of at most $c\varepsilon'^2n$ connected components. 
\end{lemma}

We prove Lemma~\ref{new main lemma} later on. Admitting Lemma~\ref{new main lemma}, we prove point~\ref{thm 3.3} of Theorem~\ref{LBgeneral}.

\begin{proof}[Proof of point~\ref{thm 3.3} of Theorem~\ref{LBgeneral} assuming Lemma~\ref{new main lemma}]
Notice that by assumption of Lemma~\ref{new main lemma}, $C_{max}\setminus A_{max}$ contains at most $c\varepsilon'^2n$ smaller connected components. Note that every component of order larger than $\Delta \sqrt{n}$ has a spanning tree of maximal degree $\Delta$. Thus, by Lemma~\ref{lem:partition}, one may divide each such component of $C_{max}\setminus A_{max}$ into connected components of order at most $\Delta \sqrt{n}$, each containing at most $\Delta^2\sqrt{n}/2$ edges. Thus, one has that the modularity of such a partition into $k$ parts is at least 
\begin{equation*}
    \dfrac{e(A_{max})+n-|A_{max}|-k}{m} - (1+o(1))\dfrac{\Delta^2 |A_{max}|^2}{4m^2}.
\end{equation*}
Since $2m = Mn + o(n)$ and $k \le \mu n + c\varepsilon'^2n+ o(n)$ (recall that $\mu$ is the a.a.s.\ limit of the number of components of $G(n)$ divided by $n$), we have that by Lemma~\ref{new main lemma} the above expression is asymptotically equal to 
\begin{align*}
    & \dfrac{n(1-\mu - c\varepsilon'^2) + C|A_{max}|^2/n}{Mn/2} - (1+o(1))\dfrac{\Delta^2 |A_{max}|^2}{(Mn)^2} + o(1)\\
    \ge \hspace{0.25em} & \dfrac{2(1-\mu)}{M} + (1+o(1))\dfrac{|A_{max}|^2}{(Mn)^2}\left(2MC - \Delta^2 - 2cM\right) + o(1).
\end{align*}
Choosing $c <C - \dfrac{\Delta^2+1}{2M}$ proves the desired result for $\varepsilon = \dfrac{\varepsilon'^2}{M^2}$.
\end{proof}

For a graph $G$, define the $2-$core of $G$ as the largest subgraph of $G$ with respect to inclusion, in which every vertex is of degree at least 2. We include the proof of the following observation for the sake of completeness.
\begin{observation}\label{2core}
The $2-$core of a graph $G$ is well-defined and may be obtained by consecutive deletions of the vertices of degree 0 and 1.
\end{observation}
\begin{proof}
In the end of the process, one obtains a possibly empty subgraph $H$ of $G$ of minimal degree 2. On the other hand, suppose for the sake of contradiction that there is another graph $H'\not \subseteq H$, which has minimal degree at least 2. Then, $H\cup H'$ is also a subgraph of $G$ of minimal degree at least 2. Let $v$ be the first vertex of $H'\setminus H$ that has been deleted throughout the construction of $H$. At the moment of its deletion, since $v\in H'$, $v$ had degree at least 2, which is a contradiction. Thus, every subgraph of $G$ of minimal degree at least 2 is contained in $H$, which proves the observation.
\end{proof}

The idea of the proof of Lemma~\ref{new main lemma} is roughly as follows. First, we prove that the $2-$core of the giant component $C_{max}$ contains a linear number of vertices by a density argument. Then, after a slight modification of the $2-$core in order to get rid of long paths of degree 2 vertices, we apply an argument similar to the one in the proof in the $3-$regular case to obtain that the modularity of the $2-$core of the giant component itself is "non-trivial". Finally, we come back to the giant component itself to conclude the proof of Lemma~\ref{new main lemma}.\par

We first prove that the giant component has density $1+\varepsilon$ for some $\varepsilon > 0$ depending only on $(p_i)_{1\le i\le \Delta}$ a.a.s. For this, initiate an exploration process of $G(n)$, similar to the one for random $3-$regular graphs, and for all $t\ge 0$ record the number $Z_t$ of open half-edges sticking out of the explored component at time step $t$.

\begin{lemma}[\cite{MR}, Lemma 8]\label{lem: MR}
There are constants $\varepsilon \in (0,1)$ and $\xi \in (0, \min\{1/4, M/4\})$ such that, for every $\delta\in (0,\xi)$, $Z_{\lceil \delta n\rceil} \ge \varepsilon \delta n$ a.a.s. Moreover, there is $0 < z = z(\boldsymbol{p}, \Delta) < 1$, for which the probability of the converse is at most $z^n$.
\end{lemma}

Now, fix a constant $\gamma = \gamma(\boldsymbol{p}, \Delta) > 0$ such that, by Lemma~\ref{lem: MR}, the hitting time $t$ of the event $\{Z_t = \gamma n\}$ is a.a.s.\ well-defined. When time $t$ arrives, we know by Theorem~\ref{MR criterion} that we are exploring the giant component. By a  simple concentration argument we deduce the following corollary.
\begin{corollary}\label{cor 5.17}
A.a.s. there is a constant $\gamma' = \gamma'(\boldsymbol{p}, \Delta) > 0$, for which 
\begin{equation*}
    e(C_{max}) - |C_{max}| \ge \gamma' n.
\end{equation*}
\end{corollary}
\begin{proof}
Consider the process $(Z_t)_{t\ge 0}$, and let $T$ be the hitting time of $\gamma n$ by $(Z_t)_{t\ge 0}$, where $\gamma > 0$ is given by Lemma~\ref{lem: MR}. Then, $T$ is a.a.s.\ smaller than $e(G(n))$. Fix the set $S$ of half-edges, incident to explored vertices at time $T$, and match these with other open half-edges or between themselves uniformly at random. By a straightforward concentration argument we deduce that the number of edges with two half-edges in $S$ is concentrated around its expected value, which is at least $\dfrac{\gamma^2 n}{2 \Delta}$. Choosing $\gamma' = \dfrac{\gamma^2}{4 \Delta}$ gives the desired corollary.
\end{proof}

\begin{corollary}\label{cor 5.18}
A.a.s. the $2-$core $C'_{max}$ of $C_{max}$ contains at least $\dfrac{2 \gamma' n}{\Delta}$ vertices of degree at least 3. 
\end{corollary}
\begin{proof}
The $2-$core of a graph may be constructed by step-by-step deletions of the vertices of degree 0 and 1 by Observation~\ref{2core}. Thus, since $C_{max}$ contains no isolated vertices, by Corollary~\ref{cor 5.17} we have that a.a.s.
\begin{equation*}
e(C'_{max}) - |C'_{max}| = e(C_{max}) - |C_{max}| \ge \gamma' n.
\end{equation*}
Since $C'_{max}$ has minimal degree $2$ and maximal degree $\Delta$,
\begin{equation*}
    e(C'_{max}) = \dfrac{1}{2} \sum_{2\le i\le \Delta} i \cdot |\{v\in V(C'_{max})\hspace{0.2em}|\hspace{0.2em}\degg(v) = i\}|\le |C'_{max}| + \dfrac{\Delta}{2} \cdot |\{v\in V(C'_{max})\hspace{0.2em}|\hspace{0.2em}\degg(v) \ge 3\}|.
\end{equation*}
Combining the above two inequalities proves the corollary.
\end{proof}

\begin{lemma}\label{lem: uniform}
Conditionally on $V(C'_{max})$ and the degrees of all vertices of $V(C'_{max})$ in $C'_{max}$, the $2-$core is distributed uniformly at random among all connected graphs on the given degree sequence.
\end{lemma}
\begin{proof}
This follows from the fact that a $2-$core with given vertices and vertex degrees possesses the same number of extensions to $G(n)$ as any other connected $2-$core $C'_{max}$ with the same degrees. Since $G(n)$ is sampled uniformly at random, the restriction of its distribution to $C'_{max}$ conditionally on the vertices and the vertex degrees of $C'_{max}$ is therefore also a uniform distribution.
\end{proof}

Let $\overline{G}(n)$ be a random graph on a given degree sequence with degrees $2,3,\dots, \Delta$ and of order $n$ such that
\begin{equation*}
    \limsup_{n\to +\infty} \dfrac{d_2(\overline{G}(n))}{n} < 1.
\end{equation*}

In particular, one has $\Delta \ge 3$.

\begin{lemma}\label{connectivity 1}
The probability that the random graph $\overline{G}(n)$ is connected is uniformly bounded from below by a positive constant.
\end{lemma}
\begin{proof}
We argue by contradiction. Extracting a subsequence of $(\overline{G}(n))_{n\ge 1}$ if necessary, one may suppose that, first, the probability of $\overline{G}(n)$ being connected tends to 0, and second, the proportion of vertices of degree $i\in [\Delta]$, $i\ge 2$ tends to $\overline{p}_i$ with $\overline{p}_2 < 1$. Under these conditions Lemma~\ref{connectivity} implies that the probability that no vertex remains outside of the largest connected component in $\overline{G}_n$ is a positive constant, which leads to a contradiction. The lemma is proved.
\end{proof}

\begin{corollary}\label{cor 5.20}
If an event $A_n$ happens a.a.s.\ for $\overline{G}(n)$, then $A_n$ happens a.a.s.\ for $\overline{G}(n)$ conditionally on $\overline{G}(n)$ being a connected graph.\qed
\end{corollary}

Now, let $C'_{cm, max}$ be the random graph, constructed on the degree sequence of $C'_{max}$ (which is random in itself, so we condition on this degree sequence) according to the configuration model. By abuse of notation we view $C_{max}$ as an extension of $C'_{max}$ as well as of $C'_{cm, max}$ below. By Lemma~\ref{lem: uniform} and Corollary~\ref{cor 5.20} we know that if some event happens a.a.s.\ for $C'_{cm, max}$, then it also happens a.a.s.\ for $C'_{max}$.\par

Fix $\varepsilon' > 0$. We start an exploration process of $C'_{cm, max}$ as follows. Let $S_0$ be an empty set. We pick a random initial vertex $v_0$ of $C'_{cm, max}$ with probability proportional to its degree in $C'_{cm, max}$, and construct $S_1 = \{v_0\}$. At time $t$, choose an arbitrary open half-edge $e_{1/2}$ sticking out of a vertex in $S_t$ and explore the edge that contains it. If it leads to a vertex $v_t$ that has never been seen before, let $S_{t+1} = S_t\cup \{v_t\}$. If not, let $S_{t+1} = S_t$ and continue. Finally, stop the process at time $t'$ when $\varepsilon' n$ vertices of $C'_{cm, max}$ have been explored, and define $S = S_{t'}$. Moreover, let $E$ be the set of explored edges up to time $t'$ (that is, the edges obtained from paired half-edges inside $S$).\par

Then, fix an even positive integer $\ell$. Now, order in a row the explored vertices with an open half-edge and, for every vertex in the row, explore step by step all vertices of $C'_{cm, max}$ at distance at most $\ell/2$ from the current one. Moreover, at any step when one finds an $S-$chain, add all vertices of this $S-$chain to $S$ and continue the exploration.\par

Let $A'_{max}$ be the graph, induced by the set $S$ together with the vertices in its $\frac{\ell}{2}-$neighborhood in $C'_{max}$ (which is additionally explored after step $t'$) that participate in $S$-chains of length at most $\ell$.
Let $A_{max}$ be the graph, induced in $C_{max}$ by the set of vertices of $A'_{max}$ and the ones which connect to $A'_{max}$ in $C_{max}\setminus C'_{cm, max}$ via paths with edges in $E(C_{max})\setminus E(C'_{cm, max})$. Otherwise said, $A_{max}$ is constructed from $A'_{max}$ by attaching trees to the vertices of $A'_{max}$ which were deleted in the construction of the $2-$core of $C_{max}$. Notice that vertices inside $C'_{cm, max}$ participate in the $2-$core, and if they are not in $A'_{max}$, they will not be added to $A_{max}$, see Figure~\ref{fig2core}.

\begin{figure}
\centering
\begin{tikzpicture}[scale = 0.7, line cap=round,line join=round,x=1cm,y=1cm]
\clip(-12.8,-6.04) rectangle (12.8,6.04);
\draw [line width=0.5pt] (-0.52-6,0)-- (0.7-6,-0.88);
\draw [line width=0.5pt] (0.7-6,-0.88)-- (2.56-6,-0.66);
\draw [line width=0.5pt] (2.56-6,-0.66)-- (2.26-6,0.9);
\draw [line width=0.5pt] (2.26-6,0.9)-- (4.2-6,0.26);
\draw [line width=0.5pt] (4.2-6,0.26)-- (2.56-6,-0.66);
\draw [line width=0.5pt] (2.26-6,0.9)-- (1.04-6,1.8);
\draw [line width=0.5pt] (1.04-6,1.8)-- (0.7-6,-0.88);
\draw [line width=0.5pt] (-0.52-6,0)-- (-0.5-6,-2.02);
\draw [line width=0.5pt] (-0.5-6,-2.02)-- (0.7-6,-0.88);
\draw [line width=0.5pt] (-0.52-6,0)-- (-1.24-6,1.26);
\draw [line width=0.5pt] (-1.24-6,1.26)-- (-3.18-6,-0.94);
\draw [line width=0.5pt] (-3.18-6,-0.94)-- (-0.5-6,-2.02);
\draw [line width=0.5pt] (-0.5-6,-2.02)-- (1.12-6,-2.82);
\draw [line width=0.5pt] (1.12-6,-2.82)-- (2.56-6,-0.66);
\draw [line width=0.5pt] (2.56-6,-0.66)-- (3.56-6,-2.14);
\draw [line width=0.5pt] (3.56-6,-2.14)-- (4.2-6,0.26);
\draw [line width=0.5pt] (3.56-6,-2.14)-- (1.12-6,-2.82);
\draw [line width=0.5pt] (-1.24-6,1.26)-- (-0.44-6,2.94);
\draw [line width=0.5pt] (-0.44-6,2.94)-- (1.04-6,1.8);
\draw [rotate around={54.77081338981406:(-1.81-6,1)},line width=0.5pt] (-1.81-6,1) ellipse (4.31271415952206cm and 3.5998615836920846cm);

\draw [line width=0.5pt] (-0.52+6,0)-- (0.7+6,-0.88);
\draw [line width=0.5pt] (0.7+6,-0.88)-- (2.56+6,-0.66);
\draw [line width=0.5pt] (2.56+6,-0.66)-- (2.26+6,0.9);
\draw [line width=0.5pt] (2.26+6,0.9)-- (4.2+6,0.26);
\draw [line width=0.5pt] (4.2+6,0.26)-- (2.56+6,-0.66);
\draw [line width=0.5pt] (2.26+6,0.9)-- (1.04+6,1.8);
\draw [line width=0.5pt] (1.04+6,1.8)-- (0.7+6,-0.88);
\draw [line width=0.5pt] (-0.52+6,0)-- (-0.5+6,-2.02);
\draw [line width=0.5pt] (-0.5+6,-2.02)-- (0.7+6,-0.88);
\draw [line width=0.5pt] (-0.52+6,0)-- (-1.24+6,1.26);
\draw [line width=0.5pt] (-1.24+6,1.26)-- (-3.18+6,-0.94);
\draw [line width=0.5pt] (-3.18+6,-0.94)-- (-0.5+6,-2.02);
\draw [line width=0.5pt] (-0.5+6,-2.02)-- (1.12+6,-2.82);
\draw [line width=0.5pt] (1.12+6,-2.82)-- (2.56+6,-0.66);
\draw [line width=0.5pt] (2.56+6,-0.66)-- (3.56+6,-2.14);
\draw [line width=0.5pt] (3.56+6,-2.14)-- (4.2+6,0.26);
\draw [line width=0.5pt] (3.56+6,-2.14)-- (1.12+6,-2.82);
\draw [line width=0.5pt] (-1.24+6,1.26)-- (-0.44+6,2.94);
\draw [line width=0.5pt] (-0.44+6,2.94)-- (1.04+6,1.8);
\draw [rotate around={54.77081338981406:(-1.81+6,1)},line width=0.5pt] (-1.81+6,1) ellipse (4.31271415952206cm and 3.5998615836920846cm);
\draw [line width=0.5pt] (-1.24+6,1.26)-- (-2.92+6,1.74);
\draw [line width=0.5pt] (-2.92+6,1.74)-- (-3.28+6,2.72);
\draw [line width=0.5pt] (-2.92+6,1.74)-- (-4.12+6,1.58);
\draw [line width=0.5pt] (-4.12+6,1.58)-- (-4.26+6,2.64);
\draw [line width=0.5pt] (-4.12+6,1.58)-- (-4.96+6,1.64);
\draw [line width=0.5pt] (-3.18+6,-0.94)-- (-4.24+6,-0.9);
\draw [line width=0.5pt] (-4.24+6,-0.9)-- (-5.16+6,-0.32);
\draw [line width=0.5pt] (-0.44+6,2.94)-- (-1.74+6,3.18);
\draw [line width=0.5pt] (-0.44+6,2.94)-- (-1+6,3.82);
\draw [line width=0.5pt] (-1+6,3.82)-- (-1.84+6,4.3);
\draw [line width=0.5pt] (-1+6,3.82)-- (-0.84+6,4.62);
\draw [line width=0.5pt] (-0.44+6,2.94)-- (-0.04+6,3.82);
\draw [line width=0.5pt] (-1.74+6,3.18)-- (-2.6+6,3.7);
\begin{scriptsize}
\draw [fill = black](-0.52-6,0) circle (2.5pt);
\draw [fill = black](0.7-6,-0.88) circle (2.5pt);
\draw [fill = black](2.56-6,-0.66) circle (2.5pt);
\draw [fill = black](2.26-6,0.9) circle (2.5pt);
\draw [fill = black](4.2-6,0.26) circle (2.5pt);
\draw [fill = black](1.04-6,1.8) circle (2.5pt);
\draw [fill = black](-0.5-6,-2.02) circle (2.5pt);
\draw [fill = black](-1.24-6,1.26) circle (2.5pt);
\draw [fill = black](-3.18-6,-0.94) circle (2.5pt);
\draw [fill = black](1.12-6,-2.82) circle (2.5pt);
\draw [fill = black](3.56-6,-2.14) circle (2.5pt);
\draw [fill = black](-0.44-6,2.94) circle (2.5pt);

\draw [fill = black](-0.52+6,0) circle (2.5pt);
\draw [fill = black](0.7+6,-0.88) circle (2.5pt);
\draw [fill = black](2.56+6,-0.66) circle (2.5pt);
\draw [fill = black](2.26+6,0.9) circle (2.5pt);
\draw [fill = black](4.2+6,0.26) circle (2.5pt);
\draw [fill = black](1.04+6,1.8) circle (2.5pt);
\draw [fill = black](-0.5+6,-2.02) circle (2.5pt);
\draw [fill = black](-1.24+6,1.26) circle (2.5pt);
\draw [fill = black](-3.18+6,-0.94) circle (2.5pt);
\draw [fill = black](1.12+6,-2.82) circle (2.5pt);
\draw [fill = black](3.56+6,-2.14) circle (2.5pt);
\draw [fill = black](-0.44+6,2.94) circle (2.5pt);
\draw [fill = black](-2.92+6,1.74) circle (2.5pt);
\draw [fill = black](-3.28+6,2.72) circle (2.5pt);
\draw [fill = black](-4.12+6,1.58) circle (2.5pt);
\draw [fill = black](-4.26+6,2.64) circle (2.5pt);
\draw [fill = black](-4.96+6,1.64) circle (2.5pt);
\draw [fill = black](-4.24+6,-0.9) circle (2.5pt);
\draw [fill = black](-5.16+6,-0.32) circle (2.5pt);
\draw [fill = black](-1.74+6,3.18) circle (2.5pt);
\draw [fill = black](-1+6,3.82) circle (2.5pt);
\draw [fill = black](-1.84+6,4.3) circle (2.5pt);
\draw [fill = black](-0.84+6,4.62) circle (2.5pt);
\draw [fill = black](-0.04+6,3.82) circle (2.5pt);
\draw [fill = black](-2.6+6,3.7) circle (2.5pt);

\draw[color=black] (-3,-3.25) node {\Large $C'_{max}$};
\draw[color=black] (9,-3.25) node {\Large $C_{max}$};
\draw[color=black] (3.45,-2.55) node {\large $A_{max}$};
\draw[color=black] (-8.5,-2.55) node {\large $A'_{max}$};

\end{scriptsize}
\end{tikzpicture}
\caption{The giant component $C_{max}$, its $2-$core $C'_{max}$ and the ``dense'' sets $A_{max}$ and $A'_{max}$.}
\label{fig2core}
\end{figure}
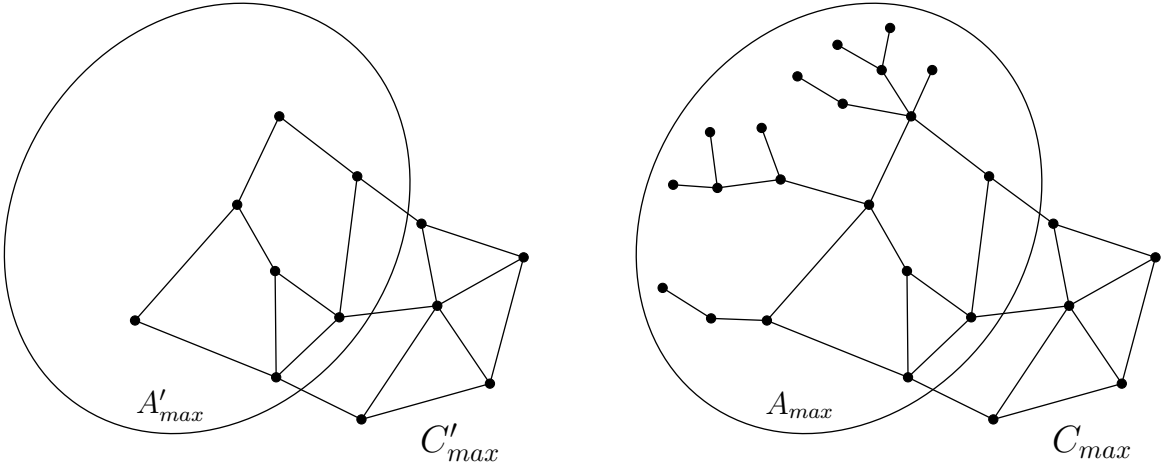

\begin{observation}\label{ob 5.22}
The graph $A'_{max}$ contains at most
\begin{equation*}
    \dfrac{(\Delta^{\frac{\ell}{2} + 1} - 1)\varepsilon' n}{\Delta - 1}
\end{equation*}
vertices.
\end{observation}
\begin{proof}
This is a simple consequence of the fact that every vertex is of degree at most $\Delta$.
\end{proof}

\begin{lemma}\label{linear nb trees}
There is a constant $C = C(\boldsymbol{p}, \Delta) > 0$, such that the number of vertices in $C'_{cm, max}$ incident to at least one edge in $E(C_{max})\setminus E(C'_{max})$ is at least $Cn$.
\end{lemma}
\begin{proof}
Recall the exploration process $(Z_t)_{t\ge 0}$ defined just before Lemma~\ref{lem: MR} and let $T$ be the a.a.s.\ well-defined hitting time of $\gamma n$ by $(Z_t)_{t\ge 0}$, where $\gamma > 0$ is given after Lemma~\ref{lem: MR}. At this stage, match all $\gamma n$ open half-edges sticking out of the connected component we are exploring. We prove that there is a constant $C > 0$ depending only on $\boldsymbol{p}$ (recall that $\gamma$ depends only on $\boldsymbol{p}$ and $\Delta$ as well) such that a.a.s.\ there are at least $Cn$ vertices among the ones incident to the open half-edges at time $T$, which are all connected to vertices of degree 1: indeed, there are a.a.s.\ at least $\gamma n/\Delta$ vertices incident to open half-edges at time $T$, and moreover a.a.s.\ a positive proportion of these vertices will match at least one of their remaining open half-edges to a vertex of degree 1, which proves the lemma.
\end{proof}

\begin{lemma}\label{cmp A and A'}
There exists $\alpha = \alpha(\boldsymbol{p}, \Delta) > 1$ such that for every small enough $\varepsilon' > 0$ a.a.s.\ $|A_{max}| \le \alpha |A'_{max}|$.
\end{lemma}
\begin{proof}
Recall that $A_{max}$ is constructed from $A'_{max}$ by adding vertices of $C_{max}\setminus C'_{max}$. We work in the configuration model $C'_{cm, max}$ associated to the degree sequence of the $2-$core. Conditionally on the vertices and the half-edges contained in $C'_{cm, max}$, one may define the \textit{type} of a vertex $v$ in $C'_{cm, max}$ as $(\deg_{C_{max}}(v) - \deg_{C'_{cm, max}}(v), \deg_{C'_{cm, max}}(v))$. Denote also by $D'_{cm, max}$ the number of vertices $v$ in $C'_{cm,max}$ with $\deg_{C_{max}}(v) - \deg_{C'_{cm, max}}(v)\ge 1$.\par

Let $C_{\Delta} = \dfrac{\Delta^{\frac{\ell}{2}+1}-1}{\Delta - 1}$ and choose $\varepsilon' < \dfrac{\gamma'}{\Delta C_{\Delta}}$. This ensures that $|A'_{max}| < |C'_{max}|/2$ by Corollary~\ref{cor 5.18} and Observation~\ref{ob 5.22}.

\begin{sublemma}\label{sublem}
Fix $d''\in [\Delta]$. Among the vertices of types $\{(s', s''): s''\le d'', s'\le \Delta - s''\}$ at most a $\frac{C_{\Delta}\Delta\varepsilon' d''}{\gamma'}-$proportion participates in $A'_{max}$ a.a.s.
\end{sublemma}
\begin{proof}[Proof of Sublemma~\ref{sublem}]
$A'_{max}$ is constructed by consecutive exploration of open half-edges and at any step the probability that a vertex of type $(s', s'')$ with $s''\le d''$ is added to $A'_{max}$ is at most $\frac{d''}{2(|C'_{cm, max}| - |A'_{max}|)}$ (every vertex in $C'_{cm, max}$ is of degree at least 2 in $C'_{cm, max}$), which is bounded from above by $\frac{d''}{|C'_{cm, max}|}$. We conclude that the expected proportion of the vertices of type $(d', d'')$ is bounded from above by $|A'_{max}|\frac{d''}{|C'_{cm, max}|}$. By a Chernoff bound\footnote{Note that the events of adding a vertex of type $(d', d'')$ over several steps are dependent, but they are dominated by a family of independent Bernoulli random variables by the above analysis.} (Lemma~\ref{chernoff:bd}) over the $|A'_{max}|$ steps when a new vertex is added to the explored component we conclude that the proportion of the vertices of type $(d', d'')$ is a.a.s.\ bounded from above by $2|A'_{max}|\frac{d''}{|C'_{cm, max}|}$. We conclude since $|A'_{max}|\le C_{\Delta} \varepsilon' n$ by Observation~\ref{ob 5.22} and $|C'_{max}|\ge \frac{2\gamma' n}{\Delta}$ by Corollary~\ref{cor 5.18}.
\end{proof}

In particular, Sublemma~\ref{sublem} applied with $d'' = \Delta$ ensures that the proportion of the vertices in $A'_{max}$ that have at least one half-edge outside $C'_{max}$ is at most $\frac{C_{\Delta}\Delta^2\varepsilon'}{\gamma'}$ of all such vertices in $C'_{cm, max}$ since $d''\le \Delta$.\par

Now, notice that the distribution of the graph consisting of the half-edges in $C_{max}$, but not in $C'_{cm, max}$, is uniform among the forests with roots in $D'_{cm, max}$. Indeed, conditionally on the vertex set of $C_{max}$, every connected graph constructed on this set of vertices has the same number of extensions to $G(n)$. Based on this observation, Theorem 2 in~\cite{Pav96}\footnote{Theorem 2 in~\cite{Pav96} gives a precise information about the distribution of the maximum size of a tree of a uniform random forest on $N$ roots and $n$ non-root vertices. The setting of a degree sequence prescribed in advance is a special case of the theorem. We provide a justification of the main technical assumption. Assume that $N/n$ tends to $b$ as $n\to +\infty$. Denote by $q_k$ the proportion of vertices in $C_{max}\setminus C'_{cm,max}$ with $k$ children in the tree attached to $C'_{cm,max}$. (Our $q_k$ corresponds to $p_k$ in the original paper but this notation has already been introduced in our setting.) Set $F(t) = \sum_{i=0}^{\Delta - 1} q_k t^k + b$. Then, let $\lambda$ be a root of the equation $t F'(t)/F(t) = n/(N+n)$ - it is well defined and unique since the function $t\mapsto t F'(t)/F(t)$ is strictly increasing for $\Delta\ge 2$ (which is our case) and its image on $\mathbb R_+$ is $[0,\Delta-1)$. We need to justify now that $\lambda/F(\lambda) \in (0,1)$. We argue by contradiction - then, $F'(\lambda) \le n/(N+n) = F'(1)$, so since $F'$ is strictly increasing on $\mathbb R_+$ one has that $\lambda \le 1$. Moreover, for any $t$ sufficiently close to $1$ one may ensure that $t < F(t)$ since $F$ is continuous and $F(1) = 1+b$, which means that $\lambda F'(\lambda)/F(\lambda)\le F'(1)/F(1) < n/(N+n)$. This contradicts the definition of $\lambda$. In particular, we deduce that the largest tree in a uniformly chosen forest is of size $O(\log n)$.} justifies that the maximum size of a tree attached to a vertex in $D'_{cm, max}$ is of order $O(\log n)$. Now, let us sample the forest induced by the vertices in $C_{max}\setminus C'_{cm, max}$ (but without providing the corresponding root in $C'_{cm,max}$) - this ensures a family of $\Omega(n)$ trees of sizes $O(\log n)$ a.a.s. Let us condition on these sizes. The last key ingredient in the proof is the following sublemma, which might be well known.
\begin{sublemma}\label{sublem 1}
Fix $n$ real numbers $x_1, \dots, x_n$ in the interval $[1,M]$ and an integer $k\in [n]$. Then, fix a random subset $\Lambda$ of $[n]$ of size $k$ chosen uniformly at random and denote $S = \sum_{i\in \Lambda} x_i$. Then, for all $t\ge 0$
\begin{equation*}
    \mathbb P(|S-\mathbb E S|\ge t)\le \exp\left(-\dfrac{t^2}{2nM}\right).
\end{equation*}
\end{sublemma}
\begin{proof}[Proof of Sublemma~\ref{sublem 1}]
For every $i\in \{0,\dots,n\}$ define $Y_i = \mathbb E[S\mid \Lambda\cap [i]]$. Then, $(Y_i)_{i=0}^n$ is a martingale satisfying $Y_0 = \mathbb E S$ and $Y_n = S$. Let us show that for every $i\in [n]$ one has $|Y_{i-1}-Y_i|\le M$. 

Fix any $i\in [n]$ and any subset $X_{i-1}$ of $[i-1]$ with at most $k$ elements. Then, we have that
\begin{align*}
    \mathbb E[S\mid \Lambda\cap [i-1] = X_{i-1}] 
    &=\hspace{0.3em} \mathbb E[S\mid \Lambda\cap [i] = X_{i-1}\cup \{i\}] \mathbb P(i\in \Lambda\mid \Lambda\cap [i-1] = X_{i-1})\\ 
    &+\hspace{0.3em} \mathbb E[S\mid \Lambda\cap [i] = X_{i-1}] \mathbb P(i\notin \Lambda\mid \Lambda\cap [i-1] = X_{i-1}).
\end{align*}
Moreover, from the set $\Lambda$ chosen uniformly at random and conditioned on $\Lambda\cap [i] = X_{i-1}\cup \{i\}$ one obtains the set $\Lambda$ chosen uniformly at random and conditioned on $\Lambda\cap [i] = X_{i-1}$ by taking out the element $i$ and adding an element in $[n]\setminus ([i]\cup \Lambda)$ chosen uniformly at random. Conversely, from the set $\Lambda$ chosen uniformly at random and conditioned on $\Lambda\cap [i] = X_{i-1}$ one obtains the set $\Lambda$ chosen uniformly at random and conditioned on $\Lambda\cap [i] = X_{i-1}\cup \{i\}$ by taking out an element in $([n]\setminus [i])\cap \Lambda$ chosen uniformly at random and adding the element $i$. We conclude that
\begin{equation*}
|\mathbb E[S\mid \Lambda\cap [i] = X_{i-1}\cup \{i\}] - \mathbb E[S\mid \Lambda\cap [i] = X_{i-1}]|\le \max_{i+1\le j\le n} |x_i - x_j| < M.
\end{equation*}
We conclude by a direct application of Azuma's martingale inequality (see for example Chapter 2 in~\cite{JLR}).
\end{proof}

We apply Sublemma~\ref{sublem 1} with $t = n^{2/3}$. In particular, the proportion of vertices in $C_{max}\setminus C'_{max}$ that attach to some root in $D'_{cm,max}$ is sharply concentrated around its expected value, and in particular at most $2 C_{\Delta} \Delta^3/\gamma'$ a.a.s. This proves that one may choose e.g. $\alpha = 1 + 2C_{\Delta}\Delta^3/\gamma'$.
\end{proof}

We denote by $(d'_2(n), \dots, d'_{\Delta}(n))$ the degree sequence of $C'_{cm, max}$ on which we conditioned earlier. We fix a positive integer $\ell$ that we specify in the sequel. Recall that $S = S_{t'}$ is the set of the first $\varepsilon' n$ explored vertices in $C'_{cm, max}$.

\begin{lemma}\label{lem 5.24}
The expected number of $S-$chains of length $\ell$ in $C'_{cm, max}$ is at least 
\begin{equation*}
    (1+o(1))\dfrac{\delta^2 \varepsilon'^2 n}{M} \left(1+\dfrac{6\gamma'}{\Delta M}\right)^{\ell-1},
\end{equation*}
where $\delta = \delta(\varepsilon') > 0$ is constant over the interval $\varepsilon'\in (0,\xi)$, for $\xi$ given by Lemma~\ref{lem: MR}, and $\gamma' > 0$ is given by Corollary~\ref{cor 5.17}.
\end{lemma}
\begin{proof}
By Lemma~\ref{lem: MR} for every small enough $\varepsilon'$ the number of half-edges, attached to vertices in $S$, which participate in an edge between a vertex in $S$ and a vertex outside $S$, is at least $\delta \varepsilon' n$ a.a.s. Denote this set of half-edges by $\partial_{1/2} S$. Conditionally on the set $S$, there remains a number of non-explored vertices (at time $t'$) of degree $i$, which we denote by $d''_i(n)$. Now, any chain of length $\ell$ is defined by:

\begin{itemize}
    \item a tuple $(\ell_2, \dots, \ell_\Delta)$,
    \item for all $i\in \{2,\dots,\Delta\}, \ell_i$ vertices outside $S$ of degree $i$,
    \item the order of the $\ell-1$ vertices outside $S$ in the chain, and
    \item the choice of half-edges that participate in the chain.
\end{itemize}

To find the expected number of $S$-chains of length $\ell$ it remains to multiply by the probability of each edge being present. In total, this gives 
\begin{align}
    & |\partial_{1/2} S|^2 \sum_{\ell_2 + \dots + \ell_{\Delta} = \ell-1} (\ell-1)! \left(\prod_{2\le i\le \Delta} \binom{d''_i(n)}{\ell_i} (i(i-1))^{\ell_i}\right)\dfrac{\left(\left(\sum_{2\le i\le \Delta} id''_i(n)\right)-2\ell-1\right)!!}{\left(\left(\sum_{2\le i\le \Delta} id''_i(n)\right)-1\right)!!}\nonumber \\ \ge
    &\hspace{0.3em} (1+o(1))(\delta\varepsilon' n)^2 \sum_{\ell_2+\dots+\ell_{\Delta} = \ell-1}\binom{\ell-1}{\ell_2, \ell_3, \dots, \ell_{\Delta}}\left(\prod_{2\le i\le \Delta}\left(i(i-1)d''_i(n)\right)^{\ell_i}\right)\dfrac{1}{\left(\sum_{2\le i\le \Delta} id''_i(n)\right)^{\ell}} \nonumber\\ =
    &\hspace{0.3em} \label{last line} (1+o(1))\dfrac{(\delta\varepsilon' n)^2}{2d''_2(n)+\dots+\Delta d''_{\Delta}(n)}\left(\sum_{2\le i\le \Delta}\dfrac{i(i-1)d''_i(n) }{\left(\sum_{2\le j\le \Delta} id''_i(n)\right)}\right)^{\ell - 1}.
\end{align}
Now, we have that, first, $\sum_{2\le i\le \Delta} id''_i(n)\le D(n)$, and second,
\begin{equation*}
 \sum_{2\le i\le \Delta} i(i-2)d''_i(n)\ge \dfrac{6\gamma'}{\Delta}n.  
\end{equation*}
by Corollary~\ref{cor 5.18}. Thus, since $D(n) = Mn + o(n)$, (\ref{last line}) is bounded from below by
\begin{equation*}
(1+o(1))\dfrac{\delta^2 \varepsilon'^2 n}{M} \left(1+\dfrac{6\gamma'}{\Delta M}\right)^{\ell-1}.
\end{equation*}
The lemma is proved.
\end{proof}

\begin{observation}\label{short cycles ob}
The number of vertices of distance at most $\ell$ from a cycle of length at most $2\ell$ in $C'_{cm, max}$ is $O(\log n)$ a.a.s.
\end{observation}
\begin{proof}
By an immediate first moment calculation (in the same way as the proof of Lemma~\ref{Number of copies} for trees, see also Theorem 9.5 in \cite{JLR} in the regular case) one may deduce that for every $\ell\in \mathbb N$ the number of cycles of length at most $2\ell$ is a.a.s.\ at most $O(\log n)$ and the same holds for the number of vertices in such cycles. Since $C'_{cm, max}$ has maximum degree at most $\Delta$, each vertex is at distance at most $\ell$ from at most $\frac{\Delta^{\ell+1}-1}{\Delta - 1}$ other vertices in $C'_{cm, max}$. Thus, the number of vertices at distance at most $\ell$ from cycles of length at most $2\ell$ is $O(\log n)$ as well.
\end{proof}

\begin{lemma}\label{lem 5.26}
There are constants $c_{\ell}, c'_{\ell} > 0$ such that the expected number of vertices:
\begin{itemize}
    \item in $C'_{cm, max}\setminus S$, which participate in an $S-$chain of length at most $\ell$, is at most $c_{\ell} \varepsilon'^2 n$.
    \item in $C'_{cm, max}\setminus S$, which participate in more than one $S-$chain of length at most $\ell$, is at most $c'_{\ell} \varepsilon'^3 n$.
\end{itemize}
\end{lemma}
\begin{proof}
We prove the second point first. If a vertex participates in more than one $S-$chain of length at most $\ell$, then either it is part of two $S-$chains with a common subpath or it is at distance at most $\ell$ from a cycle in $C'_{cm, max}$ of length at most $2\ell$, see Figure~\ref{fig: example}. By Observation~\ref{short cycles ob} there are a.a.s.\ at most $O(\log n)$ vertices of the second type in $C'_{cm, max}$. Therefore, it remains to count the number of vertices of the first type.\par
The number of vertices at distance at most $\ell$ from a vertex is at most $\frac{\Delta^{\ell+1}-1}{\Delta - 1}$. Therefore, the probability that a vertex participates in two $S-$chains of length at most $\ell$ with a common subpath is at most the probability that at least 3 of the vertices at distance at most $\ell$ from it are in $S$, which is at most 
\begin{equation*}
    \binom{\frac{\Delta^{\ell+1}-1}{\Delta - 1}}{3} (\Delta \varepsilon')^3.
\end{equation*}
Summing over all vertices in $C'_{cm, max}\setminus S$, we may choose
\begin{equation*}
c'_{\ell} = \binom{\frac{\Delta^{\ell+1}-1}{\Delta - 1}}{3} \Delta^3,
\end{equation*}
which proves the second point. The first point follows along the same lines, with the constant $c_{\ell}$ given by 
\begin{equation*}
c_{\ell} = \binom{\frac{\Delta^{\ell+1}-1}{\Delta - 1}}{2} \Delta^2.
\end{equation*}
\end{proof}

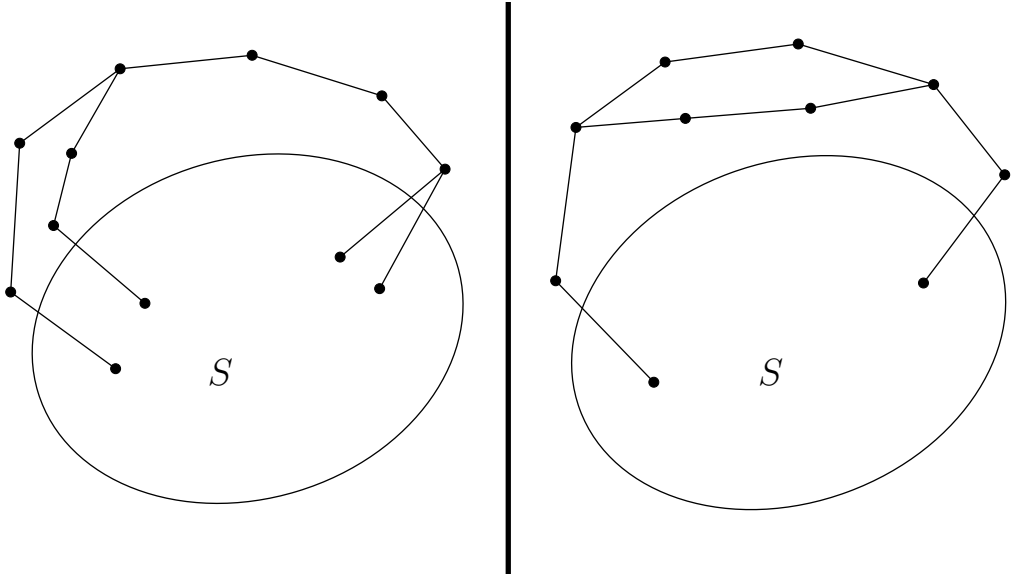
\begin{figure}
\centering
\begin{tikzpicture}[scale = 0.75, line cap=round,line join=round,x=1cm,y=1cm]
\clip(-13,-6.14) rectangle (13,4);
\draw [rotate around={16.87878388767263:(-5.82,-1.79)},line width=0.5pt] (-5.82,-1.79) ellipse (3.8831798726771267cm and 3.0165188419044777cm);
\draw [rotate around={20.213749138595876:(3.77,-1.86)},line width=0.5pt] (3.77,-1.86) ellipse (3.9455234627317224cm and 3.0134125829309406cm);
\draw [line width=0.5pt] (-8.16,-2.5)-- (-10.02,-1.14);
\draw [line width=0.5pt] (-10.02,-1.14)-- (-9.86,1.5);
\draw [line width=0.5pt] (-9.86,1.5)-- (-8.08,2.82);
\draw [line width=0.5pt] (-8.08,2.82)-- (-5.74,3.06);
\draw [line width=0.5pt] (-5.74,3.06)-- (-3.44,2.34);
\draw [line width=0.5pt] (-3.44,2.34)-- (-2.32,1.04);
\draw [line width=0.5pt] (-2.32,1.04)-- (-3.48,-1.08);
\draw [line width=0.5pt] (-8.08,2.82)-- (-8.94,1.32);
\draw [line width=0.5pt] (-8.94,1.32)-- (-9.26,0.04);
\draw [line width=0.5pt] (-9.26,0.04)-- (-7.64,-1.34);
\draw [line width=0.5pt] (-2.32,1.04)-- (-4.18,-0.52);
\draw [line width=0.5pt] (1.38,-2.74)-- (-0.36,-0.94);
\draw [line width=0.5pt] (-0.36,-0.94)-- (0,1.78);
\draw [line width=0.5pt] (0,1.78)-- (1.58,2.94);
\draw [line width=0.5pt] (1.58,2.94)-- (3.94,3.26);
\draw [line width=0.5pt] (3.94,3.26)-- (6.34,2.54);
\draw [line width=0.5pt] (6.34,2.54)-- (7.6,0.94);
\draw [line width=0.5pt] (7.6,0.94)-- (6.16,-0.98);
\draw [line width=0.5pt] (0,1.78)-- (1.94,1.94);
\draw [line width=0.5pt] (1.94,1.94)-- (4.16,2.12);
\draw [line width=0.5pt] (4.16,2.12)-- (6.34,2.54);
\draw [line width=2pt] (-1.2,-10)-- (-1.2,10);
\begin{scriptsize}
\draw [fill=black] (-8.16,-2.5) circle (2.5pt);
\draw [fill=black] (-3.48,-1.08) circle (2.5pt);
\draw [fill=black] (1.38,-2.74) circle (2.5pt);
\draw [fill=black] (6.16,-0.98) circle (2.5pt);
\draw [fill=black] (-10.02,-1.14) circle (2.5pt);
\draw [fill=black] (-9.86,1.5) circle (2.5pt);
\draw [fill=black] (-8.08,2.82) circle (2.5pt);
\draw [fill=black] (-5.74,3.06) circle (2.5pt);
\draw [fill=black] (-3.44,2.34) circle (2.5pt);
\draw [fill=black] (-2.32,1.04) circle (2.5pt);
\draw [fill=black] (-8.94,1.32) circle (2.5pt);
\draw [fill=black] (-9.26,0.04) circle (2.5pt);
\draw [fill=black] (-7.64,-1.34) circle (2.5pt);
\draw [fill=black] (-4.18,-0.52) circle (2.5pt);
\draw [fill=black] (-0.36,-0.94) circle (2.5pt);
\draw [fill=black] (0,1.78) circle (2.5pt);
\draw [fill=black] (1.58,2.94) circle (2.5pt);
\draw [fill=black] (3.94,3.26) circle (2.5pt);
\draw [fill=black] (6.34,2.54) circle (2.5pt);
\draw [fill=black] (7.6,0.94) circle (2.5pt);
\draw [fill=black] (1.94,1.94) circle (2.5pt);
\draw [fill=black] (4.16,2.12) circle (2.5pt);
\draw[color=black] (3.45,-2.55) node {\Large $S$};
\draw[color=black] (-6.3,-2.55) node {\Large $S$};
\end{scriptsize}
\end{tikzpicture}
\caption{On the left: an example of two $S-$chains of length 7 with a common subpath. On the right: an example of two $S-$chains of length 7, whose union contains a cycle.}
\label{fig: example}
\end{figure}

\begin{observation}\label{ob: concentration}
Fix $\varepsilon'\in \left(0, \dfrac{\gamma'}{2(\Delta^{\ell+1}-1)\Delta}\right)$. The number of $S-$chains of length $\ell$, the number of vertices in at least one $S-$chain of length at most $\ell$ and the number of vertices in more than one $S-$chain of length at most $\ell$ are sharply concentrated around their expected values.
\end{observation}
\begin{proof}
Given the degree sequence $(d'_i(n))_{2\le i\le \Delta}$, the number of $S-$chains of length $\ell$, the number of vertices at least one $S-$chain of length at most $\ell$ and the number of vertices in more than one $S-$chain of length at most $\ell$ can be computed via the differential equation method. To do this, explore the $\frac{\ell}{2}-$th neighborhood of $S$ by revealing one edge per time step. Start with $\mathcal C(0) = \emptyset$, which will be the number of $S-$chains of length $\ell$, $\mathcal V_1(0) = \emptyset$, which will be the number of vertices in some $S-$chain of length at most $\ell$, and $\mathcal V_{\ge 2}(0) = \emptyset$, which will be the number of vertices in at least two $S-$chains of length at most $\ell$. Notice that every revealed edge may participate in a bounded number of $S-$chains of length at most $\ell$, and therefore at any time step $t\ge 1$ each of $|\mathcal C(t) - \mathcal C(t-1)|$, $|\mathcal V_1(t) - \mathcal V_1(t-1)|$ and $|\mathcal V_{\ge 2}(t) - \mathcal V_{\ge 2}(t-1)|$ may increase in size by only a bounded number of new elements. Moreover, by choice of $\varepsilon'$, the $\frac{\ell}{2}-$th neighborhood of $S$ contains at most $\gamma' n/\Delta$ vertices (that is, by Corollary~\ref{cor 5.18} at most half of the vertices of $C'_{cm, max}$ a.a.s.), ensuring concentration throughout the process. Thus, the conditions of Theorem~\ref{Thm:DEMethod} are satisfied for $|\mathcal C(t)|$, $|\mathcal V_1(t)|$ and $|\mathcal V_{\ge 2}(t)|$ as functions of $t$, and we deduce that all of them are concentrated around their expected values in the end of the exploration process.
\end{proof}

\begin{corollary}\label{cor 5.28}
A.a.s. for every small enough $\varepsilon' > 0$, the number of $S-$chains of length $\ell$ of vertices participating in only one $S-$chain of length $\ell$ is at least $\dfrac{\delta^2 \varepsilon'^2 n}{2M}\left(1+\dfrac{\gamma'}{\Delta M}\right)^{\ell-1}$.
\end{corollary}
\begin{proof}
By Lemma~\ref{lem 5.26} and Observation~\ref{ob: concentration} there is a constant $c'_{\ell}$ such that the number of vertices in $C'_{max}$ in at least two $S-$chains is a.a.s.\ at most $2c'_{\ell} \varepsilon'^3 n$. On the other hand, by Lemma~\ref{lem 5.24} and Observation~\ref{ob: concentration} there are a.a.s.\ at least $\dfrac{3}{4}\dfrac{\delta^2 \varepsilon'^2 n}{M}\left(1+\dfrac{\gamma'}{\Delta M}\right)^{\ell-1}$ $S-$chains of length $\ell$. We conclude that a.a.s.\ for every small enough $\varepsilon'$ there are at least
\begin{equation*}
    \dfrac{1}{2}\dfrac{\delta^2 \varepsilon'^2 n}{M}\left(1+\dfrac{\gamma'}{\Delta M}\right)^{\ell-1}
\end{equation*}
$S-$chains of length $\ell$, each of which contains only vertices, which participate in only one such chain. The lemma is proved.
\end{proof}

\begin{lemma}\label{main lemma supercritical}
For every $C > 0$, for every large enough $\ell$ and every small enough $\varepsilon' > 0$ (depending on $C$) a.a.s.
\begin{equation*}
\dfrac{e(A'_{max})}{|A'_{max}|} \ge 1 + C\dfrac{|A'_{max}|}{n}.
\end{equation*}
\end{lemma}
\begin{proof}
We fix a positive integer $\ell$ that we specify in the sequel. Fix also $C > 0$. By Corollary~\ref{cor 5.28} there are a.a.s.\ at least $\dfrac{\delta^2 \varepsilon'^2 n}{2M}\left(1 + \dfrac{\gamma'}{\Delta M}\right)^{\ell - 1}$ $S-$chains of length at most $\ell$ with vertices in only one chain of this type. On the other hand, $A'_{max}$ contains a.a.s.\ at most $\varepsilon'n + c_{\ell}\varepsilon'^2 n$ vertices by Lemma~\ref{lem 5.26} and Observation~\ref{ob: concentration} (recall that $A'_{max}$ contains the set $S$ together with all $S$--chains of length at most $\ell$). We conclude that for $\varepsilon' \to 0$, a.a.s. 
\begin{equation*}
    \dfrac{n(e(A'_{max}) - |A'_{max}|)}{|A'_{max}|^2} \ge \dfrac{\dfrac{\delta^2 \varepsilon'^2 n^2}{2M}\left(1 + \dfrac{\gamma'}{\Delta M}\right)^{\ell - 1}}{(\varepsilon' (1 + O_{\varepsilon'}(\varepsilon'))^2 n^2}.
\end{equation*}
Now, taking the limits, first, with respect to $n\to +\infty$, and then, with respect to $\varepsilon' \to 0$, we conclude that it is sufficient to choose $\ell$ such that $\dfrac{\delta^2}{2M} \left(1 + \dfrac{\gamma'}{\Delta M}\right)^{\ell-1} > C$. The lemma is proved.
\end{proof}

\begin{corollary}\label{cor critical}
The conclusion of Lemma~\ref{main lemma supercritical} holds a.a.s.\ also when replacing $C'_{cm, max}$ by $C'_{max}$.
\end{corollary}
\begin{proof}
This is a direct consequence of Corollary~\ref{cor 5.20}.
\end{proof}

A \textit{smoothing} of a vertex of degree 2 in a graph consists in deleting the vertex and then joining its two neighbors by an edge. In particular, we might obtain a multigraph: if the two neighbors are already connected by one or more edges, the number of edges between them increases by 1 after the smoothing. A \textit{contraction of an edge} in a graph consists in deleting the edge and identifying its endvertices. Remark that smoothing of a vertex $v$ of degree 2 is equivalent to contracting any of the edges incident with $v$. Let $G'$ be a uniform random graph on $(d'_i)_{2\le i\le \Delta}$ vertices of degree $(i)_{2\le i\le \Delta}$, respectively. 

\begin{lemma}\label{decycling lemma}
By sampling $G'$, smoothing all vertices of degree 2 and deleting isolated vertices incident to at most one loop, we generate a uniform random graph $G''$ on $(d'_i)_{3\le i\le \Delta}$ vertices of degree $(i)_{3\le i\le \Delta}$, respectively.
\end{lemma}
\begin{proof}
Let $(Q_{s,t,j})_{1\leq s\leq j, 1\leq t\leq d'_j, 2\le j\le \Delta}$ be the points of the matching at the origin of the configuration model for $G'$. Let also $(R_{s,t,j})_{1\leq s\leq i, 1\leq t\leq d'_j, 3\le j\le \Delta}$ be the points of the matching at the origin of the configuration model for the random graph $G''$. We present a coupling between the probability space of the matchings of $(Q_{s,t,j})_{1\leq s\leq j, 1\leq t\leq d'_j, 2\le j\le \Delta}$ and of $(R_{s,t,j})_{1\leq s\leq j, 1\leq t\leq d'_j, 3\le j\le \Delta}$. We perform the following algorithm generating the graphs $G'$ and $G''$ at the same time.
\begin{enumerate}
    \item Choose an arbitrary point $Q_{s', t', j'}$ with $1\leq s'\leq j, 1\leq t'\leq d'_j, 3\le j'\le \Delta$ (if it exists, if not, go to point 5.) that has not been matched yet. Prepare to match the point $R_{s', t', j'}$. 
    \item Match $Q_{s', t', j'}$ with some unmatched point $Q = Q_{s'', t'', j''}$ among $(Q_{s,t,j})_{1\leq s\leq j, 1\leq t\leq d'_j, 2\le j\le \Delta}$.
    \item If $j'' \ge 3$, match $R_{s', t', j'}$ and $R_{s'', t'', j''}$. Then, return to 1.
    \item If $j'' = 2$, then keep the point $R_{s', t', j'}$ waiting and perform 2. with $Q_{3 - s'', t'', j''}$ instead of $Q_{s', t', j'}$.
    \item Match all points among $(Q_{s,t,j})_{1\leq s\leq j, 1\leq t\leq d'_j, 2\le j\le \Delta}$ that remain unmatched uniformly at random.
\end{enumerate}
\end{proof}

\begin{lemma}\label{lem 5.31}
For every small enough $\varepsilon' > 0$ there exists a constant $c > 0$ such that the graph $C'_{cm, max}\setminus A'_{max}$ contains at most $c\varepsilon'^2 n$ connected components a.a.s.
\end{lemma}
\begin{proof}
By Lemma~\ref{connectivity}, $C'_{cm, max}$ contains a.a.s.\ a giant connected component which contains all but at most $\log n$ vertices (hence at most $\log n + 1$ connected components). Let $C''_{cm, max}$ be obtained by smoothing $C'_{cm, max}$ and deleting isolated vertices (that have loops attached to them) and let $A''_{max}$ be the subgraph of $C''_{cm, max}$ obtained from $A'_{max}$ by these operations. Now, perform the following procedure:
\begin{enumerate}
    \item Expose all edges going out of $A''_{max}$.
    \item For every connected component on the set of vertices in $C''_{cm, max}\setminus A''_{max}$, do the following: if this component is incident to only one unexposed half-edge, expose this half-edge. Repeat this procedure as long as such components on the set of currently exposed edges exist (but do not add them to $A''_{max}$; see Figure~\ref{A'_max} for an illustration before the smoothing).
    \item Delete the graph $A''_{max}$ and contract all connected components on the set of exposed edges. The graph that remains consists of only unexposed half-edges and all vertices are of degree different from 1.
\end{enumerate}

Here smoothing the graph $C'_{cm, max}$ is done to avoid the exploration of long paths of degree 2 and to a.a.s.\ ensure higher connectivity of the $2-$core. Another reason is that random graphs of minimal degree 3 are a.a.s.\ $3-$connected, which does not hold for random graphs of minimal degree 2.\par 

Consider the set of non-isolated vertices in $C''_{cm,max}$. Since the set of unexposed half-edges is matched according to the configuration model on a graph of minimal degree 2 and a positive proportion of vertices of degree at least 3 (this is ensured by Observation~\ref{ob 5.22} by choosing $\varepsilon'$ small enough), $C''_{cm, max}$ contains a.a.s.\ at most $\log n$ connected components, different from the giant component, by Lemma~\ref{connectivity}.\par

On the other hand, the set of isolated vertices in $C''_{cm, max}$ corresponds to connected components of exposed edges, which remain disconnected from the rest of the graph $C''_{cm, max}$ after deletion of $A''_{max}$.
Note that any such component of $C''_{cm,max} \setminus A''_{max}$ is a.a.s. connected by at least $3$ edges to $A''_{max}$ since $C''_{max}$ has minimum degree $3$ and by (\cite{FernholzRama}, Theorem 5.1) it is a.a.s.\ $3-$connected (note that our sequence is $2-$smooth in the terminology of~\cite{FernholzRama}, since $\lim_{n \to \infty} d_i(n)/n = p_i$ with $d_i(n) = 0$ for every $i\ge \Delta + 1$ and $n\in \mathbb N$ and this applies then also to the degree sequence of $C''_{max}$).\par

Define $C'_{del}$ to be the graph induced by $E(C'_{cm, max})\setminus E(A'_{max})$, and define $C''_{del}$ to be the graph induced by $E(C''_{cm, max})\setminus E(A''_{max})$. Observe that both $C'_{del}$ and $C''_{del}$ follow a configuration model. Set $C_{\Delta} = \dfrac{\Delta^{\frac{\ell}{2}+1} - 1}{\Delta - 1}$.
\begin{itemize}
    \item We saw that at most $\log n$ vertices of $C'_{cm,max}$ are not in the connected component of $A'_{max}$ a.a.s.
    \item The connected components in $C'_{del}$ are obtained from subdivisions of connected components in $C''_{del}$. If one such connected component contains a vertex of $C''_{cm, max}\setminus A''_{max}$, by $3-$connectivity of $C''_{cm, max}$ (see~\cite{FernholzRama}, Theorem 5.1) it connects to $A'_{max}$ using at least 3 edges. The remaining analysis is therefore similar to the analysis given in Subsection~\ref{subsec 5.1}. In particular, we have:
    \begin{itemize}
        \item the number of such components of order more than $\sqrt{\log n}$ is at most $o(n)$;
        \item as before, the number of components of order at most $\sqrt{\log n}$ containing a cycle is $o(n)$ by a first moment calculation similar to the one in Lemma~\ref{Number of copies} for trees;
        \item the expected number of acyclic components of order at most $\sqrt{\log n}$ is $C_1\varepsilon'^3 n$ for some constant $C_1 = C_1(\boldsymbol{p}, \Delta) > 0$. Indeed, all components of this type are connected by at least 3 edges to $A'_{max}$ and the number of open half-edges incident to $A'_{max}$ is at most $\Delta C_{\Delta}\varepsilon' n$. Moreover, the number of such components is concentrated around its expected value, which may be seen once again by a second moment method applied analogously to the proof of Lemma~\ref{Number of copies} in the subcritical case.
    \end{itemize}
On the other hand, if the connected component of $C'_{del}$ does not contain a vertex from $C''_{cm,max} \setminus A''_{max}$, we may obtain an $A'_{max}$-chain. Summing over all possible lengths of such $A'_{max}-$chains and applying the second moment method again analogously to the proof of Lemma~\ref{Number of copies} we obtain that there is a constant $C_2 = C_2(\boldsymbol{p}, \Delta) > 0$, such that the number of $A'_{max}-$chains is concentrated around its expected value, which is $C_2\varepsilon'^2 n$. Indeed, as above, all chains of this type are connected by exactly two edges to $A'_{max}$ and the number of open half-edges incident to $A'_{max}$ is at most $\Delta C_{\Delta}\varepsilon' n$. 
\end{itemize}
Thus, for every small enough $\varepsilon' > 0$ one may define $c = 2C_2$. This proves the lemma.
\end{proof}

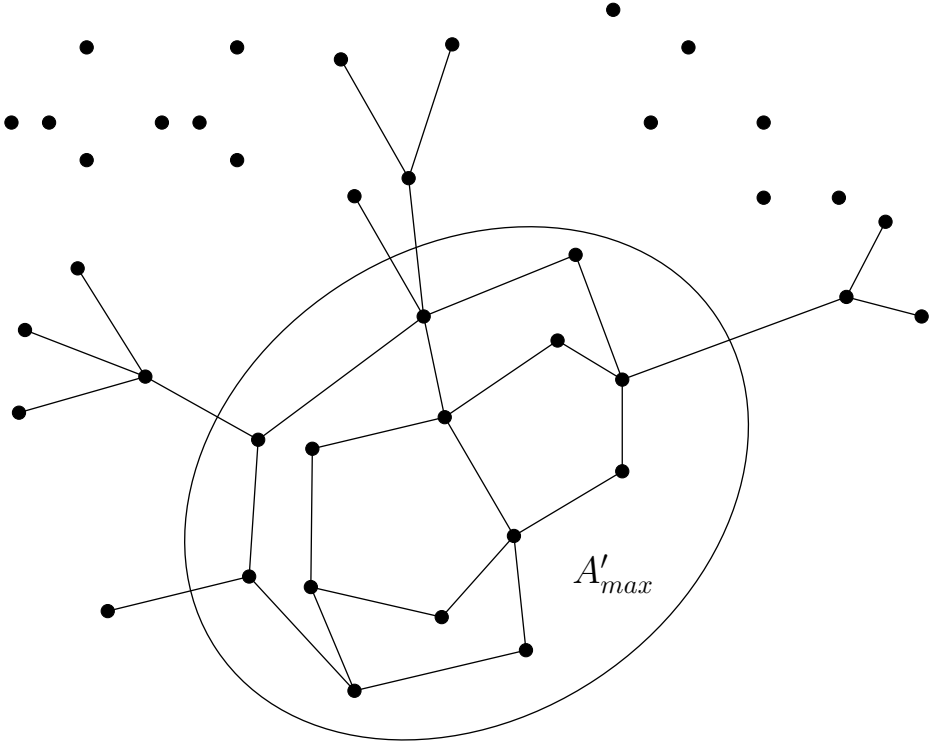
\begin{figure}
\centering
\begin{tikzpicture}[line cap=round,line join=round,x=1cm,y=1cm]
\clip(-8.5,-6) rectangle (8,4.6);
\draw [rotate around={33.69006752597976:(0.05,-1.8)},line width=0.5pt] (0.05,-1.8) ellipse (3.992680696373897cm and 3.122851124084646cm);
\draw [line width=0.5pt] (-2.02,-3.18)-- (-0.28,-3.58);
\draw [line width=0.5pt] (-0.28,-3.58)-- (0.68,-2.5);
\draw [line width=0.5pt] (0.68,-2.5)-- (-0.24,-0.92);
\draw [line width=0.5pt] (-0.24,-0.92)-- (-2,-1.34);
\draw [line width=0.5pt] (-2,-1.34)-- (-2.02,-3.18);
\draw [line width=0.5pt] (-2.02,-3.18)-- (-1.44,-4.56);
\draw [line width=0.5pt] (-1.44,-4.56)-- (0.84,-4.02);
\draw [line width=0.5pt] (0.84,-4.02)-- (0.68,-2.5);
\draw [line width=0.5pt] (0.68,-2.5)-- (2.12,-1.64);
\draw [line width=0.5pt] (2.12,-1.64)-- (2.12,-0.42);
\draw [line width=0.5pt] (2.12,-0.42)-- (1.26,0.1);
\draw [line width=0.5pt] (1.26,0.1)-- (-0.24,-0.92);
\draw [line width=0.5pt] (-0.24,-0.92)-- (-0.52,0.42);
\draw [line width=0.5pt] (-0.52,0.42)-- (1.5,1.24);
\draw [line width=0.5pt] (1.5,1.24)-- (2.12,-0.42);
\draw [line width=0.5pt] (-0.52,0.42)-- (-2.72,-1.22);
\draw [line width=0.5pt] (-2.72,-1.22)-- (-2.84,-3.04);
\draw [line width=0.5pt] (-2.84,-3.04)-- (-1.44,-4.56);
\draw [line width=0.5pt] (-0.52,0.42)-- (-1.44,2.02);
\draw [line width=0.5pt] (-0.52,0.42)-- (-0.72,2.26);
\draw [line width=0.5pt] (-0.72,2.26)-- (-1.62,3.84);
\draw [line width=0.5pt] (-0.72,2.26)-- (-0.14,4.04);
\draw [line width=0.5pt] (-2.72,-1.22)-- (-4.22,-0.38);
\draw [line width=0.5pt] (-4.22,-0.38)-- (-5.12,1.06);
\draw [line width=0.5pt] (-4.22,-0.38)-- (-5.82,0.24);
\draw [line width=0.5pt] (-4.22,-0.38)-- (-5.9,-0.86);
\draw [line width=0.5pt] (-2.84,-3.04)-- (-4.72,-3.5);
\draw [line width=0.5pt] (2.12,-0.42)-- (5.1,0.68);
\draw [line width=0.5pt] (5.1,0.68)-- (5.62,1.68);
\draw [line width=0.5pt] (5.1,0.68)-- (6.1,0.42);
\begin{scriptsize}
\draw [fill=black] (-2.02,-3.18) circle (2.5pt);
\draw [fill=black] (2.12,-0.42) circle (2.5pt);
\draw [fill=black] (-0.28,-3.58) circle (2.5pt);
\draw [fill=black] (0.68,-2.5) circle (2.5pt);
\draw [fill=black] (-0.24,-0.92) circle (2.5pt);
\draw [fill=black] (-2,-1.34) circle (2.5pt);
\draw [fill=black] (-1.44,-4.56) circle (2.5pt);
\draw [fill=black] (0.84,-4.02) circle (2.5pt);
\draw [fill=black] (2.12,-1.64) circle (2.5pt);
\draw [fill=black] (1.26,0.1) circle (2.5pt);
\draw [fill=black] (-0.52,0.42) circle (2.5pt);
\draw [fill=black] (1.5,1.24) circle (2.5pt);
\draw [fill=black] (-2.72,-1.22) circle (2.5pt);
\draw [fill=black] (-2.84,-3.04) circle (2.5pt);
\draw [fill=black] (-1.44,2.02) circle (2.5pt);
\draw [fill=black] (-0.72,2.26) circle (2.5pt);
\draw [fill=black] (-1.62,3.84) circle (2.5pt);
\draw [fill=black] (-0.14,4.04) circle (2.5pt);
\draw [fill=black] (-4.22,-0.38) circle (2.5pt);
\draw [fill=black] (-5.12,1.06) circle (2.5pt);
\draw [fill=black] (-5.82,0.24) circle (2.5pt);
\draw [fill=black] (-5.9,-0.86) circle (2.5pt);
\draw [fill=black] (-4.72,-3.5) circle (2.5pt);
\draw [fill=black] (5.1,0.68) circle (2.5pt);
\draw [fill=black] (5.62,1.68) circle (2.5pt);
\draw [fill=black] (6.1,0.42) circle (2.5pt);
\draw [fill=black] (5,2) circle (2.5pt);
\draw [fill=black] (4,2) circle (2.5pt);
\draw [fill=black] (4,3) circle (2.5pt);
\draw [fill=black] (3,4) circle (2.5pt);
\draw [fill=black] (3,5) circle (2.5pt);
\draw [fill=black] (2,4.5) circle (2.5pt);
\draw [fill=black] (2.5,3) circle (2.5pt);
\draw [fill=black] (-2,5) circle (2.5pt);
\draw [fill=black] (-3,4) circle (2.5pt);
\draw [fill=black] (-3.5,3) circle (2.5pt);
\draw [fill=black] (-3,2.5) circle (2.5pt);
\draw [fill=black] (-4,3) circle (2.5pt);
\draw [fill=black] (-2-3,4) circle (2.5pt);
\draw [fill=black] (-3-2.5,3) circle (2.5pt);
\draw [fill=black] (-3-2,2.5) circle (2.5pt);
\draw [fill=black] (-2-4,3) circle (2.5pt);
\draw[color=black] (2,-3) node {\Large $A'_{max}$};
\end{scriptsize}
\end{tikzpicture}
\caption{The component $A'_{max}$, the set of explored vertices and edges in $C'_{max}$ and a number of vertices, non-explored throughout the construction of $A'_{max}$ (these are shown as isolated in the figure).}
\label{A'_max}
\end{figure}

\begin{corollary}\label{cor 5.30}
The conclusion of Lemma~\ref{lem 5.31} holds for $C'_{max}$ instead of $C'_{cm, max}$ as well. 
\end{corollary}
\begin{proof}
This follows directly from Corollary~\ref{cor 5.20}.
\end{proof}

\begin{proof}[Proof of Lemma~\ref{new main lemma}]
If $p_1 = 0$, then $C_{max} = C'_{max}$ and the claim is a direct consequence of Lemma~\ref{main lemma supercritical}.\par 
By Lemma~\ref{main lemma supercritical} we have that for every positive constant $C$ there exists $\varepsilon_0 > 0$, such that, for every $\varepsilon' \in (0,\varepsilon_0)$, the set of vertices $A'_{max}$ satisfies a.a.s.
\begin{equation*}
 \dfrac{e(A'_{max})}{|A'_{max}|}\ge \left(1 + \dfrac{C|A'_{max}|}{n}\right).
\end{equation*}
Moreover, by Lemma~\ref{cmp A and A'}, there exists $\alpha > 1$, such that a.a.s.\ $A_{max}$ has size at most $\alpha |A'_{max}|$. Thus, a.a.s. the density of $A_{max}$ satisfies
\begin{equation*}
\dfrac{e(A_{max})}{|A_{max}|} \ge \dfrac{e(A'_{max})+(\alpha-1)|A'_{max}|}{\alpha |A'_{max}|}\ge \dfrac{\alpha+C\dfrac{|A'_{max}|}{n}}{\alpha}\ge 1+\dfrac{C|A_{max}|}{\alpha^2 n}.
\end{equation*}
Now, up to the choice of small enough $\varepsilon' > 0$, one may choose $C$ arbitrarily large. The conclusion of the first part of Lemma~\ref{new main lemma} is satisfied. The conclusion of the second part of Lemma~\ref{new main lemma} follows directly by Lemma~\ref{lem 5.31}.
\end{proof}

\section{Discussion}\label{sec:Discussion}
On the one hand, we showed that in sparse random graphs coming from a supercritical configuration model the modularity is strictly larger than the trivial lower bound obtained by partitioning into connected components. On the other hand, for random $3-$regular graphs we are more precise by quantifying the gain over this trivial bound (which in this case is equal to $2/3$). Without doubt, by adding more stages (that is, by considering longer chains) in the analysis of the lower bound in the case of random $3-$regular graphs, one could improve the lower bound given by Theorem~\ref{LB3reg} by a little bit. The upper bound still being far, we opted for not pushing this to the limit. Needless to say, it would be interesting to find the exact value of the modularity of random $3-$regular graphs, but our methods are not strong enough to determine this value. Another point for further thought is to find an asymptotic expression for $q^*(G(n))$ when $Q = 0$ in Theorem~\ref{LBgeneral} - it seems to us that, as it is often the case, this critical regime may be the most complicated to describe quantitatively.

\section{Acknowledgements}
The authors would like to thank the anonymous referee for useful suggestions and remarks.

\bibliographystyle{plain}
\bibliography{References}

\end{document}